\documentclass[pdflatex,sn-mathphys-num]{sn-jnl}

\usepackage{graphicx}%
\usepackage{multirow}%
\usepackage{amsmath,amssymb,amsfonts}%
\usepackage{amsthm}%
\usepackage{mathrsfs}%
\usepackage[title]{appendix}%
\usepackage{xcolor}%
\usepackage{textcomp}%
\usepackage{manyfoot}%
\usepackage{stmaryrd, mathtools}%
\makeatletter
\newcommand{\xmapsfrom}[2][]{%
    \ext@arrow 0359{\mapsfromfill@}{#1}{#2}%
}
\def\mapsfromfill@{%
    \arrowfill@\leftarrow\relbar{\relbar\joinrel\mapstochar}
}
\makeatother
\usepackage{booktabs}%
\usepackage{algorithm}%
\usepackage{algorithmicx}%
\usepackage{algpseudocode}%
\usepackage{listings}%
\usepackage{geometry}%

\usepackage{subcaption} 
\usepackage{comment}%

\usepackage{array}
\usepackage{tabularx}

\usepackage{extarrows}

\newcommand{\RR}{\mathbb{R}}
\newcommand{\Rpos}{\mathbb{R}_+}
\newcommand{\Rnz}{\mathbb{R}^*}

\newcommand{\ZZ}{\mathbb{Z}}
\newcommand{\Zpos}{\mathbb{Z}_+}

\newcommand{\defset}[2]{ \left\{\ #1 \ \left\lvert\ #2 \right.\ \right\} }

\newcommand{\abs}[1]{\left\lvert #1 \right\rvert}
\newcommand{\absbig}[1]{ \left| #1 \right| }

\newcommand{\rmd}{\mathrm{d}}
\newcommand{\pt}{\partial}
\newcommand{\ol}{\overline}

\DeclareMathOperator{\wt}{W}
\DeclareMathOperator{\Tree}{\mathcal{T}}
\DeclareMathOperator{\TR}{\mathcal{TR}}
\DeclareMathOperator{\Lab}{\mathfrak{L}}

\DeclareMathOperator{\triv}{Triv}

\DeclareMathOperator{\Aut}{Aut}
\DeclareMathOperator{\diam}{diam}
\DeclareMathOperator{\Pt}{Pt}
\DeclareMathOperator{\dis}{dis}

\DeclareMathOperator{\rGH}{GH}

\newcommand{\pp}{\Xi}
\newcommand{\ppf}{\Xi_F}
\newcommand{\ppt}{\Xi_T}
\newcommand{\Ptt}{\mathfrak{P}}

\newcommand{\pttp}{\mathfrak{p}}

\newcommand{\ee}{\mathrm{e}}

\newcommand{\one}{\mathbf{1}}

\newcommand{\HH}{\mathcal{H}}

\newcommand{\GH}{\mathcal{GH}}
\newcommand{\Rf}{\mathfrak{R}}

\newcommand{\MM}{\mathcal{M}}
\newcommand{\Mhcmu}{\mathcal{M}hcmu}
\newcommand{\Mhcmugn}{\mathcal{M}hcmu_{g,n}(\vec{\alpha})}

\theoremstyle{thmstyleone}%
\newtheorem{theorem}{Theorem}[section]
\newtheorem{proposition}[theorem]{Proposition}%
\newtheorem{lemma}[theorem]{Lemma}%
\newtheorem{corollary}[theorem]{Corollary}%

\newtheorem{definition}[theorem]{Definition}%
\newtheorem{notation}[theorem]{Notation}%
\newtheorem{fact}[theorem]{Fact}%

\theoremstyle{thmstyletwo}%
\newtheorem{example}[theorem]{Example}%
\newtheorem{remark}[theorem]{Remark}%

\usepackage{tikz}
\usetikzlibrary{shapes.misc}
\usetikzlibrary{decorations.markings}
\usetikzlibrary{positioning}
\usetikzlibrary{patterns}
\usetikzlibrary{calc,arrows.meta}
\tikzset{
    blk/.style = {draw, circle, fill=black, text=white, font=\bfseries\boldmath},
    wht/.style = {draw, circle, fill=white, text=black},
    blkdot/.style = {draw, circle, fill=black, inner sep=0pt, minimum size=6pt},
    whtdot/.style = {draw, circle, fill=white, inner sep=0pt, minimum size=6pt, line width=0.8pt},
    mid_auto/.style = {midway, auto},
    crossmark/.style = {cross out, draw=black, fill=none, inner sep=0pt, minimum size=5pt, line width=1.8pt},
}

\usepackage{algorithm}
\usepackage{algpseudocode}
\usepackage{setspace}
\usepackage{caption}
     
\begin{document}

\title[Counting trees and geometric applications]
{Counting Weighted Bi-Colored Plane Trees and Their Geometric Applications}

\author[1]{\fnm{Sicheng} \sur{Lu}}\email{sclu@suda.edu.cn}
\equalcont{All authors contributed equally to this work.}

\author*[2]{\fnm{Yi} \sur{Song}}\email{sif4delta0@mail.ustc.edu.cn}
\equalcont{All authors contributed equally to this work.}

\affil[1]{
    \orgdiv{School of Mathematical Sciences}, 
    \orgname{Soochow University}, 
    \orgaddress{
        \street{Shizijie 1}, \city{Suzhou}, \postcode{215006}, \state{Jiangsu}, \country{P. R. China}
}}

\affil*[2]{
    \orgdiv{School of Mathematical Sciences}, 
    \orgname{University of Science and Technology of China}, 
    \orgaddress{
        \street{Jinzhai Road 96}, \city{Hefei}, \postcode{230026}, \state{Anhui}, \country{P. R. China}
}}

\abstract{ 
This work solves the enumeration problem for weighted bi-colored plane trees with prescribed numbers of black and white vertices, together with prescribed total edge weights at each vertex. An exact closed formula for a particular case is obtained, and a unified algorithmic method for the general case is provided. 
We then apply this result to two geometric problems. Firstly, we compute the strong Hurwitz number for a special class of branch datum between Riemann spheres with three branched points. This is done by counting the dessins d'enfants that faithfully record such branched covers, and a clear correspondence between dessins and weighted trees in such case. 
Secondly, we study the geometric moduli space for a special class of extremal K\"{a}hler metrics on Riemann sphere (HCMU spheres), with a single conical singularity. We classify and enumerate the connected components of the moduli space with respect to the Gromov–Hausdorff topology. 
This is based on an efficient representation of these metric surfaces, and a study of Gromov–Hausdorff limits of such surfaces. 
}



\maketitle

\tableofcontents

\section{Introduction}\label{sec1}

This paper considers two seemingly distant geometric problems about Riemann surfaces. The first is the enumeration of branched covers with prescribed branch data. The second concerns classifying connected components of moduli spaces for a special class of extremal K\"{a}hler metrics. Behind both objects lies a common combinatorial structure on surfaces. By studying and completely enumerating these special yet simple structures, we are able to answer both questions. Below, we briefly describe these two problems and our results.

\subsection{Hurwitz number}

Let $X,Y$ be two compact Riemann surfaces. A non-constant holomorphic map $ f : Y \to X $ is also called a \textbf{branched cover}. By choosing suitable local complex charts, $f$ is locally conjugate to the complex function $z \mapsto z ^{m}$ around the origin, with $m \geq 1$. The integer $m$ is called the \textbf{local degree}. 
When $m>1$, the origin of the source chart on $Y$ is a \textbf{critical point}, and its image on $X$ is called a \textbf{branched point}. 

Let $B:=\{x_1,\dots,x_n\} \subset X$ be the set of branched points and $C:=f^{-1}(B) \subset Y$. Let $d$ be the \textbf{degree} of $f$, and $g_X, g_Y$ be the genus of $X,Y$. Then the restriction of $f$ is a genuine topological covering of degree $d$ from $Y \setminus C$ to $X \setminus B$. 
For each $i=1,\dots n$, the local degree of $f$ at points in $f^{-1}(x_i)$ form a partition $\pi_i$ of $d$. 
\begin{notation}\label{nota:partition}
    A \textbf{partition} $\pi$ of $d\in\ZZ_+$, denoted by $\pi \vdash d$, is recorded as \textbf{power notation} $\pi := (1^{n_1} 2^{n_2} \dots d^{n_d})$, with $n_j \in \ZZ_{\geq0}$ and $\sum_{j=1}^d j \cdot n_j = d$. 
    This means the integer $j$ appears $n_j$ times. For a concrete partition, integers with zero power are omitted. 
    
    The \textbf{length} of $\pi$ is $|\pi|:=\sum_{j=1}^d n_j$. 
\end{notation}
The \textbf{branch datum} of the branched cover $f: Y \to X$ is set to be the array 
$$ \mathcal{D} := (g_Y,g_X,d,n ;\ \pi_1 ,\dots, \pi_n) \ . $$
The Riemann-Hurwitz Formula asserts that
\begin{equation}\label{eq:Riemann-Hurwitz}    
    \left( 2-2g_Y \right) - \left( |\pi_{1}| + \cdots + |\pi_ {n}| \right) = d \left( 2-2g_X - n \right) \ .
\end{equation}
The \textbf{Hurwitz Existence Problem} \cite{Hur91} asks the converse: given an array $\mathcal{D}$ already satisfying \eqref{eq:Riemann-Hurwitz}, can it be realized as the branch datum of some branched cover? 
An array $\mathcal{D}$ satisfying the Riemann-Hurwitz Formula will be called a \textbf{compatible} branch datum. Such problem is still open today, with a wide range of results. We refer to \cite{Petr20} as a friendly introduction to this topic. 

The \textbf{Hurwitz number} of a compatible branch datum counts the branched covers that solve Hurwitz Existence Problem. 
Fix a Riemann surface $X$ of genus $g_X$, and a set $B=\{x_1,\dots,x_n\}$ on $X$ as the branched points. 
Usually, there are two optional sets of rules. 
Two branched covers $ f_1 : Y_1 \to X,\ f_2 : Y_2 \to X $ are said to be \textbf{weakly equivalent} if there exists orientation-preserving homeomorphisms $h: Y_1 \to Y_2$ and $g : X \to X$ such that $f_1 \circ h = g \circ f_2$, and \textbf{strongly equivalent} if one can take $g = \mathrm{id}_{X}$. 
The \textbf{strong/weak Hurwitz number} of a compatible branch datum is the number of strong/weak equivalence classes of branched covers realizing it. 
There is a vast body of results in this field, obtained by diversiform methods. We refer the reader to the survey articles \cite{Land10, ACEH18} for further details.

In this paper we consider a certain class of compatible branch data with $g_X = g_Y = 0, n=3$ and one special partition. 
Our first result computes the strong Hurwitz number of these branch datum. 

\begin{theorem}\label{thm:main_H.Number}
    Let $p>q>0$ be two co-prime integers. Consider three partitions of $d := pq$
    \[ \pi_1:= (q^p), \quad \pi_2:=(p^q), \quad \pi_3:=\left( 1^{(p-1)(q-1)} (p+q-1)^1 \right) \ . \]
    Define
    \begin{itemize}
        \item $g(q) := \gcd(q-1,p)$, $g(p) := \gcd(p-1,q)$;
        \item $D:=\defset{d\in\Zpos}{d \textrm{ is a factor of } g(q) \textrm{ or } g(p)}$; 
    \end{itemize}
    Then the strong Hurwitz number of the compatible branch datum
    \[ \mathcal{D}_{p,q} := \left( 0, 0, pq, 3;\ \pi_1, \pi_2, \pi_3 \right) \]
    is given by 
    \begin{equation}\label{eq:SHN_Hpq}
        \HH(p,q) := \sum_{d \in D} \varphi(d) G(d) \ ,
    \end{equation}
    where $\varphi(-)$ is the Euler's totient function, and
    \[ G(d) = 
        \begin{cases}
            \dfrac{(p+q-2)!}{p!q!} & \text{ if } d=1 \vspace{5pt} \\ 
            \dfrac{1}{p+q-1} \dbinom{({p+q-1})/d}{{p}/{d}} & \text{ if } 1<d \mid g(q) \vspace{5pt} \\ 
            \dfrac{1}{p+q-1} \dbinom{({p+q-1})/{d}}{{q}/{d}} & \text{ if } 1<d \mid g(p) 
        \end{cases}
        \ .
    \]
\end{theorem}

This is actually a particular version of a more general result in Theorem \ref{thm:st_H_number}.
We allow $\pi_1, \pi_2$ be arbitrary two partitions of a given integer $d$. In the third partition $\pi_3$, all parts are 1 except for a single part greater than 1.

\medskip
Some Hurwitz numbers of similar compatible datum with a special $\pi_3$ is also known. 
Harary-Tutte \cite{HarTut64} determined the Hurwitz numbers for $g_Y=0$ and $\pi_3 = (d^1)$. Goupil-Schaeffer \cite{GpSch98} later extended this result to arbitrary genus $g_Y\geq0$. A combinatorial proof of Goupil-Schaeffer's result was then given by Chapuy-F\'eray-Fusy \cite{CFF13}.

\subsection{HCMU surfaces}

An HCMU surface is a Riemann surface endowed with a special type of extremal K\"{a}hler metric with finitely many conical singularities. The study of such geometric objects was originally initiated by E. Calabi \cite{Calabi82extrm} and X.-X. Chen \cite{Cxx99}, continued in works like \cite{WangZhu00, CqWyy11} and many others. 
Such metric is a critical point of some energy functional over all conical conformal metric on the Riemann surface. 

Let $P:=\{ p_1, \cdots, p_n\}$ be a finite set on the compact Riemann surface $X$, and $\vec{\alpha} := (\alpha_1,\cdots,\alpha_n) \in (\RR_{+} \setminus \{1\})^n$ be an \textbf{angle vector}. 
A conformal metric $\rho$ on $X$ is said to have a \textbf{conical singularity} at $p_i$ of cone angle $2\pi\alpha_i$, if there exists a local complex coordinate $z$ with $z(p)=0$ such that 
$ \rho = \ee^{2 u(z)} \abs{\rmd z}^2 $ on the punctured neighborhood of $0$, and that 
$u(z) - (\alpha_i-1) \ln \abs{z}$ is continuous at $0$ and smooth outside $0$. 
With this coordinate, the curvature function $K$ is locally given by
$ K = - \Delta u \cdot \ee ^{2 u} $.

We call $(X, \rho)$ an \textbf{HCMU surface} if the curvature function $K$ is non-constant and satisfying the following differential equation \cite{CCW05}: 
\begin{equation*}\label{eq:HCMU}
    \frac{\pt^2 K}{\pt z^2} - 2 \frac{\pt K}{\pt z} \frac{\pt u}{\pt z} =0 \ .
\end{equation*}
The abbreviation ``HCMU'' stands for ``the Hessian of the curvature of the metric is umbilical'' in differential geometry \cite{Cxx00}.  

Let $\vec{\alpha}$ be an angle vector. The geometric moduli space $\Mhcmugn$ is the set of all isometry classes of HCMU surfaces of genus $g$ with $n$ conical singularities, whose cone angles are prescribed by $\vec{\alpha}$. 
It is usually endowed with the Gromov-Hausdorff topology. In the special case on which we will focus, this moduli space can be directly parameterized. 

For most cases, such moduli space is not connected and may locally like an orbifold. So before deeper investigation on its geometric structure, we would like to count the connected components of $\Mhcmugn$ and tell the difference between distinct components. This problem would be very hard in general. Our second result presents an answer for the simplest case.

\begin{theorem} \label{thm:main_hcmu}
    Let $m>2$ be a prime number, and $g=0, n=1$. Then the number of connected components in the moduli space $\Mhcmu_{0,1}(m-1)$ of HCMU spheres with a single conical singularity of angle $2\pi(m-1)$ is given by
    \[ \sum_{\substack{p>q>0 \\ p+q=m}} \HH(p,q) \ .\]
    Here $\HH(p,q)$ is just the strong Hurwitz number \eqref{eq:SHN_Hpq} in Theorem \ref{thm:main_H.Number}.
\end{theorem}

Again, the actual result in Theorem \ref{thm:cnntd_comp} is stated for any integer $m>2$. The case with prime $m$ has the most direct expression. Such result can be regarded as a further classification of HCMU spheres with a unique conical singularity after Meng-Wei \cite{MyjWzq25}.

\subsection{Counting trees}

What do the enumeration of branched covers and the enumeration of connected components of moduli space for HCMU surfaces have in common? Based on the geometric and topological properties of each class of objects, both problems can ultimately be reduced to the enumeration of a certain type of combinatorial structure on the plane. This structure is called a weighted bi-colored plane tree, whose precise definition is given in Section \ref{sec:pre_tree}. Loosely speaking, it is a plane tree whose vertices are colored black and white, with additional weight assigned to the edges. The task is to count the isomorphism classes of such trees with prescribed numbers of black and white vertices, together with prescribed total weights at each vertex. 

We shall see that both Theorem \ref{thm:main_H.Number} and Theorem \ref{thm:main_hcmu} are deduced from the following enumeration of plane trees. 
\begin{theorem} \label{thm:main_tree}
    Let $ p > q > 0 $ be two co-prime integers. Then the number of weighted bi-colored plane trees with $p$ black vertices of weight $q$ and $q$ white vertices of weight $p$, is precisely given by \eqref{eq:SHN_Hpq}. 
\end{theorem}

For example, Figure \ref{fig:example_5} lists all desired weighted bi-colored plane trees with 5 vertices in Theorem \ref{thm:main_tree}. The left one is a tree with 1 black vertices of weight 4, and 4 white vertices of weight 1 (i.e., $(p,q)=(4,1)$); the right one is a tree with $(p,q)=(3,2)$. 
Therefore, the strong Hurwitz numbers of both compatible branch datum $\mathcal{D}_{4,1}$ and $\mathcal{D}_{3,2}$ in Theorem \ref{thm:main_H.Number} are $1$. And the moduli space of HCMU spheres with a single conical singularity of angle $8\pi$ in Theorem \ref{thm:main_hcmu} has $1+1=2$ connected components. 

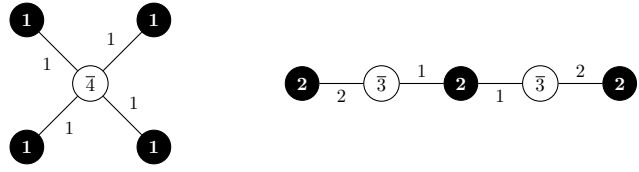
\begin{figure}
    \centering
\begin{tikzpicture}[scale=0.7, transform shape]
  \node[blk](1) at ( 1.2, 1.2) {$1$};
  \node[blk](2) at (-1.2, 1.2) {$1$};
  \node[blk](3) at (-1.2,-1.2) {$1$};
  \node[blk](4) at ( 1.2,-1.2) {$1$};
  \node[wht](0) at (0,0) {$\overline{4}$};
  \draw (0) -- (1) node[mid_auto] {$1$};
  \draw (0) -- (2) node[mid_auto] {$1$};
  \draw (0) -- (3) node[mid_auto] {$1$};
  \draw (0) -- (4) node[mid_auto] {$1$};

\begin{scope}[shift={(7,0)}]
  \node[blk](u1) at (-3, 0) {$2$};
  \node[blk](u2) at ( 0, 0) {$2$};
  \node[blk](u3) at (+3, 0) {$2$};
  \node[wht](d1) at (-3/2, 0) {$\overline{3}$};
  \node[wht](d2) at (+3/2, 0) {$\overline{3}$};
  \draw (d1) -- (u1) node[mid_auto] {$2$};
  \draw (d1) -- (u2) node[mid_auto] {$1$};
  \draw (d2) -- (u2) node[mid_auto] {$1$};
  \draw (d2) -- (u3) node[mid_auto] {$2$};
\end{scope}
\end{tikzpicture}
\vspace{10pt}
\caption{A list of weighted bi-colored plane trees with 5 vertices in Theorem \ref{thm:main_tree}.}
\label{fig:example_5}
\end{figure}

Actually, in Section \ref{sec:count.tree}, a more general tree enumeration problem is solved. Namely, for arbitrarily specified numbers of black and white vertices together with the total weights of edges connecting to each vertex, we provide a counting method for weighted bi-colored plane trees with such prescribed vertex data (Theorem \ref{thm:number_of_Tree_Xi}). This is based on a previous enumeration formula by Kochetkov \cite{Kyy13}. 
In the most general case, however, an explicit closed-form expression is difficult to obtain; instead, the count can be computed via a unified algorithm (see \nameref{alg:cap}). In the special co-prime case, such process yields the exact formula stated in \eqref{eq:SHN_Hpq}.

\medskip
The reduction of the aforementioned two geometric problems to tree enumeration requires some preparatory work.

In Section \ref{sec:count.covering}, we explain how a branched cover corresponds to an embedded graph on surfaces via the theory of dessins d'enfants, with the branch datum encoded in the black-and-white vertex information of the graph. 
Due to the particularity of the third partition, under the assumptions of Theorem \ref{thm:main_H.Number}, such graphs become weighted bi-colored plane trees. Consequently, the strong Hurwitz numbers can be obtained by counting these trees. This procedure is in fact well-known, so we only give a brief review and summary.

In Section \ref{sec:count.cpnt}, we explain how each HCMU sphere with a single conical singularity can be represented as a weighted bi-colored plane tree, together with some additional data. Furthermore, we endow the moduli space for isometry classes of such surfaces with Gromov–Hausdorff topology, and prove that under this topology, each connected component of the moduli space corresponds precisely to one weighted bi-colored plane tree appearing in the representation. Thus, by counting these trees, we obtain the number of connected components of the moduli space.

\medskip
\noindent \textbf{Acknowledgment} 

The authors would like to express their sincere thanks to Bin Xu for providing influential support on this work. Yi Song is supported by the National Natural Science Foundation of China (Grant No. 125B1019).

\bigskip
\section{Preliminary on plane trees}\label{sec:pre_tree}
In this section, we give precise definitions of the combinatorial objects involved and clarify the notation adopted here. Finally, we state an important enumeration formula for a class of trees, together with its variants. This formula serves as one of the foundational pillars of the present work.

\subsection{LWBP-trees}

A weighted bi-colored plane tree (\textbf{WBP-tree} in short) consists of 
    \begin{itemize}
        \item an embedded tree $(V,E)$ in $\RR^2$, with vertex set $V$ and edge set $E$; 
        \item a partition of $V$ into \textbf{black vertices} $V^+$ and \textbf{white vertices} $V^-$, such that each edge connects a black vertex to a white vertex;  
        \item a \textbf{weight function} $\wt_E : E \to \Rpos$ on edges.
    \end{itemize}
The weight function $W_E$ naturally extends to a function on $V$ by defining
\[ \wt_V(v^\pm):=\pm \sum_{e \in E(v^\pm)} \wt_E(e),\quad \forall v^\pm \in V^\pm \ . \]
Here $E(v)$ is the set of edges adjacent to $v$, and the sign is used to distinguish black and white vertices. 

\begin{definition}\label{def:passport}
    A \textbf{passport} $\pp = (S,\lambda,\wt)$ consists of 
    \begin{itemize}
        \item a nonempty finite set $S$ called the \textbf{index set}; 
        \item a \textbf{multiplicity function} $\lambda: S \to \Zpos$; and  
        \item a \textbf{weight function} $\wt : S\to \Rnz$ such that
        \begin{equation}\label{eq:passport_weight}
            \sum_{s\in S} \lambda(s) \cdot \wt(s) =0 \ .
        \end{equation}
    \end{itemize}
    Also define
    \begin{equation*}
        p:= \sum_{\wt(s)>0} \lambda(s), \quad q:= \sum_{\wt(s)<0} \lambda(s),\ \quad |\pp|:=p+q, \quad \Vert \pp \Vert:= \frac12 \sum_{s\in S} \lambda(s)\cdot\absbig{\wt(s)}
    \end{equation*}
    to be the number of black vertices, white vertices, total vertices and total weights respectively. 
\end{definition}

\begin{definition} \label{def:labeled_tree}
    A \textbf{labeled weighted bi-colored plane tree} (\textbf{LWBP-tree} in short) of passport $\pp = (S,\lambda,\wt)$ is a WBP-tree $(V,E,\wt_E)$ with a surjection $\Lab: V \to S$ called \textbf{labeling}, such that
    \begin{enumerate}
        \item as weight functions on $V$, $\wt \circ \Lab = \wt_V$; 
        \item as multiplicity functions on $S$, $|\Lab^{-1} (s)|= \lambda(s)$ for all $s\in S$. 
    \end{enumerate} 
    In this manuscript, all trees are LWBP-trees and we shall use the notation $T = (V, E, \wt_E, \Lab)$. 
\end{definition}
The labeling distinguishes vertices of the same weight. Such formalization is useful in enumeration problems. 

Two LWBP-trees $T_1, T_2$ of passport $\pp$ are said to be \textbf{isomorphic} if there exists an orientation-preserving homeomorphism $I:\RR^2 \to \RR^2$ such that $I(V^{\pm}_1) = V^{\pm}_2$ , $I(E_1) = E_2$ as elements of graphs, $\wt_{E_2} \circ\ I = \wt_{E_1}$ as weight functions, and $\Lab_2 \circ\ I = \Lab_1$ as labeling. 
Denote $\Tree(\pp)$ as the set of all isomorphic classes of LWBP-trees of passport $\pp$. 

\bigskip
Here is a practical method for describing an WBP-tree, using the ribbon graph structure. We refer to \cite[\S 1.3, \S 1.5]{LsZak04}, \cite[\S 1.1]{EmM13} for details. 

\begin{definition}\label{def:cyclic_action}
    Let $(V,E,\wt_E)$ be a WBP-tree. For each $v \in V$, $E(v)$ is the set of edges adjacent to $v$. The orientation of the plane induces a \textbf{cyclic action} $\sigma_v$ of order $\deg(v)$ on $E(v)$: $\forall e \in E(v)$, $\sigma_v(e)$ is the anticlockwise next edge in $E(v)$ around $v$. 
\end{definition}

\begin{proposition}\label{prop:eq_ribbon}
    An LWBP-tree is completely determined, up to isomorphisms, by an abstract tree $(V,E)$ with a family of cyclic actions $\sigma_v$ at each vertex $v$, together with the weight function $\wt_E$ and the labeling $\Lab$.
    
    More precisely, two LWBP-trees $T_1, T_2$ 
    are isomorphic if and only if there is a graph isomorphism $I:(V_1, E_1) \to (V_2, E_2)$ such that $\wt_{E_2} \circ\ I = \wt_{E_1}$, $\Lab_2 \circ\ I = \Lab_1$ and  
    \begin{equation}\label{eq:cyclic_compatible}
        I\big(\sigma_v (e) \big)= \sigma_{I(v)}\big( I(e) \big) 
    \end{equation}
    for any $v \in V_1, e \in E(v) \subset E_1$. 
\end{proposition}

\begin{definition}\label{def:types_passport} 
The following two types of passports are essential. 
\begin{enumerate}
    \item A passport is called \textbf{trivial}, 
    denoted by $\pp_T$, if $\wt : S\to \Rnz$ is injective. 
    Then the index set and multiplicity function are induced by the weight function itself. 
    \item A passport is called \textbf{full}, 
    denoted by $\ppf$, if $\lambda(s) \equiv 1$ for all $s\in S$. Such multiplicity function will be denoted as $\one$. 
    Hence for a LWBP-tree of a full passport, its labeling $\Lab:V\to S$ is a bijection. 
\end{enumerate}    
\end{definition}

Enumerations arising from geometry problems in this work all involve counting objects of trivial passport. But usually the enumeration of objects with full passports is easier and important intermediate steps.

\begin{notation}
    If the index set of a passport $\pp$ is $S = \{s_1, s_2, \ldots, s_N\}$, then we shall denote $\pp$ in the \textbf{power notation} 
    \[ 
    \pp = \left(\ {s_1}^{\lambda(s_1)} \ {s_2}^{\lambda(s_2)} \ \cdots \ {s_N}^{\lambda(s_N)} \ \right) \ . 
    \] 
    
    For a concrete passport, let $w\in \wt(S) \subset \Rnz$.  We shall present $\wt^{-1}(w) \subset S$ as a set of integers with distinct subscript $\{ w_1, w_2, \ldots, w_a \}$, where $a=\lvert \wt^{-1}(w) \rvert$. 
    Also denote negative weight $-w \in \ZZ_-$ as $\ol{w}$ for compactness. 
    For example, $\pp=(1_1^2\ 1_2\ \ol{3}_1)$ means 
    \begin{itemize}
        \item $S=\{1_1, 1_2, \ol{3}_1\}$; 
        \item $\lambda(1_1)=2$, $\lambda(1_2) = \lambda(\ol{3}_1)=1$; 
        \item $\wt(1_1)= \wt(1_2) = 1$, $\wt(\ol{3}_1)=-3$. 
    \end{itemize}
\end{notation}

\begin{example} \label{eg:tree_label}
    Figure \ref{fig:tree_label} exhibits some LWBP-trees of various kinds of passports. 
    \begin{enumerate}
        \item Figure \ref{subfig:tree_label_A} is a tree of trivial passport $(1^3\ \ol{3})$. 
        \item Figure \ref{subfig:tree_label_A*} is a tree of passport $(1_1^2\ 1_2\ \ol{3})$. 
        \item Figure \ref{subfig:tree_label_B}, \ref{subfig:tree_label_C},  \ref{subfig:tree_label_D} are trees of full passport $(1_1\ 1_2\ 1_3\ \ol{3})$.
        Trees in Figure \ref{subfig:tree_label_B} and \ref{subfig:tree_label_D} are isomorphic, differing by a $2\pi/3$-rotation about the central white vertex. 
        But trees in Figure \ref{subfig:tree_label_B} and \ref{subfig:tree_label_C} are not isomorphic. This can be checked by the cyclic action at the white vertex.     
        \hfill $\square$
    \end{enumerate}
\end{example}

\begin{figure}[t]
    \begin{subfigure}{0.3\textwidth}
        \centering
        \begin{tikzpicture}[scale=0.7, transform shape]
            \node[blk] (up) at (0, 2) {1};
            \node[wht] (mid) at (0, 0) {$\ol{3}$};
            \node[blk] (left) at (-1.5, -1) {1};
            \node[blk] (right) at (1.5, -1) {1};
            \draw (up) -- (mid) node[mid_auto] {1};
            \draw (left) -- (mid) node[mid_auto] {1};
            \draw (right) -- (mid) node[mid_auto] {1};
        \end{tikzpicture}
        \caption{}
        \label{subfig:tree_label_A}
    \end{subfigure}
    \begin{subfigure}{0.3\textwidth}
        \centering
        \begin{tikzpicture}[scale=0.7, transform shape]
            \node[blk] (up) at (0, 2) {\footnotesize $1_2$};
            \node[wht] (mid) at (0, 0) {$\ol{3}$};
            \node[blk] (left) at (-1.5, -1) {\footnotesize $1_1$};
            \node[blk] (right) at (1.5, -1) {\footnotesize $1_1$};
            \draw (up) -- (mid) node[mid_auto] {1};
            \draw (left) -- (mid) node[mid_auto] {1};
            \draw (right) -- (mid) node[mid_auto] {1};
        \end{tikzpicture}
        \caption{}
        \label{subfig:tree_label_A*}
    \end{subfigure}
    
    \vspace{0.3cm}
    
    \begin{subfigure}{0.3\textwidth}
        \centering
        \begin{tikzpicture}[scale=0.7, transform shape]
            \node[blk] (up) at (0, 2) {\footnotesize $1_1$};
            \node[wht] (mid) at (0, 0) {$\ol{3}$};
            \node[blk] (left) at (-1.5, -1) {\footnotesize $1_2$};
            \node[blk] (right) at (1.5, -1) {\footnotesize $1_3$};
            \draw (up) -- (mid) node[mid_auto] {1};
            \draw (left) -- (mid) node[mid_auto] {1};
            \draw (right) -- (mid) node[mid_auto] {1};
        \end{tikzpicture}
        \caption{}
        \label{subfig:tree_label_B}
    \end{subfigure}
    \begin{subfigure}{0.3\textwidth}
        \centering
        \begin{tikzpicture}[scale=0.7, transform shape]
            \node[blk] (up) at (0, 2) {\footnotesize $1_1$};
            \node[wht] (mid) at (0, 0) {$\ol{3}$};
            \node[blk] (left) at (-1.5, -1) {\footnotesize $1_3$};
            \node[blk] (right) at (1.5, -1) {\footnotesize $1_2$};
            \draw (up) -- (mid) node[mid_auto] {1};
            \draw (left) -- (mid) node[mid_auto] {1};
            \draw (right) -- (mid) node[mid_auto] {1};
        \end{tikzpicture}
        \caption{}
        \label{subfig:tree_label_C}
    \end{subfigure}
    \begin{subfigure}{0.3\textwidth}
        \centering
        \begin{tikzpicture}[scale=0.7, transform shape]
            \node[blk] (up) at (0, 2) {\footnotesize $1_2$};
            \node[wht] (mid) at (0, 0) {$\ol{3}$};
            \node[blk] (left) at (-1.5, -1) {\footnotesize $1_3$};
            \node[blk] (right) at (1.5, -1) {\footnotesize $1_1$};
            \draw (up) -- (mid) node[mid_auto] {1};
            \draw (left) -- (mid) node[mid_auto] {1};
            \draw (right) -- (mid) node[mid_auto] {1};
        \end{tikzpicture}
        \caption{}
        \label{subfig:tree_label_D}
    \end{subfigure}
    \caption{Some LWBP-trees of various kinds of passports, all with three vertices of weight $+1$ and one vertex of weight $-3$.}
    \label{fig:tree_label}
\end{figure}
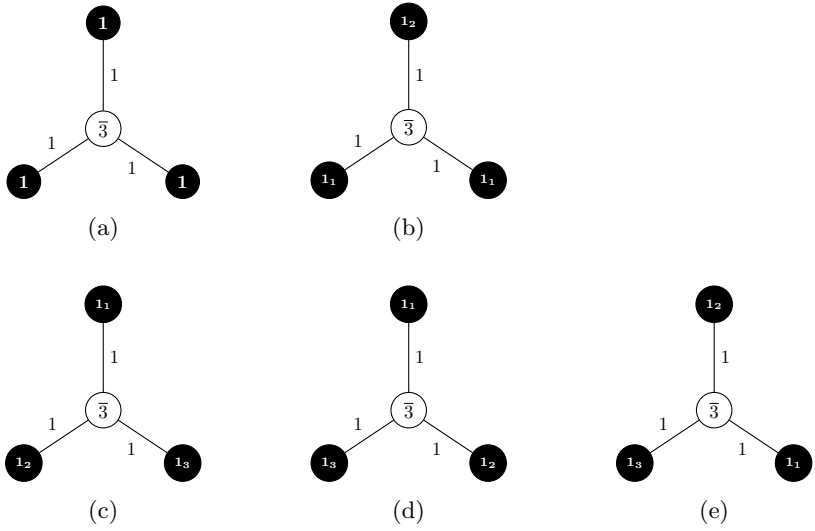

\subsection{Kochetkov's formula}

This subsection presents the enumeration formula for trees with a full passport. Our work builds upon this formula. In addition, we shall provide an modified version tailored to our problem. 

We begin with the necessary concepts, mainly subpassports and partitions.

\begin{definition}\label{def:factorial}
    The \textbf{factorial} of a passport $\pp = (S, \lambda, \wt)$ is defined as
    \begin{equation}
        (\pp)! := \prod_{s \in S} \lambda(s) ! \ . 
    \end{equation}
\end{definition}

\begin{definition}\label{def:subpp}
    Given a passport $\pp = (S, \lambda, \wt)$, let $S' \subset S$ be a non-empty subset and $\lambda' : S' \to \Zpos$ be a multiplicity function. If  
    \begin{itemize}
        \item $\lambda'(s) \leq \lambda(s)$ for all $s \in S'$,  
        \item $\wt'(s) = \wt(s)$ for all $s \in S'$, and 
        \item $\sum_{s \in S'} \lambda'(s) \wt'(s) = 0$,  
    \end{itemize}
    then $\pp(S', \lambda') := (S', \lambda', \wt')$ is a passport, called a \textbf{subpassport} of $\pp$. 
    The multiplicity function $\lambda'$ naturally extends by zero to a function $\lambda':S \to \ZZ_{\geq0}$. 
\end{definition}

\begin{definition}\label{def:partition}
    Denote an unordered sequence of subpassports with multiplicity (or a multi-set of subpassports) in power notation 
    $ \pttp := \{ \pp_1^{n_1}\ \pp_2^{n_2} \cdots \pp_k^{n_k} \} $, where each $\pp_i:=\pp(S_i, \lambda_i)$ is a subpassport. This means there are $n_i\geq0$ copies of $\pp_i$ in $\pttp$. 
    
    If the ``union'' of subpassports in $\pttp$ is actually $\pp$, i.e.
    \begin{equation*}
        \bigcup_{i=1}^k S_i = S \ , \quad \text{ and } \sum_{i=1}^k n_i \cdot \lambda_i(s) = \lambda(s) ,\ \forall s \in S \text{ , }
    \end{equation*}
    with $\lambda_i$ extended by zero, then the unordered sequence $\pttp$ is called a \textbf{partition} of $\pp$. 

    \begin{enumerate}
        \item The number of elements $\abs{\pttp} := \sum_{i=1}^k n_i$ is called the \textbf{length} of $\pttp$. 
            
        \item $\Ptt(\pp)$ is the set of all partitions of $\pp$. $\mathfrak{e}:=\{\pp\}$ is called \textbf{trivial partition}.
            
        \item Define the \textbf{X-function} of a partition as
        \begin{equation*}
            X(\pttp) := \prod_{i=1}^k \big[(\abs{\pp_i} - 1)! \big]^{n_i} \ .
        \end{equation*}

        \item Define the \textbf{factorial} of partition as
        \begin{equation*}
            (\pttp)! := \prod_{i=1}^k 
            \left( n_i ! \ [(\pp_i)!]^{n_i} \right) \ .
        \end{equation*}
        
        \item $\pp$ is called \textbf{non-decomposable} if its unique subpassport is $\pp$ itself. Otherwise, it is called \textbf{decomposable}.
    \end{enumerate}
\end{definition}

\begin{theorem}[Kochetkov's formula, \cite{Kyy13}] \label{thm:passport_simple} 
    Let $\ppf$ be a full passport, then 
    \begin{equation}\label{eq:enumeration_passport_decomposable}
        \absbig{\Tree(\ppf)} = \sum_{\pttp \in \Ptt(\ppf)} (-1)^{|\pttp|-1} (\absbig{\ppf}-1)^{|\pttp|-2} X(\pttp) \ .
    \end{equation}
    In particular, when $\ppf$ is non-decomposable, 
    \begin{equation} \label{eq:enumeration_passport_nondecomposable}
        \absbig{\Tree(\ppf)} = (\absbig{\ppf} - 2)! \ .
    \end{equation}   
\end{theorem} 
Besides an algebraic proof in the original paper, the authors of the present work also provide an alternative constructive proof in \cite{LuSong26}. 

\bigskip

To facilitate the computation, we refine the above formula in two respects: 
\begin{itemize}
    \item In practice, enumerating partitions of a full passport is rather cumbersome. To address this, we replace the summation on the right-hand side by a summation over partitions of the corresponding trivial passport. 
    \item The counting result for a general passport does not follow straightforwardly from Kochetkov's formula \eqref{eq:enumeration_passport_decomposable}. As an intermediate step, we introduce the enumeration of rooted trees.
\end{itemize}

The trivial passport naturally corresponding to a general passport is defined as below. 
\begin{definition} \label{def:trivialization}
    The \textbf{trivialization} of a passport $\pp = (S, \lambda, \wt)$ is a trivial passport $\triv(\pp) = (S_T, \lambda_T, \wt_T)$ defined by
    \begin{equation}\begin{split}
        S_T &:= \wt(S) \subset \RR^* \ , \\
        \lambda_T(w) &:= \sum_{s \in \wt^{-1}(w)} \lambda(s) \ , \\
        \wt_T(w) &:= w \ .
    \end{split}\end{equation}
    In other words, the index set of $\triv(\pp)$ is the image set of weight function $W$. The new multiplicity of index $w \in W(S)$ is the sum of those having weight $w$ in $\pp$.  
\end{definition}

The intermediate space of rooted trees is the following. 
\begin{definition}\label{def:rooted_tree}
A \textbf{rooted tree} is an LWBP-tree with a distinguished edge. Denote 
\[ \TR(\pp) := \defset{\left( T;e \right)}{
    \begin{array}{c}
          e \textrm{ is an edge of } T \in \Tree(\pp) 
    \end{array} 
    } . 
\]
as the set of all (isomorphic classes of) rooted trees of passport $\pp$. 
Here two rooted trees are isomorphic if the homeomorphism between two planes also keeps the root edge.
\end{definition}

Here is our modified version of Theorem \ref{thm:passport_simple}, whose proof is given in Appendix \ref{sec:modified_YYK}. 

\begin{theorem} \label{thm:number_of_TR_trivial_passport}
    Let $\pp$ be a passport and $\ppt := \triv(\pp)$ be its trivialization. Then
    \begin{equation}\label{eq:modified_YYK_formula}
          \abs{\TR(\pp)} 
        = \frac{(\ppt)!}{(\pp)!} \sum_{\pttp \in \Ptt(\ppt)} 
          (-1)^{\abs{\pttp} - 1} (\abs{\ppt} - 1)^{\abs{\pttp} - 1} \frac{X(\pttp)}{(\pttp)!} \ .
    \end{equation}
\end{theorem}

The passport of trees in \eqref{eq:modified_YYK_formula} can be general ones, which better suits our needs; the RHS is also amenable to concrete numerical computation via programming, as well as to a complete listing of all objects.

\bigskip
\section{Counting general LWBP-tree}\label{sec:count.tree}

To count unrooted trees from the enumeration of rooted trees, one must be especially careful with those unrooted trees with symmetry.
In this section, we provide a method to enumerate all trees with specific symmetry. This is done by studying symmetric trees and the properties of their passports.

\subsection{Symmetric tree and divided passport}

The symmetry of plane trees is not very complicated. 
The following result can be found in \cite[\S 1.5.2]{LsZak04}. 

\begin{lemma} \label{lem:cyclic_auto}
    The automorphism group $\Aut(T)$ of an LWBP-tree $T$ is always isomorphic to a cyclic group $\ZZ_d := \ZZ / d \ZZ$ of some finite order $d\geq1$. 
    When $d>1$, one can always choose a rotationally symmetric tree in its isomorphism class, such that $\Aut(T)$ is generated by a rotation through an angle of $2\pi/d$ about some central vertex $v_0$.
    
    However, the automorphism group of a rooted tree is always trivial. 
\end{lemma}
\begin{proof}

By the topology of plane trees, any automorphisms of a tree $T$ is induced by a finite-order orientation-preserving homeomorphisms of the plane. 
Such homeomorphism is isotopic to a rotation of the plane around a point. 
By applying this isotopy, $T$ becomes a rotationally symmetric tree, with rotation angle $2\pi/d$. 
The center of rotation must be a vertex on tree $T$. Otherwise the graph will contain a closed path, not a tree.  

On the other hand, the root edge of a rooted tree must be preserved by its automorphisms. 
But any nontrivial automorphism of a bi-colored plane tree do not preserve any edge.
\end{proof}

\begin{definition} \label{def:cyclic_auto}
    A tree $T$ is called \textbf{d-fold symmetric} if $\Aut(T) \cong \ZZ_d$. 
    The set of $d$-fold symmetric trees of passport $\pp$ is denoted by $\Tree(\pp, d) \subset \Tree(\pp)$. 
\end{definition}

\begin{proposition}\label{prop:rooted_tree_and_nonrooted_tree}
    For a general passport $\pp$,
    \begin{equation} \label{eq:rooted_tree_enum}
        \absbig{\TR(\pp)} 
        = (\absbig{\pp} - 1) \sum_{d = 1}^{\infty} \frac{\absbig{\Tree(\pp, d)}}{d} \ .
    \end{equation}
\end{proposition} 

\begin{proof}
With Lemma \ref{lem:cyclic_auto}, assume $T\in \Tree(\pp, d)$ is invariant under the rotation through an angle of $2\pi/d$ about vertex $v_0$. Choose an arbitrary edge and consider its orbit under the rotation. By symmetry, taking any of the $d$ edges in this orbit as the root edge yields isomorphic rooted trees. Thus, there are effectively only $(|\pp| - 1) / d$ distinct choices for the root edge. 
Obviously, $\Tree(\pp) = \bigsqcup_{d\geq1} \Tree(\pp, d)$. 
So the total number of rooted trees is given by \eqref{eq:rooted_tree_enum}.
\end{proof}

The passport of a tree with $d$-fold symmetry exhibits notable  divisibility properties. 

\begin{definition} \label{def:passport_divide}
    Let $\pp=(S,\lambda,\wt)$ be a passport and $d>1$ be an integer. We say that $\pp$ \textbf{admits division by d} if there exists $s_0 \in S$ such that 
    \begin{itemize}
        \item $d\ \big| \left(\lambda(s_0) - 1\right)$; 
        \item for $s\neq s_0$, $d \mid \lambda(s)$ . 
    \end{itemize}
    Then define the \textbf{divided passport} $\pp / d = (S^{(d)}, \lambda^{(d)}, \wt^{(d)})$ as 
    \begin{equation*}\begin{split}
        &S^{(d)} :=
        \begin{cases}
            S \sqcup\{s_*\} \ &,\text{ if } \lambda(s_0)>1 \\
            (S\setminus\{s_0\}) \sqcup\{s_*\} \ &,\text{ if } \lambda(s_0) = 1
        \end{cases} \\
        &\begin{cases}
            \lambda^{(d)}(s_*) := 1 \ ,\\
            \lambda^{(d)}(s_0) := \dfrac{\lambda(s_0) - 1}{d} \ &,\text{ if } s_0 \in S^{(d)} \\
            \lambda^{(d)}(s) \ \, := \dfrac{\lambda(s)}{d} \ &,\text{ if } s \in S\setminus\{s_0\} \subseteq S^{(d)}
        \end{cases} \\
        &\begin{cases}
            \wt^{(d)}(s_*) := \dfrac{\wt(s_0)}{d} \ &, \\
            \wt^{(d)}(s) \ \, := \wt(s) \ &,\text{ if } s \in S^{(d)}\setminus\{s_*\}
        \end{cases} \\
    \end{split}\end{equation*}
    Here $s_*$ is a new element outside $S$. 

    Besides, by convention $\pp$ admits division by $1$ and that $\pp / 1 := \pp$.
\end{definition}

\begin{lemma} \label{lem:exist_Tree_Xi_d_then_exist_passport_Xi/d}
    If $\Tree(\pp, d) \neq \varnothing$, then $\pp$ admits division by $d$. 
\end{lemma}

\begin{proof}

According to Lemma \ref{lem:cyclic_auto}, assume that $T\in \Tree(\pp, d)$ is rotationally symmetric, and that $\Aut(T)$ is generated by a rotation through an angle of $2\pi/d$ about vertex $v_0$. Let $T / d$ be the quotient tree of $T$ by the action of $\Aut(T)$. As a graph, it is isomorphic to a subtree of $T$ (see Figure \ref{fig:def_i_d}). Let $v_*$ be the image of vertex $v_0$, and label it with the distinguished index $s_*$. 
Except for $v_0$, all other vertices of $T$ lie in orbits of size $d$ under the $\Aut(T)$-action, sharing the same color and weight. The weight at $v_0$ is divided by $d$ in the quotient. Then the passport of $T / d$ is exactly $\pp / d$. Thus $\pp$ admits division by $d$. 
\end{proof}

By the lemma, if a tree has non-trivial symmetry, then the passport of its quotient tree is exactly the divided passport. Next, we describe the further correspondence between symmetric trees and quotient trees.

We first consider the criterion for a passport to admit division by $d$. 

\begin{proposition}\label{prop:basic_of_divided_passport}
    Let $\pp=(S,\lambda,W)$ be a passport. For each $s\in S$, let 
    \begin{equation} \label{eq:def_gs}
        g(s) := \mathrm{gcd} \big(\{\lambda(s')\}_{s'\neq s} \ ,\ \lambda(s)-1\big) \ .
    \end{equation}
    And define
    \begin{equation} \label{eq:def_D}
        D := \defset{d \in \Zpos}{\exists\ s \in S \text{ s.t. } d\ |\ g(s)} \ .
    \end{equation}
    Then $\pp$ admits division by $d$ if and only if $d \in D$. 

    Besides, if $\pp$ admits division by $d$, then the choice of $s_0$ in Definition \ref{def:passport_divide} is unique. Hence $\pp / d$ is well-defined and unique whenever $\pp$ admits division by $d$. 
    \qed
\end{proposition}

\begin{remark}
    If $\forall\ s \in S, \wt(s) \in \Zpos$, then in the divided passport $\pp / d = (S^{(d)}, \lambda^{(d)}, \wt^{(d)})$, $\wt^{(d)} (r) \in \Zpos$ for any $r \in S^{(d)}$. 
    That is to say, for integral passport $\pp$, its divided passport $\pp / d$ is always integral. 
    \qed
\end{remark}

The main result of this subsection is the following bijection between symmetric trees and quotient trees. This allows us to apply Proposition \ref{prop:rooted_tree_and_nonrooted_tree} on divided passports.

Such theorem reveals why the special label $s_*$ is introduced in divided passport. 
Without $s_*$, for example, the two trees on the far left of Figure \ref{fig:def_i_d} are isomorphic, and the bijection will not hold. 
Moreover, the necessity of handling divided passports is precisely what motivates our definition of general passports.

\begin{theorem} \label{thm:passport_divide_rotation}
    Let $\pp$ be a passport that admits division by $d\in\Zpos$, and $k \in \Zpos$. 
    Then there is a bijection from $\Tree(\pp, kd)$ to $\Tree(\pp / d, k)$. 
    Therefore,   
    \begin{equation} \label{eq:passport_divide_rotation}
        \absbig{\Tree(\pp, kd)} = \absbig{\Tree(\pp / d, k)} \ .
    \end{equation}
\end{theorem}

\begin{proof}
Let $T \in \Tree(\pp, kd)$ such that $\Aut(T) \cong \ZZ_{kd}$ is generated by a rotation through ${2\pi}/{kd}$ about a vertex $v_0$. Then the rotation by angle ${2\pi}/{d}$ generates a normal subgroup $A \triangleleft \Aut(T)$, isomorphic to $\ZZ_d$. Let $T / d$ be the quotient tree of $T$ by the action of $A$, and $v_*$ be the image of $v_0$. Then the quotient group $\Aut(T)/A \cong \ZZ_k$ also acts isomorphically on $T / d$. Hence $T / d \in \Tree(\pp / d, k)$. 

On the other hand, consider a rotationally symmetric
tree $T' \in \Tree(\pp / d, k)$. Let $v_* = \Lab^{-1}(s_*)$ be the unique vertex of $T'$ labeled as $s_*$. 
If $d>1$, any automorphism of $T'$ must fix $v_*$. Hence $v_*$ must be the center of rotation. 
Now glue $d$ copies of $T'$ cyclically along the center vertex $v_*$ to form a multiplication tree $T \cdot d$ (see Figure \ref{fig:def_i_d}). 
It is easy to verify that $T \cdot d$ has passport $\pp$ and is $dk$-fold symmetric. 
Hence $T \cdot d \in \Tree(\pp, dk)$. 

In the construction, the rotation center $v_0$ of $T \in \Tree(\pp, kd)$ and the special vertex $v_* = \Lab^{-1}(s_*)$ of $T' \in \Tree(\pp / d, k)$ correspond to each other. 
Therefore, the quotient and multiplication operations above are inverse to each other. 
Thus
\begin{equation*}
\begin{split}
    i : \Tree(\pp, kd) \quad &\xrightarrow{\qquad\ } \quad\Tree(\pp / d, k) \\
    T \quad\quad\ \  &\xmapsto{\qquad\ } \quad\quad T\ /\ d \\
    T' \cdot d \quad\quad &\xmapsfrom{\qquad} \quad\quad\quad T'
\end{split}
\end{equation*} 
is a bijection. 
\end{proof}

\begin{example}

Figure \ref{fig:def_i_d} gives an example of quotient and multiplication trees. 
Here $\pp = (1^2\ 3^2\ \bar{2}^2\ \bar{4})$ and $d = 2, k = 1$.
In each row, the right-most tree is rotationally $2$-fold symmetric, and the left-most one is the quotient tree. From left to right, one glues 2 copies of quotient tree along the vertices labeled as $\ol{2}_*$, to obtain a $2$-fold symmetric multiplication tree.

Note that the vertices labeled as $\ol{2}_*$ in two quotient trees are different. Equivalently, they are distinct trees in $\Tree(\pp/2)$ for the divided passport $\pp/2=(1\ 3\ \ol{2}\ \ol{2}_*)$.
Therefore, they correspond different symmetric trees through the bijection. 
Such distinctions are lost when considering the trivial passport $\triv(\pp  / 2)=(1\ 3\ \bar{2}^2)$.
\end{example}

\begin{figure}[htbp]
    \vspace{-3pt}
    \centering
    \begin{subfigure}{\textwidth}
        \centering
        \begin{tikzpicture}
            \begin{scope} [shift = {(-4, -1.5)}, scale=0.7, transform shape]
            \node[wht] (2D) at (0, 1) {\footnotesize $\ol{2}_*$};
            \node[blk] (1) at (1, 2) {1};
            \node[blk] (3) at (-1, 2) {3};
            \node[wht] (2U) at (-2, 3) {$\ol2$};
            \draw (2D) -- (1) node[mid_auto] {1};
            \draw (2D) -- (3) node[auto, pos = 0.7] {1};
            \draw (3) -- (2U) node[mid_auto] {2};
            \draw[->, thick, color = red] (2D) ++(60:0.7) arc[start angle=60, end angle=120, radius=0.7];
            \draw[->, thick, color = red] (2D) ++(150:0.7) arc[start angle=150, end angle=390, radius=0.7];
            \end{scope}
            
            \draw[-Stealth] (-2.2, 0) -- (-1.2, 0) node[mid_auto] {copy};

            \begin{scope} [shift = {(0, 0)}, scale=0.5]
                \node[wht] (2D) at (0, 1) {};
                \node (2D') at (0, 0.85) {*};
                \node[blk] (1) at (1, 2) {};
                \node[blk] (3) at (-1, 2) {};
                \node[wht] (2U) at (-2, 3) {};
                \draw (2D) -- (1); 
                \draw (2D) -- (3); 
                \draw (3) -- (2U); 
                \draw[->, thick, color = red] (2D) ++(60:0.7) arc[start angle=60, end angle=120, radius=0.7];
            \end{scope}

            \draw [<->, thick, color = red] (0, -0.2) -- (0, 0.2);

            \begin{scope} [shift = {(0, 0)}, scale=-0.5]
                \node[wht] (2D) at (0, 1) {};
                \node (2D') at (0, 1.1) {*};
                \node[blk] (1) at (1, 2) {};
                \node[blk] (3) at (-1, 2) {};
                \node[wht] (2U) at (-2, 3) {};
                \draw (2D) -- (1); 
                \draw (2D) -- (3); 
                \draw (3) -- (2U); 
                \draw[->, thick, color = red] (2D) ++(60:0.7) arc[start angle=60, end angle=120, radius=0.7];
            \end{scope}

            \draw [-Stealth] (1, 0) -- (2, 0) node[mid_auto] {glue};
            
            \begin{scope} [shift = {(4, -1)}, scale=0.7, transform shape]
                \node[wht] (M) at (0, 1) {$\ol{4}$};
                \node[blk] (1) at (1, 2) {1};
                \node[blk] (3) at (-1, 2) {3};
                \node[wht] (2U) at (-2, 3) {$\ol{2}$};
                \draw (M) -- (1) node[mid_auto] {1};
                \draw (M) -- (3) node[mid_auto] {1};
                \draw (3) -- (2U) node[mid_auto] {2};
                \node[blk] (1) at (-1, 0) {1};
                \node[blk] (3) at (1, 0) {3};
                \node[wht] (2U) at (2, -1) {2};
                \draw (M) -- (1) node[mid_auto] {1};
                \draw (M) -- (3) node[mid_auto] {1};
                \draw (3) -- (2U) node[mid_auto] {2};
                \draw[->, thick, color = red] (M) ++(60 :0.6) arc[start angle=60 , end angle=120, radius=0.6];
                \draw[->, thick, color = red] (M) ++(150:0.6) arc[start angle=150, end angle=210, radius=0.6];
                \draw[->, thick, color = red] (M) ++(240:0.6) arc[start angle=240, end angle=300, radius=0.6];
                \draw[->, thick, color = red] (M) ++(330:0.6) arc[start angle=330, end angle=390, radius=0.6];
            \end{scope}

            \draw [-Stealth] (4, -1.5) -- (4, -2.1) -- (-4, -2.1) -- (-4, -1.5); 
            \node (Div) at (0, -1.8) {quotient}; 
        \end{tikzpicture}
        \caption{Construction of a multiplication tree} 
        \label{subfig:def_i_d_A}
    \end{subfigure}
    \vspace{0.2cm}
    \centering
    \begin{subfigure}{\textwidth}
        \centering
        \begin{tikzpicture}
            \begin{scope} [shift = {(-4, -1.5)}, scale=0.7, transform shape] 
            \node[wht] (2D) at (0, 1) {$\ol{2}$};
            \node[blk] (1) at (1, 2) {1};
            \node[blk] (3) at (-1, 2) {3};
            \node[wht] (2U) at (-2, 3) {\footnotesize $\ol{2}_*$};
            \draw (2D) -- (1) node[mid_auto] {1};
            \draw (2D) -- (3) node[mid_auto] {1};
            \draw (3) -- (2U) node[auto, pos = 0.3] {2};
            \draw[->, thick, color = red] (2U) ++(-30:0.7) arc[start angle=-30, end angle=300, radius=0.7];
            \end{scope}

            \draw [-Stealth] (-2.6, 0.2) -- (-0.8, 0.2) node [above, pos=0.5] {multiply} ;

            \draw [Stealth-] (-2.6, -0.2) -- (-0.8, -0.2) node [below, pos=0.5] {quotient};

            \begin{scope} [shift = {(1.8, 0)}, scale=0.7, transform shape] 
            \node[wht] (M) at (0, 0) {$\ol{4}$};
                \node[blk] (1) at (-3, 1) {1};
                \node[blk] (3) at (-1, 1) {3};
                \node[wht] (2D) at (-2, 2) {$\ol{2}$};
                \draw (2D) -- (1) node[mid_auto] {1};
                \draw (2D) -- (3) node[mid_auto] {1};
                \draw (3) -- (M) node[mid_auto] {2};
                \node[blk] (1) at (3, -1) {1};
                \node[blk] (3) at (1, -1) {3};
                \node[wht] (2D) at (2, -2) {$\ol{2}$};
                \draw (2D) -- (1) node[mid_auto] {1};
                \draw (2D) -- (3) node[mid_auto] {1};
                \draw (3) -- (M) node[mid_auto] {2};
                \draw[->, thick, color = red] (M) ++(-30:0.6) arc[start angle=-30, end angle=120, radius=0.6];
                \draw[->, thick, color = red] (M) ++(150:0.6) arc[start angle=150, end angle=300, radius=0.6];
            \end{scope}
        \end{tikzpicture}
        \caption{Another example} 
        \label{subfig:def_i_d_B}
    \end{subfigure}
    \caption{Examples of quotient and multiplication tree. 
    }
    \label{fig:def_i_d}
\end{figure}
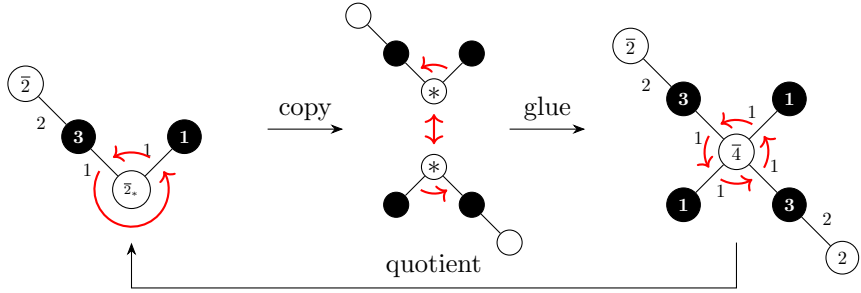
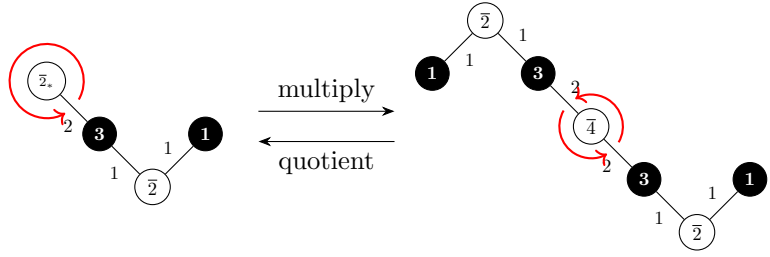

\subsection{Counting trees of trivial passport}

Now we are ready to present the main result of this section (Theorem \ref{thm:number_of_Tree_Xi}), also the cornerstone of the whole work.  

We begin with some functions from number theory, used in later theorem. 
Let $\varphi(n)$ be the \textbf{Euler's totient function}, which counts the number of positive integers $\leq n$ that are relatively prime to $n$.
$\mu(n)$ is \textbf{Möbius function}, defined to be $1$ if $n=1$, $(-1)^k$  if $n$ is a product of $k$ distinct primes, and $0$ otherwise.

The following two identities are used later.
\begin{fact}
For any positive integer $n \in \Zpos$, we have
\begin{equation}\label{eq:mobius_function}
    \sum_{\substack{m \,|\, n }} \mu(m) = \delta_{1,n} = 
    \begin{cases}
        1 ,& n = 1 \\
        0 ,& n \neq 1
    \end{cases}
    \ ,
\end{equation}
and 
\begin{equation}\label{eq:mobius_to_euler}
    \varphi(n) = \sum_{m \,|\, n} \mu\left(\frac{n}{m}\right) \cdot m \ .
\end{equation}
\qed
\end{fact}

\begin{definition} \label{def:G(d)}
    Fix a passport $\pp$ and let $D$ as in Proposition \ref{prop:basic_of_divided_passport}. Define 
    \begin{equation}\label{eq:def_G(d)}
        G(d) := \frac{\abs{\TR(\pp  / d)}}{\abs{\pp} - 1} 
    \end{equation}
    for each $d \in D$. With modified Kochetkov's formula (Theorem \ref{thm:number_of_TR_trivial_passport}), $G(d)$ is computable. 
\end{definition}

We enumerate $\Tree(\pp)$ by counting $d$-fold symmetric trees first. 

\begin{theorem} \label{thm:number_of_Tree_Xi_d}
    Let $\pp$ be a general passport and $d \in \Zpos$. Then
    \begin{equation} \label{eq:number_of_Tree_Xi_d}
        \absbig{\Tree(\pp, d)} = d \sum_{\substack{k\in D\\d\,|\,k}} \mu\left(\frac{k}{d}\right) G(k) \ .
    \end{equation}
\end{theorem}

\begin{proof}
If $d \notin D$, there is no $k\in D$ such that $d\ |\ k$, thus the RHS is zero. 
By Lemma \ref{lem:exist_Tree_Xi_d_then_exist_passport_Xi/d}, $d \notin D$ also implies $\Tree(\pp, d) = \varnothing$. 

Now assume $d \in D$. By Definition \ref{def:G(d)}, Proposition \ref{prop:rooted_tree_and_nonrooted_tree} and Theorem \ref{thm:passport_divide_rotation}, together with $\abs{\pp/d}-1 = \frac1d(\abs{\pp}-1)$, we see

\begin{equation}\label{eq:fg_dual_no_F}
\begin{split}
    G(d) := \frac{\abs{\TR(\pp / d)}}{\abs{\pp} - 1}
    &\xlongequal{\eqref{eq:rooted_tree_enum}} 
    \frac1{d(\abs{\pp/d}-1)}{\cdot (\abs{\pp / d} - 1)} \sum_{k \in \Zpos} \frac{\absbig{\Tree(\pp/d, k)}}{k} \\
    &\xlongequal{\eqref{eq:passport_divide_rotation}}
    \sum_{k \in \Zpos} \frac{\absbig{\Tree(\pp, dk)}}{dk} 
    = \sum_{\substack{l\in D \\ d\,|\,l}} \frac{\absbig{\Tree(\pp, l)}}{l} \ .\\
\end{split}\end{equation}

Each $d \in D$ gives a linear equation in the unknowns $\{\abs{\Tree(\pp,l)}\}_{l \in D}$ via \eqref{eq:fg_dual_no_F}. They form a system of $\abs{D}$ linear equations in $\abs{D}$ unknowns.
Since the coefficient matrix is upper triangular, the system has a unique solution. 

Solving such linear system will be similar to M\"{o}bius inversion.

\begin{align*}
    \sum_{\substack{k \in D \\ d\,|\,k}} \mu\left(\frac{k}{d}\right) G(k) 
    \xlongequal[]{\eqref{eq:fg_dual_no_F}}& \sum_{\substack{k \in D \\ d\,|\,k}} \mu\left(\frac{k}{d}\right) \sum_{\substack{l \in D \\ k\,|\,l}} \frac{\absbig{\Tree(\pp,l)}}{l} \\
    =& \sum_{\substack{l \in D \\ d\,|\,l}} \frac{\absbig{\Tree(\pp,l)}}{l} \sum_{\substack{k \in D \\ d\,|\,k\,|\,l}} \mu\left(\frac{k}{d}\right) \\
    =& \sum_{\substack{l \in D \\ d\,|\,l}} \frac{\absbig{\Tree(\pp,l)}}{l} \sum_{\substack{m \,|\, (l / d) }} \mu\left(m\right) 
    \xlongequal[]{\eqref{eq:mobius_function}} 
    \sum_{\substack{l \in D \\ d\,|\,l}} \frac{\absbig{\Tree(\pp,l)}}{l} \delta_{1,l/d} = \frac{\absbig{\Tree(\pp,d)}}{d} 
    \ .
\end{align*}
Then \eqref{eq:number_of_Tree_Xi_d} is obtained.
\qedhere
\end{proof}

\begin{theorem} \label{thm:number_of_Tree_Xi}
    The number of LWBP-trees of passport $\Xi$ is given by
    \begin{equation} \label{eq:number_of_Tree_Xi}
        \absbig{\Tree(\pp)} = \sum_{d\in D} \varphi(d) G(d) \ .
    \end{equation}    
\end{theorem}

\begin{proof}
By Theorem \ref{thm:number_of_Tree_Xi_d}, 
\begin{equation*}
    \abs{\Tree(\pp)} 
    = \sum_{d \in D} \abs{\Tree(\pp, d)} 
    = \sum_{d \in D} d \sum_{\substack{k\in D\\d\,|\,k}} \mu\left(\frac{k}{d}\right) G(k)
    = \sum_{k \in D} G(k) \sum_{\substack{d\in D\\d\,|\,k}} \mu\left(\frac{k}{d}\right) d
\end{equation*}
By the definition of $D$, if $k \in D$ and $d \mid k$, then $d \in D$. Therefore,
\begin{equation*}
    \sum_{\substack{d\in D\\d\,|\,k}} \mu\left(\frac{k}{d}\right) d = \sum_{d \,|\, k} \mu\left(\frac{k}{d}\right) d 
    \xlongequal[]{\eqref{eq:mobius_to_euler}} \varphi(k) \ .
\end{equation*}
And the desired formula is obtained.
\qedhere
\end{proof}

We summarize the complete procedure of enumerating $\Tree(\pp)$ in \nameref{alg:cap}. 

\begin{algorithm}[t]
\setstretch{1.2}
\renewcommand{\thealgorithm}{}  
\floatname{algorithm}{Main Algorithm}
\caption[Main Algorithm]{Enumerate LWBP-trees of a given passport}\label{alg:cap}
\begin{algorithmic}[1]
\Require Passport $\pp = (S, \lambda, W)$
\Ensure The number $\abs{\Tree(\pp)}$ of LWBP-trees of passprt $\pp$
\For{each $s\in S$}
    \State determine and record $g(s)$ as \eqref{eq:def_gs} in Proposition \ref{prop:basic_of_divided_passport} 
\EndFor
\State determine the set $D$ as \eqref{eq:def_D} in Proposition \ref{prop:basic_of_divided_passport} 
\For{each $d \in D$}
    \State compute the divided passport $\pp / d$ as Definition \ref{def:passport_divide}  
    \State set the trivialization $\triv(\pp/d)$ as Definition \ref{def:trivialization}
    \State list all the partitions $\Ptt(\triv(\pp/d))$ as Definition \ref{def:partition}
    \State compute $\abs{\TR(\pp / d)}$ by the modified Kochetkov's formula \eqref{eq:modified_YYK_formula} in Theorem \ref{thm:number_of_TR_trivial_passport} 
    \State compute and record $G(d)$ as Definition \ref{def:G(d)} 
    \State compute and record Euler's totient function $\varphi(d)$ 
\EndFor
\State compute $\abs{\Tree(\pp)}$ as Theorem \ref{thm:number_of_Tree_Xi}
\end{algorithmic}
\end{algorithm}

\begin{example}
We illustrate the \nameref{alg:cap} by taking $\pp = (2^2\ 4^3\ \ol{8}\phantom{}^2)$ as a trivial passport. 

\begin{itemize}
\item $g(2) = g(\ol{8}) = 1$, $g(4) = 2$. 
Thus $D=\{1,2\}$ and all divided passports are 
\[ \pp / 1 = \pp \quad \text{ and } \quad \pp / 2 = (2_*^1\ 2^1\ 4^1\ \ol{8}\phantom{}^1) \ . \] 

\item $\triv(\pp / 2) = (2^2\ 4^1\ \ol{8}\phantom{}^1)$ is non-decomposable. The only nontrivial partition of $\pp$ is 
\[ \pttp = \{(2^2\ 4\ \ol{8}), (4^2\ \ol{8}) \} \ . \] 

\item By modified Kochetkov's formula in Theorem \ref{thm:number_of_TR_trivial_passport}, 
\begin{equation*}
    G(1) 
    = \frac{\abs{\TR(\pp)}}{7 - 1} 
    = \frac16 \left( \frac{X(\mathfrak{e})}{(\mathfrak{e})!} - 6 \cdot \frac{X(\mathfrak{p})}{(\mathfrak{p})!} \right)
    = \frac16 \left( \frac{(7-1)!}{2!3!2!} - 6 \cdot \frac{(4-1)!(3-1)!}{(2!1!1!)(2!1!)} \right)
    = 2 \ ,
\end{equation*}
\begin{equation*}
    G(2) 
    = \frac{\abs{\TR(\pp / 2)}}{7 - 1} 
    = \frac16 \cdot \frac{[\triv(\pp / 2)]!}{(\pp / 2)!} \cdot \frac{X(\mathfrak{e})}{(\mathfrak{e})!}
    = \frac16 \cdot \frac{2!1!1!}{1!1!1!1!} \cdot \frac{(4-1)!}{2!1!1!}
    = 1 \ .
\end{equation*}

\item By Theorem \ref{thm:number_of_Tree_Xi}, 
\begin{equation*}
    \abs{\Tree(\pp)} = \varphi(1)G(1) + \varphi(2)G(2) = 2 + 1 = 3 \ .
\end{equation*}
By Theorem \ref{thm:number_of_Tree_Xi_d}, we can also get
\begin{equation*}
    \abs{\Tree(\pp, 1)} = \mu(1) G(1) + \mu(2) G(2) = 2-1 = 1 \ ,\ \abs{\Tree(\pp, 2)} = 2 \mu(1) G(2) = 2 \ .
\end{equation*}
\end{itemize}

All the trees in $\Tree(\Xi)$ are listed in Figure \ref{fig:tree_2^24^38^2}.
\end{example}

\begin{figure}[htbp]
    \centering
    \begin{subfigure}{0.3\textwidth}
        \centering
        \begin{tikzpicture}
            \node[blk] (M) at (0, 0) {};
            \begin{scope}
            \node[wht] (8) at (-1, 0) {};
            \node[blk] (2) at (-1.5, 0.87) {};
            \node[blk] (4) at (-1.5, -0.87) {};
            \draw (M) -- (8) node[mid_auto] {2};
            \draw (8) -- (2) node[mid_auto, swap] {2};
            \draw (8) -- (4) node[mid_auto, color = red] {4};
            \end{scope}
            \begin{scope}[rotate = 180]
            \node[wht] (8) at (-1, 0) {};
            \node[blk] (2) at (-1.5, 0.87) {};
            \node[blk] (4) at (-1.5, -0.87) {};
            \draw (M) -- (8) node[mid_auto] {2};
            \draw (8) -- (2) node[mid_auto, swap] {2};
            \draw (8) -- (4) node[mid_auto, color = red] {4};
            \end{scope}
        \end{tikzpicture}
        \caption{A tree in $\Tree(\pp, 2)$}
        \label{subfig:tree_2^24^38^2_A}
    \end{subfigure}
    \hfill
    \centering
    \begin{subfigure}{0.3\textwidth}
        \centering
        \begin{tikzpicture}
            \node[blk] (M) at (0, 0) {};
            \begin{scope}[yscale = -1]
            \node[wht] (8) at (-1, 0) {};
            \node[blk] (2) at (-1.5, 0.87) {};
            \node[blk] (4) at (-1.5, -0.87) {};
            \draw (M) -- (8) node[mid_auto] {2};
            \draw (8) -- (2) node[mid_auto] {2};
            \draw (8) -- (4) node[mid_auto, swap, color = red] {4};
            \end{scope}
            \begin{scope}[rotate = 180, yscale = -1]
            \node[wht] (8) at (-1, 0) {};
            \node[blk] (2) at (-1.5, 0.87) {};
            \node[blk] (4) at (-1.5, -0.87) {};
            \draw (M) -- (8) node[mid_auto] {2};
            \draw (8) -- (2) node[mid_auto] {2};
            \draw (8) -- (4) node[mid_auto, swap, color = red] {4};
            \end{scope}
        \end{tikzpicture}
        \caption{Another tree in $\Tree(\pp, 2)$}
        \label{subfig:tree_2^24^38^2_B}
    \end{subfigure}
    \hfill
    \centering
    \begin{subfigure}{0.3\textwidth}
        \centering
        \begin{tikzpicture}
            \node[blk] (M) at (0, 0) {};
            \begin{scope}
            \node[wht] (8) at (-1, 0) {};
            \node[blk] (2) at (-1.5, 0.87) {};
            \node[blk] (4) at (-1.5, -0.87) {};
            \draw (M) -- (8) node[mid_auto] {2};
            \draw (8) -- (2) node[mid_auto, swap] {2};
            \draw (8) -- (4) node[mid_auto, color = red] {4};
            \end{scope}
            \begin{scope}[xscale = -1]
            \node[wht] (8) at (-1, 0) {};
            \node[blk] (2) at (-1.5, 0.87) {};
            \node[blk] (4) at (-1.5, -0.87) {};
            \draw (M) -- (8) node[mid_auto] {2};
            \draw (8) -- (2) node[mid_auto] {2};
            \draw (8) -- (4) node[mid_auto, swap, color = red] {4};
            \end{scope}
        \end{tikzpicture}
        \caption{Unique tree in $\Tree(\pp, 1)$}
        \label{subfig:tree_2^24^38^2_C}
    \end{subfigure}
    \caption{The list of all trees in $\Tree(\pp)$, $\pp = (2^2\ 4^3\ \ol{8}\phantom{}^2)$. }
    \label{fig:tree_2^24^38^2}
\end{figure}

\subsection{Computation for a special case}\label{ssec:compute}

Now fix two integers $p,q \in \Zpos$ and define the trivial passport
\[ \pp_{p, q} := \left(q^p \ \ol{p}^q \right) \ . \]

In this subsection, we shall enumerate $\Tree(\pp_{p,q})$. Set 
\begin{equation}\label{eq:g0_pq}
    p = a g_0,\ q=b g_0, \text{ with } g_0:=\gcd(p,q) \ .
\end{equation}
Following Proposition \ref{prop:basic_of_divided_passport}, set 
\begin{equation}
    g(q) = \gcd(q-1, p) \ ,\quad 
    g(p) = \gcd(p-1, q) \ .
\end{equation}

We denote the set of all partitions of $n \in \Zpos$ as $\Ptt(n)$. 
By convention, we also define the \textbf{empty partition} $\xi_0 = (1^{n_1} 2^{n_2} \cdots )$ by setting $n_1 = n_2 = \cdots = 0$, and set $\Ptt(0) = \{\xi_0\}$. 

To enumerate $\Tree(\pp_{p,q})$ by Theorem \ref{thm:number_of_Tree_Xi}, we only need to compute $\abs{\TR(\pp_{p, q} / d)}$ for all factors $d$ of $g(q)$ or $g(p)$. 
By the modified Kochetkov's formula in Theorem \ref{thm:number_of_TR_trivial_passport}, it remains to find all partitions of $\triv(\pp_{p,q} /d)$ for each $d$. 

\bigskip

We begin with $d=1$. 

\begin{lemma} \label{lem:pp_pq_subpassport}
    Any subpassport of $\pp_{p, q}$ has the form 
    \begin{equation*}
        \pp(k) := (q^{k a}\ \ol{p}^{k b}) \ ,\quad \text{ with }\quad 1 \leq k \leq g_0 \ .
    \end{equation*}
\end{lemma}

\begin{proof}
Obviously, a subpassport of $\pp_{p, q}$ must have the form $\pp_1 = (q^x\ \ol{p}^y)$ with $x, y\in \Zpos$. 
$x,y$ only need to satisfy the constrains
\begin{equation*}
    q x - p y = 0 \ ,\quad 1 \leq x \leq p \ ,\quad 1 \leq y \leq q \ .
\end{equation*}
With \eqref{eq:g0_pq}, the solutions are 
\begin{equation*}
    x = ka \ ,\quad y = kb \ ,\quad \text{ where } \quad 1 \leq k \leq g_0 \ .
\end{equation*}
Thus all subpassports take the form $(q^{ka}\ \ol{p}^{kb})$. 
\end{proof}

\begin{lemma} \label{lem:pp_pq_partition}
    There is a one-to-one correspondence between partitions of passport $\pp_{p, q}$ and partitions of integer $g_0$
    \begin{equation*}\begin{split}
        \Ptt(\pp_{p, q}) \quad\qquad\qquad\qquad &\xrightarrow{\qquad\ } \quad\qquad\qquad \Ptt(g_0) \\
        \pttp = \left\{ \pp(1)^{n_1} \ \pp(2)^{n_2} \ \pp(3)^{n_3} \cdots \right\} \qquad &\xmapsto{\qquad\ }\qquad \xi = (1^{n_1}\ 2^{n_2}\ 3^{n_3} \cdots )
    \end{split}
    \quad .
    \end{equation*}
\end{lemma}

\begin{proof}

The partition of $\pp_{p, q}$ can only take the form $\pttp = \left\{ \pp(1)^{n_1} \ \pp(2)^{n_2} \ \pp(3)^{n_3} \cdots \right\}$. 
The total number of $q$ in $\pttp$ is $p$, i.e., $\sum_{k=1}^{g_0} (k a) \cdot n_k = p = a g_0$.
This yields $\sum_{k=1}^{g_0} k \cdot n_k = g_0$, thus $\xi = 1^{n_1}\ 2^{n_2}\ 3^{n_3} \cdots$ is a partition of $g_0$.
Considering the total number of $p$ in $\pttp$ yields the same result. 
\end{proof}

\begin{proposition}\label{prop:G(1)_qppq}
    For any $m\in\Zpos$ and its partition $\xi = ( 1^{n_1} \cdots m^{n_m}) \vdash m$, define
        \begin{equation}
            Y(\xi) := \prod_{k=1}^{\absbig{\xi}} {\frac{1}{n_k!}\left(\frac{(k (a+b)-1)!}{(k a)!(k b)!}\right)^{n_k}} \ .
        \end{equation}
    By convention, set $Y(\xi_0) = 1$ when $m=0$. 
    Then we have
        \begin{equation}
            G(1) = \sum_{\xi \vdash g_0}  (-1)^{\abs{\xi}-1} (p+q-1)^{\abs{\xi}-2} Y(\xi)
        \end{equation}
\end{proposition}

\begin{proof}

Given $\pttp \in \Ptt(\pp_{p, q})$, let $\xi \vdash g_0$ be the corresponding partition by Lemma \ref{lem:pp_pq_partition}.
It is easy to verify that $\abs{\xi} = \abs{\pttp}$ and $Y(\xi) = X(\pttp) / (\pttp)!$.
Thus, 
\begin{equation*}
    G(1)  
    = \frac1{\abs{\pp_{p, q}} - 1} \sum_{\pttp \in \Ptt(\pp_{p, q})} (-1)^{\abs{\pttp} - 1} (\abs{\pp_{p, q}} - 1)^{\abs{\pttp} - 1} \frac{X(\pttp)}{(\pttp)!} 
    = \sum_{\xi \vdash g_0} (-1)^{\abs{\xi} - 1} (p+q-1)^{\abs{\xi} - 2} Y(\xi) \ . 
    \tag*{\qedhere}
\end{equation*}

\end{proof}

\bigskip
Next assume $d>1$. 
We only consider the case $d \mid g(q)$. Arguments and results are similar when $d \mid g(q)$. 

When $d \mid g(q)=\gcd(q-1, p)$,
\begin{equation*}
    \pp_{p, q} / d = \left( \frac{q}{d} \quad q^\frac{p-1}{d} \quad \ol{p}^\frac{q}{d}  \right) 
\end{equation*}
and is automatically trivial.

\begin{lemma} \label{lem:pp_pq/d_subpassport}
    Assume $1 < d \mid g(q)$. 
    The subpassport of $\pp_{p, q} / d$ is either $\pp(k)$ defined in Lemma \ref{lem:pp_pq_subpassport} with $ 1 \leq k \leq \left\lfloor \frac{g_0}{d} \right\rfloor$, or
    \begin{equation*}
        \pp_d(l) := \left( \frac{q}{d} \quad q^{\frac{p-1}{d} - l a} \quad \ol{p}^{\frac{q}{d} - l b}  \right) 
        \ ,\quad \text{ with }\quad 0 \leq l \leq \left\lfloor \frac{g_0}{d} \right\rfloor
        \ .
    \end{equation*}
\end{lemma}

\begin{proof}
Let $\pp_1$ be a subpassport of $\pp$. There are two cases.

\medskip
If $\pp_1$ does not contain $\frac{q}{d}$, $\pp_1$ must take the form of $\pp(k)$ in Lemma \ref{lem:pp_pq_subpassport}. 
Then $kb\leq \frac qd$ implies $k\leq \frac{q}{db} = \frac{g_0}{d}$.

\medskip
If $\pp_1$ contains $\frac{q}{d}$, then set $\pp_1 = (\frac{q}{d}\ q^x\ \ol{p}^y)$, where $x, y \in \Zpos$ satisfying 
\begin{equation*}
    \frac{q}{d} + q x - p y = 0 \ ,\quad 1 \leq x \leq \frac{p-1}{d} \ ,\quad 1 \leq y \leq \frac{q}{d} \ .
\end{equation*}
The weight function of divided passport $\pp_{p, q} / d$ requires that 
$\frac{q}{d} + q \cdot \frac{p-1}{d} - p \cdot \frac{q}{d} = 0$. 
Subtracting the two equations yields
\begin{equation*}
    q \left(\frac{p-1}{d} - x\right) - p \left(\frac{q}{d} - y\right) = 0 \ .
\end{equation*}
Similar to Lemma \ref{lem:pp_pq_subpassport}, the solutions are 
\begin{equation*}
    x = \frac{p-1}{d} - l a \ ,\quad y = \frac{q}{d} - l b \ ,\quad 0 \leq l \leq \left\lfloor \frac{g_0}{d} \right\rfloor \ . 
    \tag*{\qedhere}
\end{equation*}

\end{proof}

\begin{lemma} \label{lem:pp_pq/d_partition}
    Assume $1 < d \mid g(q)$. 
    There is a one-to-one correspondence between partitions of passport $\pp_{p, q} / d$ and partitions of integers from $0$ to $\left\lfloor \frac{g_0}{d} \right\rfloor$. 
    \begin{equation*}\begin{split}
        \Ptt(\pp_{p, q} / d) 
        \qquad\qquad\qquad\qquad &\xrightarrow{\qquad\ } \quad\qquad 
        \bigsqcup_{0 \leq l \leq \left\lfloor \frac{g_0}{d} \right\rfloor} \Ptt(l) \\
        \pttp = \left\{\pp_d(l)\ \ \pp(1)^{n_1} \ \pp(2)^{n_2} \ \pp(3)^{n_3} \cdots \right\} 
        \qquad &\xmapsto{\qquad\ } \qquad 
        \xi = (1^{n_1}\ 2^{n_2}\ 3^{n_3} \cdots)
    \end{split}
    \qquad .
    \end{equation*}
\end{lemma}

\begin{proof}

For any $\pttp \in \Ptt(\pp_{p, q} / d)$, there is exactly one subpassport in $\pttp$ containing $q / d$. 
By Lemma \ref{lem:pp_pq/d_subpassport}, $\pttp$ must take the form $\left\{\pp_d(l)\ \ \pp(1)^{n_1} \ \pp(2)^{n_2} \ \pp(3)^{n_3} \cdots \right\}$ for some $0 \leq l \leq \lfloor \frac{g_0}{d} \rfloor$. 
The total number of $q$ in $\pttp$ is $\frac{p-1}{d}$. Then 
\begin{equation*}
    \left(\frac{p-1}{d} - l a\right) + \sum_{k \in \Zpos} (ka) \cdot n_k = \frac{p-1}{d} \ .
\end{equation*}
It follows that $\sum_{k \in \Zpos} k \cdot n_k = l$. 
Hence $\xi = (1^{n_1}\ 2^{n_2}\ 3^{n_3} \cdots)$ 
is a partition of $l$, with $0 \leq l \leq \left\lfloor \frac{g_0}{d} \right\rfloor$. 
The converse can be shown similarly. 

\end{proof}

\begin{proposition}\label{prop:G(d)_qppq}
    Assume $d > 1$, $d \mid g(q)$ or $d \mid g(p)$.
    For $d \mid g(q)$, 
    \begin{equation}
        G(d) =
        \sum_{\substack{0 \leq l \leq \left\lfloor \frac{g_0}{d} \right\rfloor \\ \xi \vdash l}} 
        (-1)^{\abs{\xi}} \frac{(p+q-1)^{\abs{\xi}-1}}{d^{\abs{\xi}}} \binom{\frac{p+q-1}{d} - l (a+b)}{\frac{p}{d} - lb} 
        Y(\xi) \ .
    \end{equation}
    
    For $d \mid g(p)$, 
    \begin{equation}
        G(d) =
        \sum_{\substack{0 \leq l \leq \left\lfloor \frac{g_0}{d} \right\rfloor \\ \xi \vdash l}} 
        (-1)^{\abs{\xi}} \frac{(p+q-1)^{\abs{\xi}-1}}{d^{\abs{\xi}}} \binom{\frac{p+q-1}{d} - l (a+b)}{\frac{q}{d} - la} \ Y(\xi) \ .
    \end{equation}
\end{proposition}

\begin{proof}

We only prove the case $d \mid g(q)$. The other is similar.

Given $\pttp \in \Ptt(\pp_{p, q} / d)$, let $\xi$ be the integer partition of $l$ corresponding to $\pttp$ by Lemma \ref{lem:pp_pq/d_partition}, with $0 \leq l \leq \left\lfloor \frac{g_0}{d} \right\rfloor$. 
It is easy to verify that $\abs{\pttp} = \abs{\xi} + 1$ and 
\begin{equation*}
    \frac{X(\pttp)}{(\pttp)!} 
    = \binom{\frac{p+q-1}{d} - l (a+b)}{\frac{q}{d} - la} \ Y(\xi) \ .
\end{equation*}
Thus, 
\begin{align*}
       G(d) 
    &= \frac1{\abs{\pp_{p, q}} - 1} 
       \sum_{\pttp \in \Ptt(\pp_{p, q} / d)} 
       (-1)^{\abs{\pttp} - 1} (\abs{\pp_{p, q} / d} - 1)^{\abs{\pttp} - 1} \frac{X(\pttp)}{(\pttp)!} \\
    &= \frac1{p+q-1} 
       \sum_{\substack{0 \leq l \leq \left\lfloor \frac{g_0}{d} \right\rfloor \\ \xi \vdash l}} 
       (-1)^{\abs{\xi}} \left( \frac{p+q-1}{d} \right)^{\abs{\xi}} \binom{\frac{p+q-1}{d} - l (a+b)}{\frac{q}{d} - la} \ Y(\xi) \\
    &= \sum_{\substack{0 \leq l \leq \left\lfloor \frac{g_0}{d} \right\rfloor \\ \xi \vdash l}} 
       (-1)^{\abs{\xi}} \frac{(p+q-1)^{\abs{\xi}-1}}{d^{\abs{\xi}}} \binom{\frac{p+q-1}{d} - l (a+b)}{\frac{q}{d} - la} \ Y(\xi) \ . 
    \tag*{\qedhere}
\end{align*}
\end{proof}

\begin{corollary} \label{cor:G(d)_qppq_nondecomposable} 
    $\pp_{p, q}$ and $\pp_{p, q}/d$ are non-decomposable if and only if $p, q$ are co-prime (or $g_0 = 1$). 
    In this case,
    \begin{equation} \label{eq:G(d)_qppq_nondecomposable}
        G(d) = 
        \begin{cases}
            \dfrac{(p+q-2)!}{p!q!} & \text{ if } d=1 \vspace{5pt} \\ 
            \dfrac{1}{p+q-1} \dbinom{({p+q-1})/d}{{p}/{d}} & \text{ if } 1<d \mid g(q) \vspace{5pt} \\ 
            \dfrac{1}{p+q-1} \dbinom{({p+q-1})/{d}}{{q}/{d}} & \text{ if } 1<d \mid g(p) 
        \end{cases}
        \ .
    \end{equation}
\end{corollary}

\begin{proof}
If $g_0 = 1$, 
by Lemma \ref{lem:pp_pq_subpassport} and \ref{lem:pp_pq/d_subpassport}, 
the only subpassport of $\pp_{p, q}$ is $\pp(1)$ and that of $\pp_{p, q} / d$ is $\pp_d(0)$. 
Both are trivial subpassports. 
Thus $\pp_{p, q}$ and $\pp_{p, q} / d$ are non-decomposable. 
If $g_0 > 1$, $\pp(1)\neq \pp_{p,q}$ is a subpassport, thus $\pp_{p, q}$ and $\pp_{p, q} / d$ will be decomposable. 

\medskip
Now assume $g_0 = 1$, then $a = p$ and $b = q$. 
Since the only partition of $g_0 = 1$ is $\xi = (1^1)$, by Proposition \ref{prop:G(1)_qppq}, 
\begin{equation*}
    G(1) 
    = (-1)^{1-1} (p+q-1)^{1-2} Y(1^1) 
    = \frac1{p+q-1} \cdot \frac{(p+q-1)!}{p! q!} \ .
\end{equation*}

Assume $1 < d \mid g(q)$. The only choice for $0 \leq l \leq \lfloor \frac{g_0}{d} \rfloor$ is $l=0$ and $\xi$ must be the empty partition $\xi_0$. Then by Proposition \ref{prop:G(d)_qppq}, 
\begin{equation*}
    G(d) 
    = (-1)^{0} \frac{(p+q-1)^{0-1}}{d^{0}} \binom{{(p+q-1)}/{d}}{{p}/{d}} 
    Y(\xi_0)
    = \frac1{p+q-1} \binom{{(p+q-1)}/{d}}{{p}/{d}} 
    \ .
\end{equation*}

The case where $1 < d \mid g(p)$ is similar. 
\end{proof}

\begin{example}
    Set $p = 7, q = 3$, i.e. $\pp_{7, 3} = (3^7\ \ol{7}^3)$. 
    Then $g_0 = g(q) = 1, g(p) = 3$. 
    By Corollary \ref{cor:G(d)_qppq_nondecomposable}, 
    \begin{equation*}
        G(1) = \frac{(7+3-2)!}{7!3!} = \frac{4}{3} ,
    \end{equation*}
    \begin{equation*}
        G(3) = \frac{1}{7+3-1} \binom{{7+3-1} / {3}}{{3} / {3}} = \frac{1}{3} .
    \end{equation*}
    Then by Theorem \ref{thm:number_of_Tree_Xi}, the total number of trees in $\Tree(\Xi_{7,3})$ is 
    \begin{equation*}
        \abs{\Tree(\Xi_{7,3})} = \varphi(1) G(1) + \varphi(3) G(3) = \frac{4}{3} + 2 \times\frac{1}{3} = 2 .
    \end{equation*}
    By Theorem \ref{thm:number_of_Tree_Xi_d}, the number of trees in $\Tree(\Xi_{7,3}, 1)$ and $\Tree(\Xi_{7,3}, 3)$ respectively are
    \begin{equation*}
        \abs{\Tree(\Xi_{7,3}, 1)} = \mu\left(1\right) G(1) + \mu\left(3\right) G(3) = 1 ,\quad \abs{\Tree(\Xi_{7,3}, 3)} = 3 \mu\left(1\right) G(3) = 1 .
    \end{equation*}
    
    All these trees are shown in Figure \ref{fig:tree_p7q3}.
\end{example}

\begin{figure}[htbp]
    \centering
    \begin{subfigure}{0.59\textwidth}
        \centering
        \begin{tikzpicture}
            \node[wht] (M) at (0, 0) {};
            \node[blk] (U) at (0, 1) {};
            \node[blk] (L) at (-1, 0) {};
            \node[wht] (LL) at (-2, 0) {};
            \node[blk] (LLU) at (-2, 1) {};
            \node[blk] (LLL) at (-3, 0) {};
            \node[blk] (R) at (1, 0) {};
            \node[wht] (RR) at (2, 0) {};
            \node[blk] (RRU) at (2, 1) {};
            \node[blk] (RRR) at (3, 0) {};

            \draw (M) -- (U) node[mid_auto] {3};
            \draw (L) -- (M) node[mid_auto] {2};
            \draw (LL) -- (L) node[mid_auto] {1};
            \draw (LL) -- (LLU) node[mid_auto] {3};
            \draw (LLL) -- (LL) node[mid_auto] {3};
            \draw (M) -- (R) node[mid_auto] {2};
            \draw (R) -- (RR) node[mid_auto] {1};
            \draw (RR) -- (RRU) node[mid_auto] {3};
            \draw (RR) -- (RRR) node[mid_auto] {3};
        \end{tikzpicture}
        \caption{The tree in $\Tree(\Xi_{7,3}, 1)$}
        \label{subfig:tree_p7q3_rot1}
    \end{subfigure}
    \hfill
    \centering
    \begin{subfigure}{0.39\textwidth}
        \centering
        \begin{tikzpicture}
            \node[blk] (M) at (0, 0) {};
            \begin{scope}[rotate =0]
                \node[wht] (D) at (0, -1) {};
                \node[blk] (DL) at (-0.87, -1.5) {};
                \node[blk] (DR) at (0.87, -1.5) {};
                \draw (M) -- (D) node[mid_auto] {1};
                \draw (DL) -- (D) node[mid_auto] {3};
                \draw (D) -- (DR) node[mid_auto] {3};
            \end{scope}
            \begin{scope}[rotate =120]
                \node[wht] (D) at (0, -1) {};
                \node[blk] (DL) at (-0.87, -1.5) {};
                \node[blk] (DR) at (0.87, -1.5) {};
                \draw (M) -- (D) node[mid_auto] {1};
                \draw (DL) -- (D) node[mid_auto] {3};
                \draw (D) -- (DR) node[mid_auto] {3};
            \end{scope}
            \begin{scope}[rotate =240]
                \node[wht] (D) at (0, -1) {};
                \node[blk] (DL) at (-0.87, -1.5) {};
                \node[blk] (DR) at (0.87, -1.5) {};
                \draw (M) -- (D) node[mid_auto] {1};
                \draw (DL) -- (D) node[mid_auto] {3};
                \draw (D) -- (DR) node[mid_auto] {3};
            \end{scope}

        \end{tikzpicture}
        \caption{The tree in $\Tree(\Xi_{7,3}, 3)$}
        \label{subfig:tree_p7q3_rot3}
    \end{subfigure}
    \caption{Two trees in $\pp_{7, 3}$.}
    \label{fig:tree_p7q3}
\end{figure}
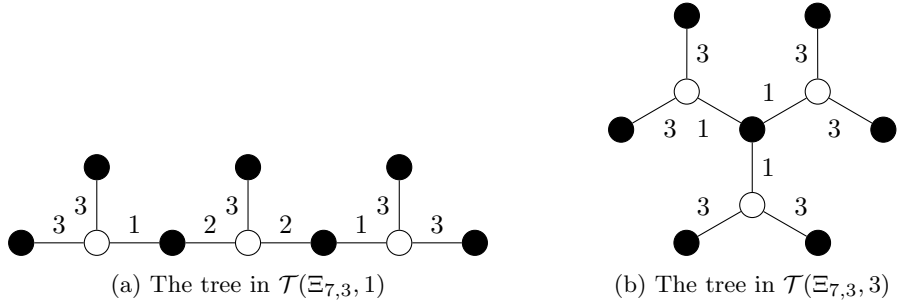

\bigskip
\section{Branched covers and dessins d'enfants}\label{sec:count.covering}

We are considering compatible branch datum with $g_X=0$ and $n=3$ : 
\[ \mathcal{D} = (g_Y,0,d,3;\ \pi_1,\pi_2,\pi_3) \ . \]
Since we are considering the strong Hurwitz number, the three partitions in $\mathcal{D}$ are regarded as an ordered triple. 
Up to a M\"{o}bius transformation on target surface $X=\overline{\mathbb{C}}$, we always assume the 3 branched points are $\{0,1,\infty\}$. Besides, we appoint that the local degrees at pre-images of $0,1,\infty$ are prescribed by $\pi_1, \pi_2, \pi_3$ respectively.

\subsection{From coverings to graphs}\label{ssec:d_d}

We first briefly review how to turn a branched cover $f: Y\to X$ with $g_X=0$ and $n=3$ into an embedded bi-colored graph on $Y$. This is essentially the dessin d'enfant originally introduced by Grothendieck \cite{Gro97}. 
The following procedures can be found in literatures like \cite[\S 1.5.2]{LsZak04} and \cite[\S 1]{Petr19}.

On the target surface $X=\overline{\mathbb{C}}$, we regard the point $0$ as a black vertex, the point $1$ as a white vertex, the closed interval $[0,1]$ as an edge between vertices. This is a bi-colored embedded graph on $\overline{\mathbb{C}}$. Its complementary is a topological polygon of 2 sides, containing $\infty$. 

Now pull back this graph on $Y$ through the covering $f$. We have an embedded bi-colored graph $H$ on $Y$ with the following properties. 
\begin{enumerate}
    \item Every point in $f^{-1}(0)$ ($f^{-1}(1)$ resp.) is a black (white resp.) vertex. 
    \item For each vertex of $H$, its degree equals to the local degree of $f$ at that vertex.
    \item $H$ contains $d$ edges. Each connected component of $Y \setminus H$ is a topological polygon. 
    \item Each component above exactly contains one point of $f^{-1}(\infty)$. Furthermore, if the local degree at that point is $m$, then the component containing it has $2m$ sides. 
\end{enumerate}

In the theory of Riemann surface, such a graph (usually together with $f$ as a meromorphic function on $Y$) is named \textbf{dessin d'enfant}, whose literal meaning is ``doodle by children''. 
The key is that, there is a canonical degree $m$ branched cover from a complementary $2m$-gon to $\overline{\mathbb{C}} \setminus [0,1]$, with a unique critical point inside it (see Figure \ref{fig:cano_poly}).

\begin{figure}[t]
\centering
\begin{tikzpicture}[scale=0.9, transform shape]
    \def\radius{1.5}
    \def\angle{60}

    \begin{scope}[xshift=-3cm]
    \fill[gray, opacity=0.3] (\angle*0:\radius) -- (\angle*1:\radius) -- (0,0) -- cycle;
    \fill[gray, opacity=0.3] (\angle*2:\radius) -- (\angle*3:\radius) -- (0,0) -- cycle;
    \fill[gray, opacity=0.3] (\angle*4:\radius) -- (\angle*5:\radius) -- (0,0) -- cycle;
    
    \foreach \i in {0,...,5}
    {
        \ifodd\i
            \node[wht](v\i) at (\i*\angle:\radius) {}; 
        \else
            \node[blk](v\i) at (\i*\angle:\radius) {}; 
        \fi
    }

    \draw[line width=1.5pt] (v0) -- (v1) -- (v2) -- (v3) -- (v4) -- (v5) -- (v0);
    \node[crossmark] at (0,0) {};
    \end{scope}

    \begin{scope}[xshift=4cm, yshift=-0.2cm]
        \coordinate (B) at (-0.8,-0.3);
        \coordinate (W) at (+0.8,-0.3);
        \coordinate (F) at (0,1.8);
        \fill[gray!30] (W) to[out=90,in=-45] (F) to[out=-135,in=90] (B) -- cycle;
        \draw[line width=1.5pt] (B) -- (W);
        \node[blk] at (B) {};
        \node[wht] at (W) {};
        \node[crossmark] at (F) {};
        \draw (0,0) circle (1.8);
        \node at (-1.1, -0.6) {0};
        \node at (+1.1, -0.6) {1};
        \node at (0, +2.2) {$\infty$};
        \node at (-1.8,+1.6) {$\overline{\mathbb{C}} \cong $};
    \end{scope}

    \draw[-stealth] (-0.7,0) -- (1.5,0) node[midway, above] {\rm $3$ to $1$};

\end{tikzpicture}    
\caption{The canonical branched cover from a complementary polygon.}
\label{fig:cano_poly}
\end{figure}
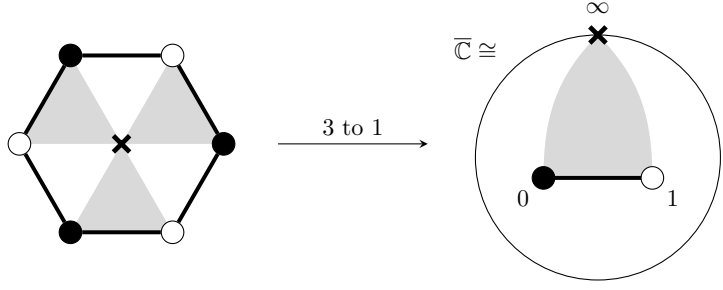

\bigskip
Conversely, the existence of such embedded graph also implies the existence of a branched cover. Given a compatible branch datum 
$ \mathcal{D} = (g_Y,0,d,3;\ \pi_1,\pi_2,\pi_3) $. 
If there exists an embedded bi-colored graph $H$ on the topological closed surface $Y$ of genus $g_Y$ satisfying all the following conditions, then there is a branched cover of branch datum $\mathcal{D}$.
\begin{enumerate}
    \item Every edge in $H$ connects a black vertex to a white vertex.
    \item The distribution of degrees at black (white resp.) vertices are exactly given by $\pi_1$ ($\pi_2$ resp.).
    \item $H$ contains $d$ edges. Each connected component of $Y \setminus H$ is a topological polygon. 
    \item The distribution of side numbers of the components above, when divided by 2, are exactly given by $\pi_3$. 
\end{enumerate}
Furthermore, the homeomorphism classes of such embedded bi-colored graphs are in 1-1 correspondence to the strong equivalence classes of branched cover of branch datum $\mathcal{D}$. 
Hence the strong Hurwitz number equals to the number of certain embedded graphs on surface. This is the key perspective of the current work.

\subsection{The special case and weighted graph}

Given two arbitrary partitions $\pi_1, \pi_2$ of $d$, we now focus on the specific branch datum 
\[ \mathcal{D} = (0,0,d,3;\ \pi_1,\pi_2,\pi_3) \qquad \text{ with }\quad  \pi_3 := ( 1^{d-M}\ M^1 ) \ . \] 
The Riemann-Hurwitz Formula \eqref{eq:Riemann-Hurwitz} requires that 
\[ M= |\pi_1|+|\pi_2| -1 \leq d \ .\]
For any branched cover $f:Y \to X=\overline{\mathbb{C}}$ of datum $\mathcal{D}$, there is a unique critical point in $f^{-1}(\infty)$, whose degree is just $M$. 
Applying a M\"{o}bius transformation on the source surface $Y=\overline{\mathbb{C}}$, we may always set this critical point at $\infty \in Y$. Then the previously defined embedded graph $H$ on $Y$ becomes an embedded plane graph, whose outer boundary has $2M$ sides. All the other components in $Y \setminus H$ are topological polygons with 2 sides. 

To have a more condensed object that encodes a branched cover, we apply the following operations, turning the above graph $H$ into a weighted one \cite[\S 1.3]{APZak20}.

If a pair of black and white vertices of $H$ are connected by $w$ contiguous edges, and if they bounded $w-1$ successive topological polygons with 2 sides, then keep one of them and eliminate all the other $(w-1)$ edges. By the assumption, all these $w$ edges are homotopic relative to their vertices, so the result is independent of the choice of a remained edge.
Next, assign weight $w$ to the remained edge. 
Finally, assign weight $1$ to all the other edges not involved in the previous steps. 
The weights on edges naturally induce weights on vertices. 

Since all components in $Y \setminus H$ are topological polygons with 2 sides, the result weighted graph does not bound any closed region on the plane. So it must be a weighted bi-colored plane tree. 
Besides, the weights on black and white vertices are exactly given by the partitions $\pi_1$ and $\pi_2$, respectively.

\begin{figure}[t]
\makebox[\textwidth]{
\centering
\begin{tikzpicture}[scale = 1.1]
\tikzset{
    cross/.style = {cross out, draw=gray, fill=none, inner sep=0pt, minimum size=4pt, line width=0.8pt},
}
    \node[blkdot] (M) at (0, 0) {};
    \node[whtdot] (a) at (+1, +0.6) {};
    \node[blkdot] (b) at (+2, -0.5) {};
    \node[blkdot] (c) at (+1, +1.8) {};
    \node[whtdot] (A) at (-1, -0.6) {};
    \node[blkdot] (B) at (-2, +0.5) {};
    \node[blkdot] (C) at (-1, -1.8) {};
    \draw (A) -- (M) -- (a);
    \draw (A) to[out= 90, in=120] (M) to[out=-60, in=-90] (a);
    \draw (A) to[out= 90, in= 45] (B) ;
    \draw (A) to[out=  0, in= 90] (0,-1) to[out=-90, in=20] (C) to[out=160, in= -90] (-2,-1) to[out=90, in=180] (A) ;
    \draw (A) to[out=-30, in=45] (C) to[out=135, in=-150] (A);
    \draw (A) to[out=150, in=-90] (B);
    \draw (a) to[out=-90, in=225] (b) ;
    \draw (a) to[out=180, in=-90] (0,+1) to[out=+90, in=200] (c) to[out=-20, in= 90] (+2,+1) to[out=-90, in=  0] (a) ;
    \draw (a) to[out=150, in=225] (c) to[out=-45, in=+30] (a);
    \draw (a) to[out=-30, in=+90] (b);
    
    \node[cross] at (-0.5,-0.1) {};
    \node[cross] at (-1.5,+0.1) {};
    \node[cross] at (-1.0,-1.2) {};
    \node[cross] at (-1.6,-1.1) {};
    \node[cross] at (-0.4,-1.1) {};
    \node[cross] at (+0.5,+0.1) {};
    \node[cross] at (+1.5,-0.1) {};
    \node[cross] at (+1.0,+1.2) {};
    \node[cross] at (+1.6,+1.1) {};
    \node[cross] at (+0.4,+1.1) {};

    \draw[-Latex, double, line width=0.6] (3,0) -- (4,0) {};
    
\begin{scope}[shift = {(7,0)}]
    \node[blk] (M) at (0, 0) {4};
    \node[wht] (a) at (+1, +0.6) {$\overline{8}$};
    \node[blk] (b) at (+2, -0.5) {2};
    \node[blk] (c) at (+1, +1.8) {4};
    \node[wht] (A) at (-1, -0.6) {$\overline{8}$};
    \node[blk] (B) at (-2, +0.5) {2};
    \node[blk] (C) at (-1, -1.8) {4};
    \draw (M) -- (a) node[mid_auto] {2};
    \draw (a) -- (b) node[mid_auto] {2};
    \draw (a) -- (c) node[mid_auto, right] {4};
    \draw (M) -- (A) node[mid_auto] {2};
    \draw (A) -- (B) node[mid_auto] {2};
    \draw (A) -- (C) node[mid_auto, left] {4};
\end{scope}

\end{tikzpicture}
}
\caption{From a special dessin d'enfant to a WBP-tree.}
\label{fig:weighted_graph}
\end{figure}
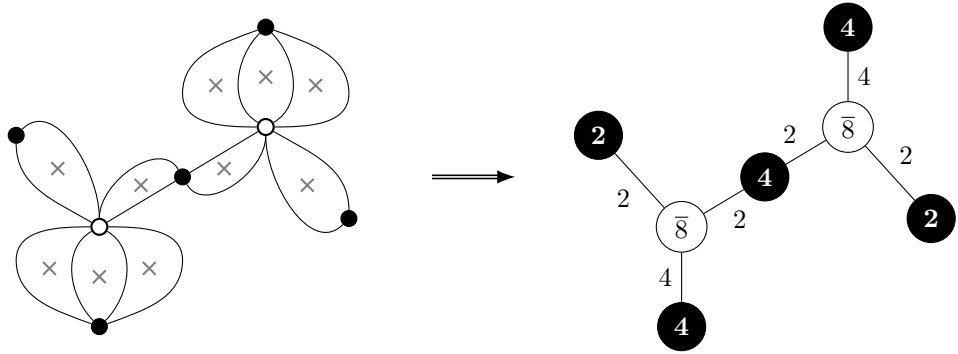

Figure \ref{fig:weighted_graph} shows an example where $ \pi_1 = (2^2\ 4^3),\ \pi_2 = (8^2),\ \pi_3= (1^{10}\ 6^1) $ are three partitions of $d=16$. The left is a dessin d'enfant from a branched cover, with $\infty \in Y$ a critical point of local degree $6$. All the other cross marks are the preimage of $\infty \in X$ but not critical points. The right is the WBP-tree obtained by previous operations, exactly the one in Figure \ref{subfig:tree_2^24^38^2_A}. 

\begin{definition}
    Let $\pi_1, \pi_2$ be two partitions of $d$ as 
    \[ \pi_1 = ( 1^{\lambda_1}\dots\ d^{\lambda_d} ), \quad \pi_2 = ( 1^{\mu_1} \dots\ d^{\mu_d} ) \ .\]
    Then define a trivial passport
    \[ 
    \pp(\pi_1, \pi_2) := 
    \left( 1^{\lambda_1} \cdots\ d^{\lambda_d} \ \ol{1}^{\mu_1} \cdots\ \ol{d}^{\mu_d} \right)
    \ . 
    \]
\end{definition}

\begin{theorem}\label{thm:st_H_number}
    Let $\pi_1, \pi_2$ be two partitions of $d>2$, $\pi_3 := ( 1^{d-M}\ M^1 )$ with $M= |\pi_1|+|\pi_2| -1$. Define the compatible branch datum 
    \[ \mathcal{D}_{(\pi_1,\pi_2)} := (0,0,d,3;\ \pi_1,\pi_2,\pi_3) . \]
    Then there is a bijection between the strong equivalence classes of branched cover of branch datum $\mathcal{D}_{(\pi_1,\pi_2)}$, and LWBP-trees of trivial passport $\pp(\pi_1, \pi_2)$. 
    
    In particular, the strong Hurwitz number of branch datum $\mathcal{D}_{(\pi_1,\pi_2)}$ is $\abs{\Tree(\pp(\pi_1, \pi_2))}$. 
\end{theorem}

By taking $d = pq$ for some coprime $p,q$, and $\pi_1:=(q^p),\ \pi_2:=(p^q)$, we see $\pp(\pi_1,\pi_2) = ( q^p\ \bar{p}^q ) = \pp_{p,q}$, which already studied in Section \ref{ssec:compute}. 

\begin{proof}[Proof of Theorem \ref{thm:main_H.Number}]
By Theorem \ref{thm:st_H_number} and Theorem \ref{thm:number_of_Tree_Xi}, 
\begin{equation*}
    \HH(p, q) = \abs{\Tree(\pp_{p, q})} = \sum_{d \in D} \varphi(d) G(d) \ .
\end{equation*}
And $G(d)$ is shown in Corollary \ref{cor:G(d)_qppq_nondecomposable}. 
\end{proof}

\bigskip
\section{Moduli space of HCMU spheres}\label{sec:count.cpnt}

To study HCMU surfaces and their moduli space, we will introduce a surface representation method utilizing embedded weighted bi-colored graphs. This is based on the football decomposition in \cite[\S 2]{CCW05}. To make a short cut for description and application, we use the modified version in \cite[\S 2]{LuXu25}. 

Such representation serves as a useful tool to study the moduli space $\Mhcmu_{0,1}(\alpha)$ of HCMU spheres with a single conical singularity. The main result of this section (Theorem \ref{thm:cnntd_comp}) asserts that, when endowed with Gromov-Hausdorff topology, the connected components of $\Mhcmu_{0,1}(\alpha)$ are precisely classified by trees of trivial passports in Section \ref{ssec:compute}. To achieve this, we study the Gromov-Hausdorff limits of HCMU surfaces in $\Mhcmu_{0,1}(\alpha)$.

\subsection{Character line element, bigon and football}

We first review the building blocks for arbitrary HCMU surfaces. Most details can be found in \cite{CCW05,LuXu25} and reference within them. We only include the description of results. 

\begin{definition}
\label{defn:line_element}
    Given two constant $K_0, K_1$ satisfying 
    \begin{equation}\label{eq:K0K1_extd}
        K_0 >0 ,\quad K_0 > K_1 > -\frac{K_0}{2}  \ .
    \end{equation}
    An HCMU \textbf{character line element} of extremal curvatures $K_0, K_1$ is a closed interval $[0,l]$ endowed with a $C^2$-function $K(v)$, satisfying the following: 
    \begin{enumerate}[(1). ]
        \item $K(v)$ is decreasing, with initial value $K(0)=K_0$.
        \item $K(v)$ satisfies the differential equation
        \begin{equation}\label{eq:curv_along_v}
            3 {K'}^2 = -(K-K_0)(K-K_1)(K+K_0+K_1) \ .
        \end{equation}
        \item $\ell=\ell(K_0,K_1)\in \RR_+$ is the unique solution of $K(v)=K_1$, called the \textbf{length} of this character line element. 
        \item The \textbf{warped function} {$h(v)$} on the interval $[0,\ell]$ is defined to be 
        \begin{equation}\label{eq:density_h}
            h(v) := \frac{K'(v)}{\overline{c}}, \quad \textrm{ where }
        \overline{c} := -\frac16 (K_0-K_1)(2K_0+K_1) \ . 
        \end{equation}
        \item The \textbf{ratio} of the character line element is 
        \begin{equation}\label{eq:Ratio}
        R = R(K_0,K_1) := \frac{K_0+2K_1}{2K_0+K_1} < 1 \ .
           \end{equation}
    \end{enumerate} 
\end{definition}

\begin{definition}\label{defn:hcmu_bigon}
    An \textbf{HCMU bigon} of extremal curvatures $K_0, K_1$ and \textbf{top angle} $2\pi\alpha$, denoted by {$B_{\alpha}(K_0, K_1)$}, is a topological disk parameterized by a rectangle
    \[ (v, \phi) \in [0,\ell] \times [0, 2\pi\alpha] \]
    endowed with the metric 
    \begin{equation}\label{eq:hcmu_bigon_metric}
        \rho = \rmd v^2 + h^2(v) \rmd \phi^2 \ .
    \end{equation} 
    Here $h(v)$ is the warped function of character line element of extremal curvatures $K_0, K_1$ in Definition \ref{defn:line_element}, and $K_0,K_1$ satisfies \eqref{eq:K0K1_extd}. 
    See Figure \ref{fig:hcmu_football}.
    
    An \textbf{HCMU football} of extremal curvatures $K_0, K_1$ and \textbf{top angle} $2\pi\alpha$, denoted by {$S_{\alpha}(K_0, K_1)$}, is obtained by gluing the two boundaries $\{\phi=0\}, \{\phi=2\pi\alpha\}$ of bigon $B_{\alpha}(K_0, K_1)$ according to the $v$-parameter. 
\end{definition}

\begin{figure}[t]
    \centering
    \includegraphics[width=\linewidth]{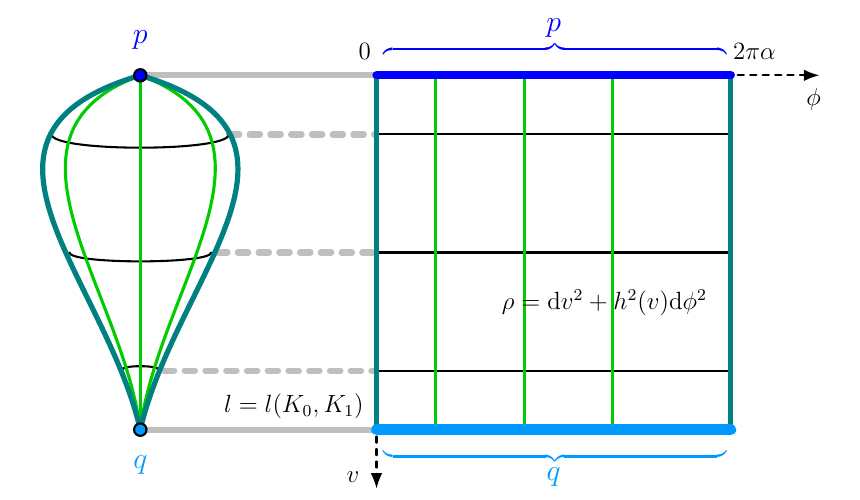}
    \caption{A typical HCMU bigon with its $(v,\phi)$ rectangular coordinates.}
    \label{fig:hcmu_football}
\end{figure}

Here are some facts about HCMU bigons and football. 

Geometrically, $B_{\alpha}(K_0, K_1)$ is a bordered surface with two geodesic boundaries. The line $\{v=0\}$ collapses to one corner point on the boundary, called the \textbf{top vertex}; and $\{v=\ell\}$ collapses to another corner point called \textbf{bottom vertex}. The inner angle at the bottom vertex, called \textbf{bottom angle}, is shown to be $2\pi\beta := 2\pi\alpha R$, where $R$ is the ratio of character line element. 

Similarly, $S_{\alpha}(K_0, K_1)$ is a rotationally symmetric genus zero HCMU surface with at most two conical singularities. The line $\{v=0\}$ collapses to a cone point, also called the \textbf{top vertex}, of angle $2\pi\alpha$. The \textbf{bottom vertex} is defined similarly, with \textbf{bottom angle} $2\pi\beta := 2\pi\alpha R$. 
In literature like \cite{CCW05, WzqWyy19, MyjWzq25}, such football is denoted by $S^2_{\{\alpha, \beta\}}$ when the ratio is $R = \beta / \alpha$. 

Intuitively, an HCMU bigon or football is weaved by some width of character line element. 
In the language of differential geometry, an HCMU bigon or football is actually a \emph{warped product} of two 1-dimensional Riemannian manifolds. 

A line or circle of constant $v$-coordinate is called a \textbf{waist line}, along whom the curvature $K$ is constant. 
By contrast, each line of constant $\phi$-coordinate is a geodesic of length $l$, called a \textbf{meridian}. 
$K$ is strictly decreasing along each meridian, and satisfies the differential equation \eqref{eq:curv_along_v}. 

The curvature function $K$ extends continuously to the two conical singularities. It attains its maximum $K_0$ (minimum $K_1$) at the top (bottom) vertex with larger angle $2\pi\alpha$ (smaller angle $2\pi\beta$).  

\medskip
Here are some asymptotic behavior about above objects used later. 

\begin{lemma} \label{lem:line_element_estimate}
	Fix $0<R <1$, there exits a constant $I(R)$ such that the length of character line element of extremal curvature $K_0>0,\ K_1 := \frac{2R-1}{2-R} K_0$ is 
	\[ \ell(K_0) := \frac{I(R)}{\sqrt{K_0}} \ . \]
	In particular, when $R$ is fixed, 
	\[ 
        \lim_{K_0 \to +\infty} \ell \left( K_0 \right) = 0 \ ,\ 
        \lim_{K_0 \to 0^+} \ell \left( K_0 \right) = + \infty \ .
    \]
\end{lemma}

\begin{proof}

Let $r := \dfrac{2R-1}{2-R}$, which is also fixed. Since $K(v)$ is decreasing and satisfies \eqref{eq:curv_along_v},  
$$
\frac{\rmd K}{\rmd v} = \frac{-1}{\sqrt3}  \sqrt{(K_0 - K)(K - K_1)(K + K_0 + K_1)}. 
$$
By the definition of $\ell = \ell(K_0, K_1)$, separating variables in this differential equation yields
\[
\ell = \int_{0}^{\ell} \rmd v 
=  \int_{K_1}^{K_0} \frac{\sqrt{3}\ \rmd K}{\sqrt{(K_0-K)(K-K_1)(K+K_0+K_1)}}
= \frac{\sqrt{3}}{\sqrt{K_0}} \int_{r}^{1} \frac{dx}{\sqrt{(1-x)(x-r)(x+1+r)}}.
\]
The last step applies substitution $K = x K_0$. Hence we can define
\[
I(R) := \sqrt{3} \cdot \int_{r}^{1} \frac{dx}{\sqrt{(1-x)(x-r)(x+1+r)}}, \quad r = \frac{2R-1}{2-R} < 1\ .
\]
This yields the desired result.
\end{proof}

\begin{lemma} \label{lem:waist_estimate}
    Fix $0 < R < 1$, let $\gamma_v$ denote the waist line whose first component is constantly $v$ in rectangular coordinates, and $L(\gamma_v)$ be its length. 
    Then as $K_0 \to + \infty$, the maximum length $\max\{ L(\gamma_v)\ |\ 0 \leq v \leq \ell(K_0) \}$ of all waist lines tends to $0$. 
\end{lemma}

\begin{proof}

By \eqref{eq:hcmu_bigon_metric} and \eqref{eq:density_h}, 
$L(\gamma_v) = 2 \pi \alpha h(v) = 2 \pi \alpha K'(v) \,/\, \bar{c}$. 
Set $r := (2R-1) \,/\, (2-R) = K_1 \,/\, K_0$ and $x(v) := K(v) / K_0$, with $x(v) \in [r, 1]$.
Then
$$
    \bar{c} = - \frac16 K_0^2 (1-r)(2+r) \ ,\quad 
    K' = - \frac13 K_0^{3/2} \sqrt{(1-x)(x - r)(x + 1 + r)} \ .
$$
The length of waist line satisfies
\begin{equation*}\begin{split}
    L(\gamma_v) 
    = 2 \pi \alpha \cdot \frac{K'}{\bar{c}} 
    &= 2 \pi \alpha \cdot \frac{2 \sqrt{(1-x)(x - r)(x + 1 + r)}}{(1-r)(2+r)} \cdot K_0^{-1/2} \\
    &\leq 2 \pi \alpha \cdot \frac{2 \sqrt{(1-r)(1 - r)(2 + r)}}{(1-r)(2+r)} \cdot K_0^{-1/2} = \frac{4 \pi \alpha}{\sqrt{2+r}} \cdot K_0^{-1/2} \ .
\end{split}\end{equation*}
Hence, as $K_0 \to \infty$, we have $\max \{L(\gamma_v)\} \to 0$.
\end{proof}

\medskip
\subsection{Data set representation}
Essentially, all concerned HCMU surfaces are built from gluing finitely many HCMU bigons along their boundaries. For a given HCMU surface, all these bigons are woven by the same character line element. 
So any HCMU surface is foliated by character line elements of the same extremal curvature. Our embedded graphs are basically used to recording the foliation. 

Here we shall adopt the statement in \cite[\S 2.4]{LuXu25}. It is equivalent to the original ``football decomposition'' version in \cite[Theorem 1.3]{CCW05}. 
With such expression, the appearance of embedded graphs is more natural. Furthermore, such decomposition is closely related to the structure of meromorphic 1-forms on Riemann surfaces. In fact, an HCMU surface can also be completely determined by a special class of meromorphic 1-forms \cite{CqWyy11, CWX15, MyjWzq25}.

\begin{proposition}[]\label{prop:hcmu_strip_decomp}
    Every HCMU surface is canonically decomposed into finitely many HCMU bigons by cutting along finite number of meridian geodesic segments, and all these HCMU bigons have the same extremal curvatures. 
\end{proposition}

When gluing bigons back to a surface, there is no conical singularity inside each bigon. Cone points in the result surface must come from the following two cases:
\begin{itemize}
    \item a point glued from top or bottom vertices of several bigons; 
    \item a point glued from several points on geodesic boundaries of bigons. 
\end{itemize}

Any point glued from top (or bottom) vertices of bigons must be a maximizer (or minimizer) of curvature function. Therefore, these points are called the \textbf{maximal (or minimal) point} of the HCMU surface. Such a point may either be a cone point, or a smooth point (regarded as a cone point of angle $2\pi$). 

Any point glued from boundary points of bigons must be a cone point on the result surface. We call it a \textbf{saddle point} because it is a saddle point of the curvature function. 
Since one geodesic boundary contributes angle $\pi$, if the cone angle at a saddle point is glued from $m$ boundary points, then $m$ must be even and the cone angle is $m\pi$ ($m\neq2$, otherwise it is a smooth point) .

\begin{definition}\label{def:critical_meridian}
    The segments that decompose an HCMU surface into bigons as Proposition \ref{prop:hcmu_strip_decomp} are called \textbf{critical meridians}. Such a segment must connect to at least one saddle point. 

    The \textbf{critical graph} of an HCMU surface is the union of all critical meridians as a ribbon graph. Equivalently, when a surface is glued from bigons, it is the image of all geodesic boundaries.
\end{definition}

With the decomposition result in Proposition \ref{prop:hcmu_strip_decomp}, we have a coarse classification of HCMU spheres with a unique conical singularity. 

\begin{lemma}[Also see \cite{Cxx00}, Theorem 2]\label{lem:classify}
    Let $g=0, n=1$ and $\alpha\in \RR_+\setminus\{1\}$. Assume $\Sigma \in \Mhcmu_{0,1}(\alpha)$. 
    \begin{enumerate}
        \item When $0 < \alpha < 1$, any $\Sigma$ is a football $S_1(K_0, \frac{2\alpha-1}{2-\alpha}K_0)$ for some $K_0\in\RR_+$.
        
        \item\label{lem:ftb_case_2} When $\alpha>1$ and $\alpha \notin \ZZ$, any $\Sigma$ is a football $S_\alpha(K_0, \frac{2-\alpha}{2\alpha-1}K_0)$ for some $K_0\in\RR_+$.
        
        \item When $\alpha>1$ and $\alpha \in \ZZ$, if the unique cone point of $\Sigma$ is not a saddle point, then it is a football as case \ref{lem:ftb_case_2} again.  
    \end{enumerate}
\end{lemma}

Therefore, we only need to consider $\alpha \in \ZZ_{>1}$ and surfaces in $\Mhcmu_{0,1}(\alpha)$ whose unique cone point is a saddle point. For convenience, we shall call such a HCMU surface \textbf{generic}. 
This is compatible with the original definition in \cite[\S 3.1]{LuXu25}.

The following is a specific and simplified version of the data set representation theorem of \cite[\S 3.4]{LuXu25}, enough and complete for our study of $\Mhcmu_{0,1}(\alpha)$. 
To avoid introducing excessive detail, the result is directly stated without proof. Details of the procedure in general case is basically the same as our next example. 

\begin{theorem}[Data set representation, for $g=0, n=1$]
\label{thm:hcmu_data} ~\\
    Let $\alpha \in \ZZ_{>1}$. 
    There is a 1-1 correspondence between all generic 
    HCMU spheres with a unique conical singularity of angle $2\pi\alpha$, and the data sets of the following forms: 
    \[ \left( p,q,T; K_0, L \right) \]
    where 
    \begin{enumerate}[1. ]
        \item $p>q>0$ are positive integers with $p+q=\alpha+1$; 
        \item if $q>1$, then $q \nmid p$;
        \item $T$ is an LWBP-tree of trivial passport $\pp_{p,q} := (q^p\ \bar{p}^q)$;
        \item $K_0\in\RR_+$ is arbitrary, and 
        \[ R:=\frac{q}{p}\ ,\quad K_1 := \frac{2R-1}{2-R}K_0 = \frac{2q-p}{2p-q}K_0 \ ; \] 
        \item $L \in (0,\ell)$, where $\ell=\ell(K_0,K_1)$ is the length of character line element of extremal curvatures $K_0, K_1$.
    \end{enumerate}
\end{theorem}

We shall explain how to build a generic HCMU surface from a data set through the following example. The reverse process is just the bigon decomposition. 
Note that the tree $T$ here is not the critical graph in Definition \ref{def:critical_meridian}.

\begin{example}\label{eg:hcmu_31}
    Let $\alpha=3, p=3, q=1$ and $\pp_{3,1}:= (1^3\ \bar{3}^1)$. The unique tree $T_0 \in \Tree(\pp_{3,1})$ is shown in Figure \ref{subfig:tree_label_A} before. 
    Take $K_0 \in \RR_+$ and let 
    $ R=\frac13, K_1= \frac{2q-p}{2p-q}K_0 = -\frac{K_0}{5} $. 
    Let $\ell$ be the length of character line element of extremal curvatures $K_0, K_1$, and $L \in (0,\ell)$. 
    Here are the steps to recover the generic HCMU surface whose data set representation is $(3,1,T_0; K_0, L)$.

    \medskip
    \noindent (Step 1). For each edge $e$ of $T_0$, build an HCMU bigon $B_e := B_{\alpha(e)}(K_0,K_1)$ with top angle 
    \[ 2\pi\alpha(e) := 2\pi\wt_E(e)/q \ . \] 
    Recall that $\wt_E$ is the weight function on edges of $T_0$.
    
    \medskip
    \noindent (Step 2). Mark the point at distance $L$ from the top vertex on every geodesic boundary in each HCMU bigon. 
    Since all these geodesic boundaries have the same length $\ell$ and $0<L<\ell$, each marked point must lie on the interior of boundary segment. 

    \medskip
    \noindent (Step 3). Each geodesic boundary is divided into two meridian segments by the marked point. Now glue the bigons along these segments by isometries, according to cyclic action at each vertex of $T_0$. 
    If $v$ is a black vertex, $e,e' \in E(v)$ and $\sigma_{v}(e) = e'$, then isometrically glue $B_{e}, B_{e'}$ along the boundary meridian segments bounded by top vertex and marked point, placing $B_{e'}$ anticlockwise from $B_{e}$ about the common top vertex. $e, e'$ may be the same edge. Apply the same procedure to white vertices, with `top vertex' replaced by `bottom vertex'.
    See the right of Figure \ref{fig:gluing_bigons}. 

    \medskip
    \noindent (Step 4). It is easy to check that the result surface is a topological sphere. 
    The top vertex of each bigon is glued to a cone of angle $2\pi$. So we have 3 smooth maximal points. Three bottom vertices are glued to one cone of angle $\frac{2\pi}{3} \cdot 3 = 2\pi$, hence a smooth minimal points. 
    The six marked boundary points are glued to a saddle point of angle $6\pi$. 
    Therefore, the result HCMU surface belongs to $\Mhcmu_{0,1}(3)$, whose data set representation is just $(3,1,T_0;K_0,L)$. 
    All geodesic boundaries of the bigons form the critical graph (see Figure \ref{fig:critical_graph}).  
    \qed
\end{example}

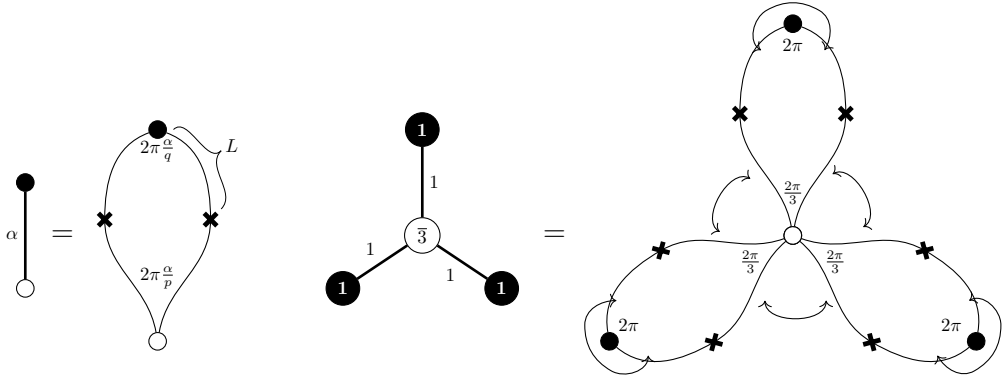
\begin{figure}[t]
    \centering
    \begin{tikzpicture}[scale=0.7, transform shape]
        \tikzset{
        bigon2/.pic={
        \node[blk] (u) at (0, 2) {};
        \node[wht] (d) at (0,-2) {};
        \node(ml) at (-1,0) {};
        \node(mr) at ( 1,0) {};            
        \draw (u) to[out=-160,in=90] (-1,0.3) to[out=-90,in=100] (d);
        \draw (u) to[out=-20,in=90] ( 1,0.3) to[out=-90,in=80] (d);
        \node[crossmark] at (-1,0.3) {};
        \node[crossmark] at ( 1,0.3) {};
        }}

        \tikzset{
        bigon2n/.pic={
        \pic{bigon2};
        \draw[<->] (-0.6,1.5) to[out=150, in =-150](-0.5, 2.25) to[out=30, in =150](0.5, 2.25) to[out=-30, in =60] (0.6,1.5);
        \draw[<->] (0.75,-0.84) to[out=30, in=30](1.33, -1.84);
        }}
        
        \node[blk] (up) at (0, 1) { };
        \node[wht] (dn) at (0,-1) { };
        \draw[line width=1pt] (up) -- (dn) node[midway, left] {$\alpha$};

        \node() at (0.7, 0) {\Large $=$}; 
            
        \begin{scope}[xshift=2.5cm]
        \pic{bigon2};
        \node() at (0, 1.6) {$2\pi\frac{\alpha}{q}$};
        \node() at (0,-0.8) {$2\pi\frac{\alpha}{p}$};
        \draw[line width=0.2] (0.3,2) to[out=45, in=-135] (1.2, 1.6) to[out=-135, in=45] (1.2, 0.5) ;
        \node at (1.4,1.7) {$L$};
        \end{scope}
        
        \begin{scope}[xshift=7.5cm, yshift=0cm]
        \node[blk] (up) at (0, 2) {1};
        \node[wht] (mid) at (0, 0) {$\overline{3}$};
        \node[blk] (left) at (-1.5, -1) {1};
        \node[blk] (right) at (1.5, -1) {1};
        \draw[line width=1pt] (up) -- (mid) node[mid_auto] {1};
        \draw[line width=1pt] (left) -- (mid) node[mid_auto] {1};
        \draw[line width=1pt] (right) -- (mid) node[mid_auto] {1};
        
        \node() at (2.5, 0) {\Large $=$};
        \end{scope}
        
        \begin{scope}[xshift=14.5cm]
        \pic[yshift=2cm]{bigon2n};
        \pic[rotate=120, yshift=2cm]{bigon2n};
        \pic[rotate=240, yshift=2cm]{bigon2n};
        \end{scope}

        \node() at (14.5,3.6) {$2\pi$};
        \node() at (14.5,0.8) {$\frac{2\pi}{3}$};
        \node() at (11.4,-1.7) {$2\pi$};
        \node() at (13.7,-0.5) {$\frac{2\pi}{3}$};
        \node() at (17.5, -1.7) {$2\pi$};
        \node() at (15.3,-0.5) {$\frac{2\pi}{3}$};
        
    \end{tikzpicture}
    
\caption{Building an HCMU sphere from a data set that includes an LWBP-tree.}
\label{fig:gluing_bigons}
\end{figure}

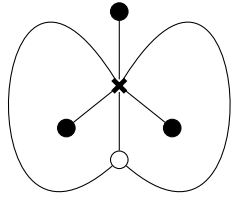
\begin{figure}[h]
\centering
\begin{tikzpicture}[line join=round, scale=0.7, transform shape]
    \clip (-2.5,-2.5) rectangle (2.5,1.6);
    \node[crossmark](x) at (0,0) {};
    \node[blk](a) at (0,1.4) {};
    \node[blk](b) at (-1,-0.8) {};    
    \node[blk](c) at (+1,-0.8) {};
    \node[wht](w) at (0,-1.4) {};
    \draw (a) -- (x);
    \draw (b) -- (x);
    \draw (c) -- (x);
    \draw (w) -- (x);
    \draw (x) .. controls (-2.5,4) and (-3,-4) .. (w);
    \draw (x) .. controls (+2.5,4) and (+3,-4) .. (w);
\end{tikzpicture}
\caption{The critical graph of the surface in Example \ref{eg:hcmu_31}.}
\label{fig:critical_graph}
\end{figure}

In summary, if a generic HCMU surface $\Sigma\in \Mhcmu_{0,1}(\alpha)$ is represented by the data set $(p,q,T;K_0,L)$, then there are $p$ maximal points and $q$ minimal points on $X$, all smooth. 
The whole surface is woven by character line elements of extremal curvature $K_0, K_1=\frac{2q-p}{2p-q}K_0$. Its bigon decomposition is faithfully recorded by the LWBP-tree $T$ of trivial passport $(q^p\ \bar{p}^q)$. 
The unique conical singularity is a saddle point of cone angle $2\pi\alpha$, whose position is controlled by the parameter $L$. 

\begin{definition}\label{def:M_T}
	A generic $\Sigma \in \Mhcmu_{0,1}(\alpha)$ represented by data set $(p,q,T;K_0,L)$ will be denoted by $\Sigma(T,K_0,L/\ell)$ in short, with $K_0 \in \Rpos,\ 0 < L/\ell < 1$. 
	
	Assume $p,q$ satisfy the conditions in Theorem \ref{thm:hcmu_data}. For an LWBP-tree $T \in \Tree(\pp_{p,q})$, 
	let $\mathcal{M}(T) \subset \Mhcmu_{0,1}(\alpha)$ be the subset of all generic surfaces that contain $T$ in their data set representation. Equivalently, $\mathcal{M}(T):=\defset{\Sigma(T,K_0,\lambda)}{K_0>0,\ 0<\lambda<1}$.  
\end{definition}
The ratio is determined by $p,q$ and encoded in the tree $T$. Hence the notation $\Sigma(T,K_0,\lambda)$ is well-defined. 
By Theorem \ref{thm:hcmu_data}, the subset of all generic HCMU surfaces in $\Mhcmu_{0,1}(\alpha)$ has a natural decomposition 
\[ \bigsqcup_{p,q}\ \bigsqcup_{T \in \Tree(\pp_{p,q}) } \mathcal{M}(T) \ . \]
This is almost the connected component decomposition of the geometric moduli space. The remainder of this paper explains how to adjust it to obtain the genuine decomposition of the whole space $\Mhcmu_{0,1}(\alpha)$.

\subsection{The moduli space and Gromov-Hausdorff limit}

To show that the connected components of $\Mhcmu_{0,1}(\alpha)$ are classified by the trees appearing in the data set representation, we will consider the finite Gromov-Hausdorff limits of these surfaces at the boundary of moduli space.  
Though more objects are involved, such approach has two main advantages. 
\begin{itemize}
    \item All footballs, which are not represented by the data sets, are now included.
    \item Connectedness can be characterized using an intrinsic geometric quantity, thereby avoiding the influence of different parameterization. 
\end{itemize}

\bigskip
Let $\GH$ be the space of all compact metric space, endowed with the \textbf{Gromov-Hausdorff distance} $d_{\rGH}$ (G-H in short). 
In general, the G-H distance between two metric space is difficult to compute directly. We refer to \cite[\S 7.3]{BurIva01} for an expository literature. Here is a convenient method to estimate G-H distance. 

\begin{definition} \label{def:correspondence}
    A \textbf{correspondence} between two metric spaces $A, B$ is a set $\Rf \subset A \times B$ such that $\forall\ a \in A, \exists\ b \in B: (a, b) \in \Rf$ and $\forall\ b \in B, \exists\ a \in A: (a, b) \in \Rf$. 

    The \textbf{distortion} of $\Rf$ is defined as 
    \begin{equation}
        \dis \Rf := \sup \left\{\ |d_A(x, y) - d_B(x', y')| \ \big|\ (x, x'), (y, y') \in \Rf \ \right\} \ .
    \end{equation}
    A homeomorphism $f:A\to B$ induces a correspondence $\Rf_f := \defset{(a,f(a))}{a \in A}$. The distortion of $\Rf_f$ will be denoted as $\dis f$ in short.
\end{definition}

\begin{fact}
\label{fac:correspondence_estimate}
    For any two metric spaces $A, B$, 
    \begin{equation}
        d_{\rGH}(A, B) = \frac12 \inf_{\Rf} \dis \Rf \ .
    \end{equation}
    That is to say, if $\Rf$ is a correspondence between $A,B$, then $d_{\rGH}(A, B) \leq \frac12 \dis \Rf$. 
    \qed
\end{fact}

\begin{fact}
\label{fac:diameter_estimate}
    Let $\diam A < +\infty$ be the diameter of $A \in \GH$. Then 
    $$ d_{\rGH}(A, B) \geq \frac12 \abs{\diam A - \diam B} \ . $$
    In particular, if $\Pt$ is the single-point space in $\GH$, then $d_{\rGH}(A, \Pt) = \frac12 \diam A$.
    \qed
\end{fact}

In addition, we will use the following facts about G-H distance. 

\begin{fact}[{\cite[\S 2.1]{AllSor20}}] \label{fac:Riemann_matric_estimate}
    Let $(A, g_A),\ (B, g_B)$ be two compact Riemannian metric spaces, and $f : A \to B$ a diffeomorphism. 
    If there exist $0 < a < 1 < b$ such that 
    \[ a^2 \cdot f^* g_B \leq g_A \leq b^2 \cdot f^* g_B \ , \] 
    then for any points $x, y \in A$, $a \cdot d_B(f(x), f(y)) \leq d_A(x, y) \leq b \cdot d_B(f(x), f(y))$, and 
    \begin{equation*}
        d_{\rGH}(A, B) \leq \frac12 \dis f  \leq \frac12 \max\{1-a, b-1\} \cdot \diam(A) \ .
    \end{equation*}
    \qed
\end{fact}

To study the G-H limit of surfaces in $ \Mhcmu_{0, 1}(\alpha) $, we will add boundaries to each subset $\MM(T)$ in Definition \ref{def:M_T}. Each $\MM(T)$ is automatically parameterized as $\Rpos \times (0,1)$.
We will next investigate the G-H limits along three directions: 
\[ \lambda \to 0^+, \quad \lambda \to 1^-, \quad K \to + \infty \ . \]
The main result of this subsection (Theorem \ref{thm:map_bij_continuous}) will assert that these 3 directions are enough to construct the closure of $\MM(T)$ in $\GH$. 

\bigskip
Now we always assume that $p,q$ satisfy the conditions in Theorem \ref{thm:hcmu_data}. Then for any $T \in \Tree(\pp_{p,q})$, the subset $\mathcal{M}(T) \subset \Mhcmu_{0,1}(\alpha)$ in Definition \ref{def:M_T} is non-empty. We begin by describing the limit when $\lambda \to 0^+$ or $1^-$. 

\begin{definition}\label{def:limit_sigma}    
    We directly construct geometric objects $\Sigma(T, K_0, 0)$ when $K_0 > 0$. 
\begin{enumerate}
    \item For each white vertex $v^- \in V^-$ of $T$, construct a football $S_{p / q}(K_0, \frac{2q-p}{2p-q} K_0)$ whose bottom angle is $2\pi$ and ratio is $q/p$. 
    \item Define $\Sigma(T, K_0, 0)$ to be the one-point union of the above $q$ footballs along their top vertex.  
\end{enumerate}
    For $\lambda=1$, apply the same procedure for each black vertex, and swap the roles between bottom angle or vertex and top ones. 
    This yields $\Sigma(T, K_0, 1)$. 

    Note that when $p=\alpha,q=1$ and $T_0$ is the unique tree in $\Tree(\pp_{\alpha,1})$, the constructed object $\Sigma(T_0, K_0, 0)$ is just the football $S_\alpha(K_0, \frac{2-\alpha}{2\alpha-1}K_0)$. Such $T_0$ will be called the \textbf{star-shape tree}.
    In other cases, $\Sigma(T, K_0, 0)$ or $\Sigma(T, K_0, 1)$ is merely a CW-complex, not a topological sphere. 
\end{definition}	

The geometric object defined above is also the result of gluing bigons in the same manner when $\lambda \in (0,1)$. 
In Example \ref{eg:hcmu_31} and Figure \ref{fig:gluing_bigons}, when $\lambda=0$, all marked boundary points coincide with top vertices. So the bigons around a common minimal point merge to a football of bottom angle $2\pi$. 
With this geometric intuition in mind, we will show that these objects are exactly the G-H limits of $\Sigma(T, K_0, \lambda)$ as $\lambda \to 0$ or $\lambda \to 1$.

In view of the asymptotic behavior in Lemma \ref{lem:line_element_estimate}, the single-point space $\Pt$ is also taken into consideration. 
\begin{definition}\label{def:M_T_star}
    Fix a subset $\MM(T) \subset \Mhcmu_{0,1}(\alpha)$. 
	For all $\lambda\in[0,1]$, define $\Sigma(T, +\infty, \lambda)$ to be the single-point space $\Pt \in \mathcal{GH}$. 
	Together with Definition \ref{def:limit_sigma}, set  
	\[ \MM(T)^* := \{\ \Sigma(T, K_0, \lambda) \ |\  0 < K_0 \leq +\infty,\ 0 \leq \lambda \leq 1 \ \} \subset \mathcal{GH}\]
    and 
    \[     \varTheta := (0, +\infty] \times[0,1] \ \big{/} \sim \quad     \]
    where $(+\infty,\lambda) \sim (+\infty,\lambda')$ for any $\lambda,\lambda' \in [0,1]$. $\varTheta$ is homeomorphic to an Euclidean triangle deleting a closed edge. 
\end{definition}

Although many technical details remain to be discussed, the earlier geometric intuition suggests that $\MM(T)^*$ is indeed parameterized by $\varTheta$, therefore containing all G-H limits of HCMU spheres. In this subsection, we first prove that $K_0, \lambda$ are continuous parameters indeed.  

\begin{theorem} \label{thm:map_bij_continuous}
	Fix a subset $\MM(T) \subset \Mhcmu_{0,1}(\alpha)$ as Definition \ref{def:M_T}. 
    The mapping $(K_0,\lambda) \mapsto \Sigma(T,K_0,\lambda)$ induces a continuous bijection 
    \begin{equation}\label{eq:MT_para}
    F : \varTheta \longrightarrow \MM(T)^* \ . 
    \end{equation}
	In other words, we have the following two cases.
    \begin{enumerate}
    \item 
    If $K_n \to K_0 \in \Rpos$ and $\lambda_n \to \lambda \in [0, 1]$, then $\Sigma(T, K_n, \lambda_n) \xrightarrow{\rGH} \Sigma(T, K_0, \lambda)$.

    \item  When $K_n \to +\infty$, the G-H limit of $\Sigma(T, K_n, \lambda)$ is the single-point space $\Pt$.
\end{enumerate}
\end{theorem}

Now we give a proof skeleton of the theorem. 

First, with the rectangular coordinate system $(v, \phi)$ in Definition \ref{defn:hcmu_bigon}, a canonical homeomorphism between any two bigons can be easily constructed. This helps us to show that \(B_\alpha(K_0, K_1)\) depends continuously on \(K_0, K_1, \alpha\) (Proposition \ref{prop:bigon_continuous}).

Secondly, since all surfaces in a fixed $\MM(T)$ are formed by gluing bigons in the same manner, a correspondence \(\Rf\) can be induced by the canonical homeomorphism between bigons (Definition \ref{def:correspondence_from_bigon}). Such $\Rf$ is usable even when $\lambda = 0,1$.  
To estimate its distortion, we will study the shortest path between any pair of points (Fact \ref{fac:min_piecewise}). By dividing a path into segments, each of whom lies in a single bigon (Lemma \ref{lem:piecewise_segmentation}), we can have a uniform estimate on the distortion of such path. 
This leads to the continuity of $\varTheta$ for finite $K_0$ (Proposition \ref{prop:Sigma_continuous}). 

Finally, the continuity when \(K_0 \to +\infty\) is a direct consequence of the limit behavior of diameter, and the G-H limit is the single-point space \(\Pt\) (Proposition \ref{prop:Kinfty_estimate}).

\bigskip

The project begins with the canonical homeomorphism between bigons. 
To simplify notations, in the below we always denote two involved bigons as
\begin{equation}\label{eq:two_bigons}
    B := B_{\alpha}(K_0, K_1), \quad B' := B_{\alpha'}(K_0', K_1')
\end{equation}
with $\alpha,\alpha'>0$ and $(K_0, K_1),\ (K_0', K_1')$ satisfying \eqref{eq:K0K1_extd}. The lengths of their character line elements are denoted by $\ell = \ell (K_0, K_1)$ and $\ell' := \ell(K_0', K_1')$. 
The bigon $B$ is usually given beforehand. 

\begin{definition} \label{def:homeo_bigon}
    For any two bigons as \eqref{eq:two_bigons}, there is a \textbf{canonical homeomorphism} induced by the rectangular coordinates defined as
    \begin{equation}
        f : B_\alpha(K_0, K_1) \to B_{\alpha'}(K_0', K_1') \ ,\ 
        (v,\ \phi) \mapsto  \left( \frac{\ell(K_0', K_1')}{\ell(K_0, K_1)} \cdot v ,\  \frac{\alpha'}{\alpha} \cdot \phi \right)
    \end{equation}
\end{definition}

To show the continuity, we consider two cases: varying only the parameters $K_0, K_1$, and varying only the parameter $\alpha$. 

\begin{lemma} \label{lem:bigon_estimate_K}
    Given bigon $B$ in \eqref{eq:two_bigons}. For any $\varepsilon > 0$, if $B'$ has the same top angle $2\pi\alpha$ as $B$, then there exists $\delta > 0$ such that whenever $|K_0 - K_0'|, |K_1 - K_1'| \leq \delta$,
    \[ d_{\rGH}\bigl( B,\, B' \bigr) 
    \leq  \frac12 \dis f  \leq \varepsilon \ .\] 
    Here $f$ is the canonical homeomorphism in Definition \ref{def:correspondence_from_bigon}. 
\end{lemma}

\begin{proof}

Let $(v, \phi)$ and $(w, \psi)$ be the rectangular coordinates of $B$ and $B'$ respectively. 
Denote by $K(v)$ and $K'(w)$ the curvature functions along the character line elements of $B$ and $B'$, and by $h(v)$ and $h'(w)$ their warped functions.
To avoid ambiguity, the $n$-th derivatives of curvature functions are temperately written as $K^{(n)}(v),\ K'^{(n)}(w)$. 

Then the metrics $g, g'$ on $B, B'$, and the pullback metric $f^* g'$ are given by
\begin{equation*}
    g = \rmd v^2 + h^2(v) \rmd \phi^2 \ ,\quad 
    g' = \rmd w^2 + h'^2(w) \rmd \psi^2 \ ,\quad
    f^* g' = \left(\frac{\ell'}{\ell}\right)^2 \rmd v^2 + h'^2\left(\frac{\ell'}{\ell} \cdot v\right) \rmd \phi^2 \ .
\end{equation*}
Take $\eta := \frac{2 \varepsilon}{\diam(B)}$. By Fact \ref{fac:Riemann_matric_estimate}, it remains to prove that we can have  
\begin{equation*}
    \abs{\frac{\ell'}{\ell} - 1} \leq \eta \quad \text{ and } \quad
    \abs{\frac{h'(\frac{\ell'}{\ell}v)}{h(v)} - 1} \leq \eta \quad (\ \forall\ v \in [0, \ell] \ )
\end{equation*}
when $(K_0', K_1')$ are sufficiently close to $(K_0, K_1)$.  

\medskip
First, by Lemma \ref{lem:line_element_estimate}, the length $\ell$ of character line element depends continuously on the parameters $K_0, K_1$. 
This directly yields $\abs{\ell' \,/\, \ell - 1} \leq \eta$ when $(K_0', K_1') \to (K_0, K_1)$. 

\medskip

Next, for any fixed pair $(K_0', K_1')$, consider the function 
\begin{align*}
   H(v) &:= \frac{ h'( \frac{\ell'}{\ell}v )} { h(v) }
= \frac{\ol{c}(K_0, K_1)}{\ol{c}(K_0', K_1')} \cdot \frac{K'^{(1)}( \frac{\ell'}{\ell}v )}{K^{(1)}(v)} \\
&= \frac{(K_0 - K_1) (2K_0 + K_1)}{(K_0' - K_1') (2K_0' + K_1')} \cdot \sqrt{\frac{(K_0' - K'(w)) (K'(w) - K_1') (K'(w) + K_0' + K_1')}{(K_0 - K(v)) (K(v) - K_1) (K(v) + K_0 + K_1)}} \ .
\end{align*}
Here $w = \frac{\ell'}{\ell}v$ for brevity. 
Since $K_1 < K(v) < K_0$ and $K_1' < K'(w) < K_0'$ when $0 < v < \ell$, 
$H(v)$ is well defined on the open interval $(0, \ell)$.

To extend $H(v)$ to the endpoints, we shall study the limit as $v \to 0^+, \ell^- $. 
The second derivative of the curvature function will be used later: 
\begin{equation*}
    K^{(2)}(v) = - \frac16 \big[ (K - K_0) (K - K_1) + (K - K_0) (K + K_0 + K_1) + (K - K_1) (K + K_0 + K_1) \big] \ .
\end{equation*}
This is obtained by differentiating both sides of equation \eqref{eq:curv_along_v}. Also see \cite[\S 2.2]{CCW05}. 
In particular, 
\begin{equation*}
    K^{(2)}(0) = - \frac16 (K_0 - K_1) (2K_0 + K_1) = \ol{c} \neq 0\ ,\quad 
    K^{(2)}(\ell) = \frac16 (K_0 - K_1) (K_0 + 2K_1) \neq 0\ .
\end{equation*}
Then by L'Hospital's Rule, 
\begin{align*}
   \lim_{v \to 0^+} H(v) 
&= \frac{\ol{c}(K_0, K_1)}{\ol{c}(K_0', K_1')} \cdot \lim_{v \to 0^+} \frac{K'^{(1)}(w)}{K^{(1)}(v)} \\
&= \frac{\ol{c}(K_0, K_1)}{\ol{c}(K_0', K_1')} \cdot \lim_{v \to 0^+} \frac{(\ell' \,/\, \ell) \cdot K'^{(2)}(w)}{K^{(2)}(v)} 
= \frac{\ol{c}(K_0, K_1)}{\ol{c}(K_0', K_1')} \cdot \frac{\ell'}{\ell} \cdot \frac{K'^{(2)}(0)}{K^{(2)}(0)} 
= \frac{\ell'}{\ell} \ .
\end{align*}
Similarly, 
\begin{equation*}
    \lim_{v \to \ell^-} H(v)   
    = \frac{\ol{c}(K_0, K_1)}{\ol{c}(K_0', K_1')} \cdot \frac{\ell'}{\ell} \cdot \frac{K'^{(2)}(\ell')}{K^{(2)}(\ell)}
    = \frac{(2K_0 + K_1) (K_0' + 2K_1') \cdot \ell'}{(2K_0' + K_1') (K_0 + 2 K_1) \cdot \ell} \ .
\end{equation*}
Thus, $H(v)$ is a well-defined function on the closed inerval $[0, \ell]$.

\medskip
The curvature function $K'(w)$ satisfies the differential equation \eqref{eq:curv_along_v}, and the equation depends continuously on the parameters $K_0', K_1'$. 
Consequently, $K'(w)$ is a continuous function of $K_0', K_1'$. 
So $H(v) = H(v; K_0', K_1')$ is continuous with respect to $v, K_0', K_1'$. 
When $K_0' = K_0, K_1' = K_1$, we have $H(v; K_0, K_1) \equiv 1$. 
Hence, for any $\eta > 0$, there exists $\delta > 0$ such that whenever $|K_0 - K_0'|, |K_1 - K_1'| \leq \delta$, the estimate holds:
\begin{equation*}
    \abs{H(v; K_0', K_1') - H(v; K_0, K_1)} = \abs{\frac{h'(\frac{\ell'}{\ell}v)}{h(v)} - 1} \leq \eta, \quad \forall\ v \in [0, \ell] \ .
    \tag*{\qedhere}
\end{equation*}
\end{proof}

\begin{lemma} \label{lem:bigon_estimate_alpha}
    Given bigon $B$ in \eqref{eq:two_bigons}. For any $\varepsilon > 0$, if $B'$ has the same extremal curvature $K_0, K_1$ as $B$, then there exists $\delta > 0$ such that whenever $\abs{\alpha - \alpha'} \leq \delta$, $d_{\rGH}\big( B,\, B' \big) \leq \varepsilon$.  
\end{lemma}

\begin{proof}

Without loss of generality, assume $\alpha < \alpha'$.  
Using the rectangular parametrization, 
we can isometrically embed $B$  into $B'$:  
\begin{equation*}  
    i : B \hookrightarrow B',  \quad (v, \phi) \mapsto (v, \phi).  
\end{equation*}  
Note that this is different from the canonical homeomorphism $f : (v, \phi) \mapsto (v, \frac{\alpha'}{\alpha}\phi)$.

Now take two points $x = (v_1, \phi_1),\ y = (v_2, \phi_2)$ in $B$.  
The embedding $i$ maps them to points with the same coordinates in $B'$. 
Besides, $f(x) = (v_1, \frac{\alpha'}{\alpha} \phi_1)$ and $f(y) = (v_2, \frac{\alpha'}{\alpha} \phi_2)$. The distance between $(v,\phi)$ and $(v, \frac{\alpha'}{\alpha} \phi)$ in $B'$ is at most $2\pi(1-\frac{\alpha}{\alpha'})$ times the length of waist line. 
Then by the triangle inequality,  
\begin{equation*}  
    \abs{d_B(x, y) - d_{B'} \bigg( f(x), f(y) \bigg)} \leq d_{B'} {\bigg(} i(x), f(x) {\bigg)} + d_{B'} {\bigg(}i(y), f(y){\bigg)} \leq 2\pi(\alpha' - \alpha)(h(v_1) + h(v_2)).  
\end{equation*}  
Taking $\delta = \varepsilon / (2\pi \max_{v}\{h(v)\})$, we see $\abs{d_B(x, y) - d_{B'}(f(x), f(y))} \leq 2\varepsilon$ for any $x,y\in B$, whenever $\alpha' - \alpha \leq \delta$.  
Hence $d_{\rGH}(B, B') \leq \frac12 \dis f \leq \varepsilon$.
\end{proof}

\begin{proposition} \label{prop:bigon_continuous}
    In the sense of G-H topology, the bigon \(B_\alpha(K_0, K_1)\) depends continuously on the parameters \(K_0, K_1, \alpha\). 
    That is to say, if 
	\[  \lim_{n\to\infty} \alpha_n = \alpha \in \Rpos, \quad 
    \lim_{n\to\infty} K_0(n) = K_0 \in \Rpos, \quad
    \lim_{n\to\infty} K_1(n) = K_1 \in \Rpos, 
    \]
	then $B_{\alpha(n)}({K_0}(n),{K_1}(n)) \xrightarrow{\rGH} B_{\alpha}(K_0,K_1)$. 
    The same holds for footballs. 
\end{proposition}

\begin{proof}

By Lemma~\ref{lem:bigon_estimate_K} and Lemma~\ref{lem:bigon_estimate_alpha}, we can choose $\delta > 0$, such that when 
$|K_0 - K_0'|, |K_1 - K_1'|, |\alpha - \alpha'| \leq \delta$,
\begin{equation*}\begin{split}
    &\ d_{\rGH}\big(B_\alpha(K_0, K_1),\, B_{\alpha'}(K_0', K_1')\big) \\
    \leq& \
    d_{\rGH}\big(B_\alpha(K_0, K_1),\, B_{\alpha'}(K_0, K_1)\big) + d_{\rGH}\big(B_{\alpha'}(K_0, K_1),\, B_{\alpha'}(K_0', K_1')\big) \leq 2 \cdot \frac\varepsilon{2} = \varepsilon \ .
\end{split}
\tag*{\qedhere}
\end{equation*}
\end{proof}

\bigskip
Secondly, we show that in each $\MM(T)^*$, $\Sigma(T, K_0, \lambda)$ depends continuously on the parameters \(K_0 \in \Rpos\) and \(\lambda \in [0, 1]\). 
We need to extend canonical homeomorphisms between bigons to a correspondence between any two metric spaces in $\MM(T)^*$. 

Recall that in Definition \ref{def:critical_meridian}, the critical graph consists of all gluing image of boundaries of bigons. 
Since $\Sigma(T, K_0, 0), \Sigma(T, K_0, 1) \in \MM(T)^*$ can also be obtained by gluing bigons with the same manner, the critical graphs on such geometric objects are also defined.  

\begin{definition} \label{def:correspondence_from_bigon}
    Fix a space $\MM(T)^*$. 
    For any two metric spaces $\Sigma := \Sigma(T, K_0, \lambda),\ \Sigma' := \Sigma(T, K_0', \lambda')$ in $\MM(T)^*$, there is a \textbf{canonical correspondence} $\Rf$ defined as below. 
    \begin{enumerate}
    \item Each edge $e$ of the tree $T$ corresponds to a bigon $B_e$ in $\Sigma$ and a bigon $B_e'$ in $\Sigma'$. Let $f_e : B_e \to B_e'$ be the canonical homeomorphism in Definition \ref{def:homeo_bigon}.

    \item Each homeomorphism $f_e : B_e \to B_e'$ yields a set $\Rf_e := \{(x, y) \mid f_e(x) = y\} \subset \Sigma \times \Sigma'$.

    \item The canonical correspondence is defined as the union of above sets over all edges 
    $$ \Rf := \bigsqcup_{e \in E} \Rf_e \ . $$
    \end{enumerate}
\end{definition}

It is easy to verify that $\Rf$ is a correspondence, symmetric about the two metric spaces. 
Moreover, $\Rf$ induces a bijection between the complementary of critical graphs. 
To estimate $\dis \Rf$, we have to investigate the correspondence on critical graphs. 

Let $\Sigma := \Sigma(T, K_0, \lambda),\ \Sigma' := \Sigma(T, K_0', \lambda')$ be two fixed metric spaces in $\MM(T)^*$, and let $\Rf$ be the canonical correspondence between $\Sigma$ and $\Sigma'$. 
For simplicity, let $\ell(K_0)$ be the length of character line element of extremal curvature $K_0,\ K_1:=\frac{2q-p}{2p-q}K_0$.

\begin{lemma} \label{lem:correspondence_difference}
    If $x$ lies on the critical graph of $\Sigma$, then there exists at most two different points $x', x'' \in \Sigma'$ such that $(x, x'), (x, x'') \in \Rf$ and $x', x''$ also lie on the critical graph of $\Sigma'$. 
    Besides, $d_{\Sigma'}(x', x'') \leq 2 |\lambda - \lambda'| \cdot \ell(K_0')$.    
\end{lemma}

\begin{proof}

Suppose $x$ lies on a critical meridian $m$ connecting a maximal point $a$ and a saddle point $z$, with the distance from $x$ to $a$ along $m$ being $v_x$.
Let $B_{e_1}, B_{e_2}$ be two bigons on two sides of $m$, possibly the same. 
Assume the rectangular coordinates of $x$ in $B_{e_1}$ and $B_{e_2}$ are $(v_x, \alpha), (v_x,0)$ respectively. 
Then $f_{e_1}(x)$ has coordinates $(v_x\cdot{\ell(K_0')} / {\ell(K_0)} , \alpha')$ in $B_{e_1}'$, and $f_{e_2}(x)$ has coordinates $(v_x\cdot{\ell(K_0')}/{\ell(K_0)}, 0)$ in $B_{e_2}'$.
Here $2\pi\alpha$ and $2\pi\alpha'$ are the top angles of $B_{e_1}$ and $B_{e_1}'$. 

Similarly, the saddle point $z$ has coordinates $(\lambda \cdot \ell(K_0), \alpha)$ in $B_{e_1}$. 
And $z' := f_{e_1} (z)$ has coordinates $(\lambda \cdot \ell(K_0'), \alpha')$ in $B_{e_1}' $. 
However, the saddle point $z^* \in \Sigma'$ has coordinates $(\lambda' \cdot \ell(K_0'), \alpha')$ in $B_{e_1}'$. So $z^*$ may not coincide with the image point $z' = f_{e_1} (z)$. 
Let $a' \in \Sigma'$ be the corresponding maximal point to $a$, and $m'$ be the critical meridian connecting $a'$ to $z^*$ on the boundary of $B_{e_1}'$ and $B_{e_2}'$. 

When $\lambda \leq \lambda'$, the image point $z'$ lies on the critical meridian $m'$, hence $f_{e_1}(m) \subset m'$. Similarly, $f_{e_2}(m) \subset m'$. 
Thus $f_{e_1}(x), f_{e_2}(x)$ glue to one point $x'$ on $m'$, and $(x,x') \in \Rf$. 

When $\lambda > \lambda'$, $\lambda \cdot \ell(K_0')$ exceeds the length of $m'$. So $f_{e_1}(x), f_{e_2}(x)$ glue to two different points $x', x''$ in the critical graph of $\Sigma'$, outside $m'$, i.e. $\lambda' \cdot \ell(K_0')$.  
But the distance from $x'$ or $x''$ to $z^*$ along $m'$ is 
$\lambda \cdot \ell(K_0') - v_x\cdot{\ell(K_0')} / {\ell(K_0)} \leq (\lambda - \lambda') \cdot \ell(K_0) $.  
Hence $d_{\Sigma'}(x', x'') \leq 2 |\lambda - \lambda'| \ell(K_0')$.
See Figure \ref{fig:shortest_path} later as an illustration, especially the point $x_1$. 

The same conclusion holds when the status of $\Sigma$ and $\Sigma'$ are exchanged. 
\end{proof}

\bigskip
Now we study the shortest path between any two points on $\Sigma \in \MM(T)^*$. 
Since $\Sigma$ is either a compact cone surface or a one-point union of compact cone surfaces, it is locally compact and metrically complete. 
According to \cite[\S 2.5.2]{BurIva01}, the following fact holds: 

\begin{fact} \label{fac:min_piecewise}
    For any two points $x$ and $y$ on $\Sigma$, there always exists a shortest path $\gamma$ realizing the distance $L(\gamma) = d(x, y)$. 
    Moreover, $\gamma$ is smooth and geodesic away from the cone point (when $\lambda \in (0, 1)$) or the unique gluing point (when $\lambda = 0, 1$). 
    \qed
\end{fact}

\begin{lemma} \label{lem:inter_point_les_1}
    Let $\gamma$ be a shortest path connecting two points in $\Sigma$, which intersects a critical meridian $m$. Then the intersection $\gamma \cap m$ is either a point, or a maximal segment of $m$ whose endpoints can only be endpoints of $\gamma$ or of $m$.  
\end{lemma}

\begin{proof}

Suppose $\gamma$ intersects $m$ at two points $x \neq y \in \Sigma$. Since $m$ is already a minimal smooth geodesic, the segment of $m$ between $x$ and $y$ is also a minimal smooth geodesic. Replacing the arc of $\gamma$ between $x$ and $y$ by this segment must also yields a shortest path $\gamma'$.

As a shortest path, $\gamma'$ can not bend except at saddle points by Fact \ref{fac:min_piecewise}. At least one of $x,y$ is a smooth point of the metric, so $\gamma'$ must be smooth at $x$ or $y$.  
Consequently, near that point, both $\gamma$ and $\gamma'$ follow $m$, which is already a geodesic. So $\gamma$ must contain the segment between $x$ and $y$, i.e., $\gamma = \gamma'$.

Since $\gamma$ cannot bend except at saddle points, the intersection of $\gamma$ with $m$ must be a maximal segment of $m$, hence its endpoints can only be endpoints of $\gamma$ or of $m$.
\end{proof}

\begin{lemma} \label{lem:piecewise_segmentation}
    Let $\gamma$ be a shortest path connecting $x \neq y \in \Sigma$. Then one can choose
    \[ x = x_0, x_1, \cdots, x_k, x_{k+1}=y \] 
    along $\gamma$, such that \\ 
    \textbullet\quad for $ 1 \leq i \leq k$, $x_i$ lies on the critical graph;\\
    \textbullet\quad for $ 1 \leq i \leq k+1$, the segment from $x_{i-1}$ to $x_i$ is \\
    \indent (a). either a geodesic segment connecting two different boundaries of one bigon \\
    \indent (b). or a maximal segment of a critical meridian. \\
    See the left of Figure \ref{fig:shortest_path} as an example. 
    
    Moreover, the number of points is bounded by $k \leq 2\alpha$, where $2 \pi \alpha$ is the cone angle.
\end{lemma}

\begin{proof}

Since $\Sigma$ has a unique cone point, by Fact \ref{fac:min_piecewise}, $\gamma$ passes the cone point at most once. In such case, we choose the cone point as one of $x_i$'s and divide $\gamma$ into two smooth geodesics. Each of them is a shortest path from the cone point. 
Therefore, we may assume $\gamma$ is smooth and avoids the cone point in its interior.  

Let $\gamma(t) : [0, 1] \to \Sigma$ be a parametrization of $\gamma$. 
Define the division points recursively by $x_i := \gamma(t_i)\ (i=1,\cdots,k)$, so that each segment lies in a single bigon (including its boundary): 
\begin{equation*}
    t_i := \sup \defset{ t \in (t_{i-1}, 1]}{  \gamma(u) \text{ lies in the same bigon for all } u \in [t_{i-1}, t] }.
\end{equation*}

By construction, division points can only lie on the boundaries of bigons, i.e., on critical meridians. 
Moreover, by Lemma \ref{lem:inter_point_les_1}, for a critical meridian $m$ intersecting $\gamma$, $m \cap \gamma$ is either one point or a maximal segment of $m$. 
In both case, there is at most one division point on $m$.
Hence the number of division points is at most the number of critical meridians. 

Since the whole boundary of each bigon is divided into 4 critical meridians by the extremal and saddle points, and each meridian is adjacent to 2 bigons, the number of edges in the critical graph is twice the number of bigons. 
Finally, the number of bigons equals to the number of edges in the LWBP-tree $T$, which is exactly $p+q-1=\alpha$.   
\end{proof}

\begin{proposition} \label{prop:Sigma_continuous}
    Fix a set $\MM(T)^*$ as Definition \ref{def:M_T_star}. Given $K_0 \in \Rpos$ and $\lambda \in [0,1]$, for every $\varepsilon > 0$, there exists $\delta > 0$ such that whenever  
    $|K_0 - K_0'|, |\lambda - \lambda'| \leq \delta$, 
    $$ d_{\rGH} \bigg( \Sigma(T,K_0, \lambda),\, \Sigma(T, K_0', \lambda') \bigg) \leq \frac12 \dis \Rf \leq \varepsilon \ ,$$
    where $\Rf$ is the canonical correspondence in Definition \ref{def:correspondence_from_bigon}. 

    In other words, the map $F$ in \eqref{eq:MT_para} is continuous on $\Rpos \times [0,1] \subset \varTheta$. 
\end{proposition}

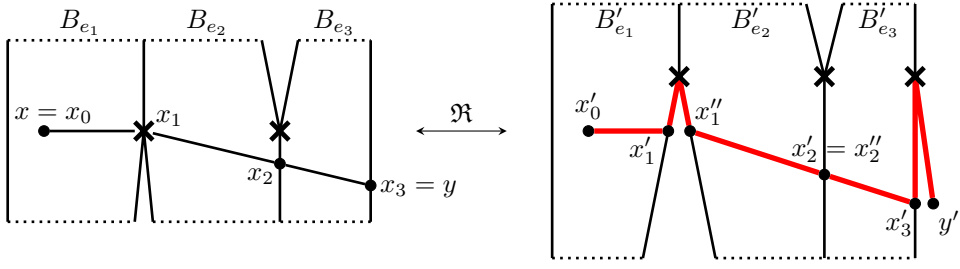
\begin{figure}
\centering
\begin{tikzpicture}[scale=1.2]  
    \begin{scope}[ ]
    \draw[line width = 1pt] (-2,1) to (-2,-1);
    \draw[line width = 1pt] (-1/2,1) to (-1/2,0);
    \draw[line width = 1pt] (-3/5,-1) to (-1/2,0) to (-2/5,-1);
    \draw[line width = 1pt] (1,-1) to (1,0);
    \draw[line width = 1pt] (4/5,1) to (1,0) to (6/5,1);
    \draw[line width = 1pt] (2,1) to (2,-1);
    \draw[dotted, line width = 1pt] (-2,1) to (4/5,1);
    \draw[dotted, line width = 1pt] (6/5,1) to (2,1);
    \draw[dotted, line width = 1pt] (-2,-1) to (-3/5,-1);
    \draw[dotted, line width = 1pt] (-2/5,-1) to (2,-1);
    \node[draw, circle, fill=black, inner sep=0pt, minimum size=4pt](x) at (-8/5,0) {};
    \node at ($(x)+(+3pt,+5pt)$) {$x = x_0$};
    \node[draw, circle, fill=black, inner sep=0pt, minimum size=4pt](y) at (2,-3/5) {};
    \node[crossmark](x1) at (-1/2,0) {};
    \node[crossmark]() at (1,0) {};
    \node[draw, circle, fill=black, inner sep=0pt, minimum size=4pt](x2) at (1,-0.36) {};
    \draw[line width = 1pt] (x) to (x1) to (x2) to (y) {}; 
    \node at ($(y)+(+15pt,-1pt)$) {$x_3 = y$};
    \node at ($(x1)+(+8pt,+4pt)$) {$x_1$};
    \node at ($(x2)+(-6pt,-4pt)$) {$x_2$};
    \node at (-6/5,6/5) {$B_{e_1}$};
    \node at (1/5,6/5) {$B_{e_2}$};
    \node at (8/5,6/5) {$B_{e_3}$};
    \end{scope}

    \draw[stealth-stealth] (2.5,0) -- (3.5,0) node[midway, above] {$\mathfrak{R}$};

    \begin{scope}[shift = {(6,0)}]
    \draw[line width = 1pt] (-2,1.4) to (-2,-1.4);
    \draw[line width = 1pt] (-0.6,1.4) to (-0.6,0.6);
    \draw[line width = 1pt] (-1,-1.4) to (-0.6,0.6) to (-0.2,-1.4);
    \draw[line width = 1pt] (1,-1.4) to (1,0.6);
    \draw[line width = 1pt] (0.8,1.4) to (1,0.6) to (1.2,1.4);
    \draw[line width = 1pt] (2,1.4) to (2,-1.4);
    \draw[dotted, line width = 1pt] (-2,1.4) to (0.8,1.4);
    \draw[dotted, line width = 1pt] (1.2,1.4) to (2.0,1.4);
    \draw[dotted, line width = 1pt] (-2,-1.4) to (-1,-1.4);
    \draw[dotted, line width = 1pt] (-0.2,-1.4) to (2,-1.4);
    \node[draw, circle, fill=black, inner sep=0pt, minimum size=4pt](x') at (-8/5,0) {};
    \node[draw, circle, fill=black, inner sep=0pt, minimum size=4pt](x1') at (-0.72,0) {};
    \node[draw, circle, fill=black, inner sep=0pt, minimum size=4pt](x1'') at (-0.48,0) {};
    \node[draw, circle, fill=black, inner sep=0pt, minimum size=4pt](x2') at (1,-0.48) {};
    \node[draw, circle, fill=black, inner sep=0pt, minimum size=4pt](x3') at (2,-0.8) {};
    \node[draw, circle, fill=black, inner sep=0pt, minimum size=4pt](y') at (2.2,-0.8) {};
    \draw[line width = 2pt, red, line join=round] (x') to (x1') to (-0.6,0.6) to (x1'') to (x2') to (x3') to (2,0.6) to (y') {}; 
    \node[crossmark] at (-0.6,0.6) {};
    \node[crossmark] at (1,0.6) {};
    \node[crossmark] at (2,0.6) {};
    \node at ($(x')+(0pt,+8pt)$) {$x_0'$};
    \node at ($(x1')+(-8pt,-6pt)$) {$x_1'$};
    \node at ($(x1'')+(+6pt,+6pt)$) {$x_1''$};
    \node at ($(x2')+(+4pt,+8pt)$) {$x_2' = x_2''$};
    \node at ($(x3')+(-5pt,-6pt)$) {$x_3'$};
    \node at ($(y')+(+5pt,-6pt)$) {$y'$};
    \node at (-1.3,6/5) {$B'_{e_1}$};
    \node at (1/5,6/5) {$B'_{e_2}$};
    \node at (8/5,6/5) {$B'_{e_3}$};
    \end{scope}

\end{tikzpicture}    
\caption{\small The canonical correspondence between three successive bigons. 
(1) The broken line $x \to x_1 \to y$ is a shortest path from $x$ to $y$. And $x_1, x_2, y$ lie on the critical graph. 
(2) The canonical homeomorphisms from $B_{e_1}$ and $B_{e_2}$ map $x_1$ to two different points: $f_{e_1}(x_1) = x_1' \neq x_1'' = f_{e_2}(x_1)$. 
(3) The bold red path on right is used for estimation in Proposition \ref{prop:Sigma_continuous}. }
\label{fig:shortest_path}
\end{figure}

\begin{proof}

Take $x \neq y \in \Sigma$, let $\gamma$ be a shortest path connecting them. 
Choose division points $x = x_0 \to x_1 \to \cdots \to x_{k+1} = y$ along $\gamma$ as Lemma \ref{lem:piecewise_segmentation}, where $k \leq 2\alpha$. 
Assume the segment $x_{i-1} \to x_i$ lies in the bigon $B_{e_i}$, and $f_{e_i} (x_{i-1}) := x_{i-1}'',\ f_{e_i} (x_{i}) = x_{i}'$ under the canonical homeomorphism, $i=1,\cdots, k+1$. 

Note that when $\MM(T)$ is fixed, top angles of the bigons are determined by the edge weights on $T$. Hence $B_{e_i}$ and $B_{e_i}'$ share the same top angle for any edge $e$ of $T$. 
By Proposition \ref{prop:bigon_continuous}, when $|K_0 - K_0'| \leq \delta$, 
$\abs{ d_{\Sigma}(x_{i-1}, x_i) - d_{\Sigma'}(x_{i-1}'', x_i')} \leq \eta$ for some $\eta>0$, $\forall i = 1,\cdots,k+1$. 

For $x_i$ on the critical graph ($i = 1,\cdots,k$), $f_{e_i} (x_{i}) = x_{i}'$ may not coincide with $f_{e_{i+1}} (x_{i}) = X_{i}''$. But by Lemma \ref{lem:correspondence_difference}, $d_{\Sigma'} (\ x_{i}',\ x_{i}'' \ ) \leq 2 |\lambda - \lambda'| \cdot \ell(K_0')$.

Let $(x,x'), (y,y') \in \Rf$ in the canonical correspondence between $\Sigma, \Sigma'$. Then the distance between $x', y'$ can be estimated as follows. ($\Sigma'$ in $d_{\Sigma'}$ is omitted)
\begin{align*}
    d(x', y')
    &\leq 2 |\lambda - \lambda'| \ell(K_0') 
    + d(x_0'', x_1') + d(x_1', x_1'')
    + d(x_1'', x_2') + d(x_2', x_2'') + \cdots \\
    & \hspace{100pt} \cdots + d(x_{k}'', x_{k+1}')
    + 2 |\lambda - \lambda'| \ell(K_0') \\
    &\leq \sum_{i=1}^{k+1} d(x_{i-1}'', x_i') 
    + (k+2) \cdot 2 \abs{\lambda - \lambda'} \ell(K_0')  \\
    &\leq \sum_{i=1}^{k+1} \bigg( d_\Sigma(x_{i-1}, x_i) + \eta \bigg) + (k+2) \cdot 2 \abs{\lambda - \lambda'} \ell(K_0') \\
    &= d_\Sigma(x, y) + (k+1) \cdot \eta + (k+2) \cdot  2\abs{\lambda - \lambda'} \ell(K_0') ) \\
    &\leq d_\Sigma(x, y) + (2\alpha+2) \cdot ( \eta + 2\abs{\lambda - \lambda'} \ell(K_0') ) \ .
\end{align*}
The two extra terms $2\abs{\lambda - \lambda'} \ell(K_0')$ in the first inequality is added because $x=x_0$ or $y=x_{k+1}$ may also lie on the critical graph and correspondence to two different points inside $\Rf$. 
See Figure \ref{fig:shortest_path} again as an illustration.

For the same reason, we have 
\begin{equation*}
    d_\Sigma(x, y) \leq d_{\Sigma'}(x', y') + (2\alpha + 2) \cdot (\eta + 2 \abs{\lambda - \lambda'} \ell(K_0) ) \ .
\end{equation*}

Thus, with $K_0$ given, by choosing $\delta$ appropriately, we obtain
\begin{equation*}
    \abs{d(x, y) - d(x', y')} \leq (2\alpha + 2) \cdot (\eta + 2 \abs{\lambda - \lambda'} \max\{\ell(K_0), \ell(K_0')\}) \leq \varepsilon \ .
\end{equation*}
\end{proof}

\bigskip
Finally, we prove that as $K_0 \to +\infty$, the space $\Sigma(T, K_0, \lambda)$ G-H converges to the single-point space $\Pt$. 

\begin{proposition} \label{prop:Kinfty_estimate}
    Fix a set $\MM(T)^*$ as Definition \ref{def:M_T_star}. 
    Then the diameter of the surface $\Sigma(T, K_0, \lambda)$ tend to $0$ when $K_0 \to +\infty$. Therefore, 
    $$ \Sigma(T, K_0, \lambda) \xrightarrow{\rGH} \Pt \quad (K_0 \to +\infty) .$$
\end{proposition}

\begin{proof}

By Lemma \ref{lem:line_element_estimate} and Lemma \ref{lem:waist_estimate}, as $K_0 \to +\infty$, the length $\ell(K_0)$ of character line element and length $L(\gamma_v)$ of waist line in each bigon $B_e$ tend to $0$. Hence $\operatorname{diam} B_e \to 0$.
Since the surface $\Sigma(T, K_0, \lambda)$ is glued from these bigons along boundaries, we have $\operatorname{diam} \Sigma(T, K_0, \lambda) \to 0$ as $K_0 \to +\infty$. 
By Fact \ref{fac:diameter_estimate}, $\mathrm{Pt}$ is the G-H limit of $\Sigma(T, K_0, \lambda)$.
\end{proof}

\begin{proof}[Proof of Theorem \ref{thm:map_bij_continuous}]

By Proposition \ref{prop:Sigma_continuous}, for $K_0 \in \Rpos$ and $\lambda \in [0, 1]$, the surface $\Sigma(T, K_0, \lambda)$ depends continuously on the parameters.
When $K_0 \to +\infty$, the surface $\Sigma(T, K_0, \lambda)$ converges to the single-point space $\Pt$ by Proposition \ref{prop:Kinfty_estimate}. So $F$ is continuous at the unique point at infinity of $\varTheta$.

From the data set representation, the restricted map $F|_{\Rpos \times (0, 1)} : \Rpos \times (0, 1) \to \MM(T)$ onto generic HCMU spheres is a bijection. 
The remaining spaces $\Sigma(T, K_0, 0)$, $\Sigma(T, K_0, 1)$, and $\Sigma(T, +\infty, \lambda) = \Pt$ are distinct from each other and are not generic HCMU spheres. 
Hence, the map $F : \varTheta \to \MM(T)^*$ remains bijective. 

\end{proof}

\subsection{Connected components of moduli space}
Now we are ready to show that $\varTheta$ is a parameterization of $\MM(T)^*$ with respect to G-H topology. 
Such result makes the introduction of this larger space reasonable, and leads to our classification of connected components. 

\begin{theorem}\label{thm:closure_and_para}
    Fix a subspace $\MM(T)$ as Definition \ref{def:M_T}, let $\MM(T)^*,\ \varTheta$ as Definition \ref{def:M_T_star}. 
    \begin{enumerate}
        \item 
        $\MM(T)^*$ is the closure of $\MM(T)$ in $\GH$. 
        \item $F : \varTheta \to \MM(T)^*$ in \eqref{eq:MT_para} is a homeomorphism. 
    \end{enumerate}
\end{theorem}

We will use the following facts about closed and proper maps. 

\begin{fact}[\cite{Pal70}] \label{fac:continuous_proper}
    Assume $f : A \to B$ is a continuous and proper map from a topological space $A$ to a metric space $B$. 
    \begin{enumerate}
        \item Then $f$ is a closed map, and the image set $f(A)$ is closed. 
        
        \item If $f$ is also injective, then $f$ is a homeomorphism onto its image $f(A)$. 
        \qed
    \end{enumerate}
\end{fact}

\begin{fact} \label{fac:sequence_escape}
    Assume $f : A \to B$ is a map between two first countable spaces. 
    Then $f$ is proper if and only if for any sequence $\{x_n\} \subset A$ that eventually leaves every compact set, the image sequence $\{f(x_n)\}$ also leaves every compact set of $B$.
    \qed
\end{fact}

Notice that both the Euclidean triangle $\varTheta$ and metric space $\GH$ are first countable. 

\begin{proof}[Proof of Theorem \ref{thm:closure_and_para}]
$F : \varTheta \to \GH$ is shown to be continuous and injective in Theorem \ref{thm:map_bij_continuous}. We will use Fact \ref{fac:sequence_escape} to show $F$ is proper. 

As an Euclidean triangle deleting a closed edge, a sequence $\{ (K_0(n), \lambda(n)) \}$ leaves every compact set in $\varTheta$ if and only if some subsequence $K_0(n_k) \to 0$. We may assume $K_0(n) \to 0$. 
Then by Lemma \ref{lem:line_element_estimate}, the length of character line element $\ell(K_0(n)) \to +\infty$. This means the distance between a maximal and a minimal point tends to infinite. Hence $\diam \Sigma(T, K_0(n), \lambda(n)) \to +\infty$. 
By the estimate of G-H distance in Fact \ref{fac:diameter_estimate}, 
$d_{\rGH}(\Sigma(T, K_0(n), \lambda(n)), \Pt) = \frac12 \diam \Sigma(T, K_0(n), \lambda(n)) \to + \infty$. So the sequence $\{ \Sigma(K_0(n), \lambda(n)) \}$ leaves every compact set in $\GH$. 

\medskip
Now $F$ is shwon to be proper. Together with Theorem \ref{thm:map_bij_continuous}, $F$ is a homeomorphism onto its image $F(\varTheta) = \MM(T)^*$, which is also closed. 
Since $\varTheta$ is the closure of $\Rpos \times (0, 1)$ in $\varTheta$, then $\MM(T)^* = F(\varTheta)$ is the closure of $\MM(T) = F(\Rpos \times (0, 1))$ in $\GH$.  
\end{proof}

Finally we are prepared to study the decomposition and enumeration of connected components of $\Mhcmu_{0, 1}(\alpha)$. 

\begin{proposition}\label{prop:disjoint_closure}
    Fix a pair $p,q$ satisfying the conditions in Theorem \ref{thm:hcmu_data}. For each LWBP-tree $T \in \Tree(\pp_{p,q})$, let 
    \[  \ol{\MM(T)} := \MM(T)^* \cap \Mhcmu_{0, 1}(\alpha)  \]
    be the set of genuine HCMU surfaces contained in $\MM(T)^*$. 
    Then $\ol{\MM(T)}$ is the closure of $\MM(T)$ in $\Mhcmu_{0, 1}(\alpha)$. 
    
    For any two different trees $T_1, T_2 \in \Tree(\pp_{p,q})$, 
    $\ol{\MM(T_1)} \cap \ol{\MM(T_2)} = \varnothing$. 
\end{proposition}

\begin{proof}

By Theorem \ref{thm:closure_and_para}, $\MM(T)^*$ is the closure of $\MM(T)$ in $\GH$.
Thus, in the subspace $\Mhcmu_{0, 1}(\alpha) \subset \GH$, $\ol{\MM(T)}$ is the closure of $\MM(T)$.

For any $T$ distinct from the star-shaped tree $T_0$, $\Sigma(T, K_0, 0)$ and $\Sigma(T, K_0, 1)$ are not topological spheres for all finite $K_0\in\Rpos$. 
Therefore $\ol{\MM(T)} = \MM(T)$.
For $T_0$, only $\Sigma(T_0, K_0, 0) = S_{\alpha}(K_0, \frac{2R-1}{2-R} K_0)$ becomes an HCMU football. Thus
\[
\ol{\MM(T_0)} = \MM(T_0) \sqcup \{S_{\alpha}(K_0, \tfrac{2R-1}{2-R} K_0)\}_{K_0 \in \Rpos}.
\]
Since HCMU footballs are not generic, and footballs only appear in $\ol{\MM(T_0)}$, it follows that for any two distinct trees $T_1, T_2 \in \Tree(\pp_{p,q})$, the closures are disjoint: 
$\ol{\MM(T_1)} \cap \ol{\MM(T_2)} = \varnothing$.
\end{proof}

\begin{theorem}\label{thm:cnntd_comp}
    Given $\alpha\in\ZZ_{>1}$, let $\mathcal{P}(\alpha)$ be the set of all pairs $(p,q)$ satisfying the conditions in Theorem \ref{thm:hcmu_data}. 
    Then the geometric moduli space of HCMU spheres with a single conical singularity of angle $2\pi\alpha$ can be split into connected components 
    \begin{equation*}
        \Mhcmu_{0, 1}(\alpha) = \bigsqcup_{(p,q)\in\mathcal{P}(\alpha)}\ \bigsqcup_{T \in \Tree(\pp_{p, q})} \ol{\MM(T)} \ .
    \end{equation*}
    
    In particular, the number of connected components is 
    \[ \sum_{(p,q)\in\mathcal{P}(\alpha)} \abs{\Tree(\pp_{p, q})} \ , \] 
    with each term $\abs{\Tree(\pp_{p, q})}$ computed previously in Section \ref{ssec:compute}.
\end{theorem}

\begin{proof}

We only need to show that each $\ol{\MM(T)}$ is a connected component.
For the star-shaped tree $T_0$, $\ol{\MM(T_0)} \cong \Rpos \times [0, 1)$, as in the proof of Proposition \ref{prop:disjoint_closure}. For any other tree $T \neq T_0$, $\ol{\MM(T)} = \MM(T) \cong \Rpos \times (0, 1)$. Hence all of them are connected. 
Since $\Mhcmu_{0, 1}(\alpha)$ is a disjoint union of finitely many closed connected sets $\ol{\MM(T)}$, each $\ol{\MM(T)}$ is a connected component.
\end{proof}

\bibliography{sn-bibliography}

\begin{appendices}

\section{The Modified Kochetkov Formula}\label{sec:modified_YYK}

In this section, we will prove Theorem \ref{thm:number_of_TR_trivial_passport}. 
The proof consists of two steps: 
\begin{enumerate}
    \item First, reduce the rooted tree enumeration for a general passport to the tree enumeration for a full passport; 
    \item Then, replace the partition of the full passport $\ppf$ in Kochetkov's formula \ref{thm:passport_simple} with the partition of the trivial passport $\ppt$.
\end{enumerate}

For a tree of full passport, every vertex is labeled differently. So it has trivial automorphism group. Then we have

\begin{lemma} \label{lem:rooted_no_isom}
    \begin{equation}
        \abs{\TR(\ppf)} = (\abs{\ppf} - 1) \abs{\Tree(\ppf)} \ .
    \end{equation}
    \qed
\end{lemma}

\begin{proposition} \label{prop:number_of_TR_from_trivialization}
    Let $\pp$ be a general passport and $\ppt = \triv(\pp)$ be its trivialization, then
    \begin{equation}
        \abs{\TR(\pp)} = \frac{(\ppt)!}{(\pp)!} \abs{\TR(\ppt)} \ .
    \end{equation}
\end{proposition}

\begin{proof}

We construct a trivialization map between trees. 
This map only changes the labeling function $\Lab$ of an LWBP-tree $T = (V, E, \wt_E, \Lab)$ to its trivialization $\Lab_T := \wt \circ \Lab$. 
That is, 
\begin{equation}
\begin{split}
    \triv \quad : \quad\qquad \TR(\pp) \qquad\qquad &\xrightarrow{\qquad\ } \qquad\qquad\qquad  \TR(\ppt) \\ 
    T = (V, E, \wt_E, \Lab) \qquad &\xmapsto{\qquad\ } \qquad T_T = (V, E, \wt_E, \Lab_T := \wt \circ \Lab)
\end{split}
\ . 
\end{equation}

$\triv : \TR(\pp) \to \TR(\ppt)$ is clearly surjective. 
We only need to show that $\big| \triv^{-1}(T_T) \big| = (\ppt)! / (\pp)!$ for each tree $T_T \in \TR(\ppt)$ is .

Let $\pp = (S, \lambda, \wt)$ and $\ppt = (W(S), \lambda_T, \wt_T)$. 
In $\pp$, suppose $W(S) = \{w^{(1)}, \ldots, w^{(k)}\}$. 
For each $i$, assume $W^{-1}({w^{(i)}}) = \{ w_1^{(i)}, \cdots, w_{m_i}^{(i)} \}$ for some $m_i \in \Zpos$. Denote $\lambda(w_j^{(i)}) = \lambda_j^{(i)}$ for short.  
Then for all $w^{(i)} \in W(S)$, we should have 
$$ \lambda^{(i)} := \lambda_T(w^{(i)}) = \sum_{j=1}^{m_i} \lambda_j^{(i)} \ . $$

A preimage $T \in \TR(\pp)$ of $T_T \in \TR(\ppt)$ has the same graph structure. 
The labeling function $\Lab : V \to S$ must satisfy $\Lab_T = \wt \circ \Lab$. 
This is equivalent to a balls-into-boxes problem. 
The total number of such labeling functions $\Lab$ is the number of solutions to this problem.

Vertices $v \in V$ correspond to balls. 
Since rooted trees have no automorphisms, vertices (balls) are all distinct. 
Define the $i$-th type of balls as $\wt_V(v) = w^{(1)}$ or $v \in \Lab_T^{-1}(w^{(i)})$. 
There are $\big| \Lab_T^{-1}(w^{(i)}) \big| = \lambda^{(i)}$ such balls.
The labels (boxes) $w_j^{(i)}$ are also all distinct. 
Each box $w_j^{(i)}$ must receive $\lambda_j^{(i)}$ balls of the $i$-th type for each $i$.

Considering the $i$-th type of balls, each box $w_j^{(i)}$ must receive $\lambda_j^{(i)}$ balls of that type. 
Thus the number of ways to place the $i$-th type of balls is

\begin{equation*}
    \frac{\lambda^{(i)} !}{\prod_{j=1}^{m_i} \lambda_j^{(i)} ! } \ .
\end{equation*}

Therefore the total number of ways is
\begin{equation*}
    \prod_{i=1}^k \frac{\lambda^{(i)} !}{\prod_{j=1}^{m_i} \lambda_j^{(i)} ! }
    = \frac{\prod_{i=1}^k \lambda^{(i)} !}{\prod_{i=1}^k \prod_{j=1}^{m_i} \lambda_j^{(i)} ! }
    = \frac{(\ppt)!}{(\pp)!} \ .
\end{equation*}
\end{proof}

\begin{definition}
    Fix a full passport $\ppf$ and its trivialization $\ppt = \triv(\ppf)$. 
    We define the trivialization of subpassports and partitions of $\ppf$. 
    \begin{enumerate}
    \item The trivialization of subpassport $\pp'$ of $\ppf$ is just $\triv(\pp')$. 

    \item The trivialization $\triv (\pttp_F)$ of partition $\pttp_F = \{\pp_1^{n_1}\  \ldots\ \pp_k^{n_k} \} \in \Ptt(\ppf)$ 
    consists of the trivialization of each subpassport, counted with multiplicity, hence a partition of $\ppt$. 
    \end{enumerate}

\end{definition}

\begin{lemma} \label{lem:partition_full_trivial}
    Fix a full passport $\ppf$ and its trivialization $\ppt = \triv(\ppf)$. 

    The trivialization of partition $\triv : \Ptt(\ppf) \to \Ptt(\ppt)$ is a surjective map, preserving the length and $X$-function:
    \begin{equation*}
        \abs{\pttp_F} = \abs{\triv(\pttp_F)} \ ,\ X(\pttp_F) = X(\triv(\pttp_F)) \ .
    \end{equation*}

    Furthermore, for any partition $\pttp_T \in \Ptt(\ppt)$, the number of its preimages is
    \begin{equation}
        \big| \triv^{-1}(\pttp_T) \big| = \frac{(\ppt)!}{(\pttp_T)!} \ .
    \end{equation}
\end{lemma}

\begin{proof}

The first two results are simple. We only prove the last formula.

Let $\ppf = (S, \one, \wt)$ and $\ppt = (\wt(S), \lambda, \wt_T)$.
Set $\wt(S) = \{w^{(1)}, w^{(2)}, \ldots, w^{(m)}\}$ be all the weights.
Write $\lambda^{(i)} := \lambda(w^{(i)})$ for short.
In $\ppf$, denote each index with weight $w^{(i)}$ by $w^{(i)}_{j}$, where $1 \leq j \leq \lambda^{(i)}$.

Now let $\pttp_T = \{\pp_1^{n_1} \cdots \pp_k^{n_k} \}$.
For each subpassport $\pp_j$, write its multiplicity function as $\lambda_j^{(i)} := \lambda_j(w^{(i)})$.
For any weight not appearing in a subpassport, define the corresponding multiplicity function as $0$.
Thus

\begin{equation*}
    \pp_j = \prod_{i=1}^m (w^{(i)})^{\lambda_j^{(i)}} \ ,\ \sum_{j=1}^k n_j \cdot \lambda_j^{(i)} = \lambda^{(i)} \ .
\end{equation*}

Now consider a partition $\pttp_F \in \triv^{-1}(\pttp_T)$ in the preimage of $\pttp_T$.
In $\pttp_F$, there will be $n_j$ subpassports $\pp_{j, x}$, $1 \leq x \leq n_j$, such that $\triv(\pp_{j, x}) = \pp_j$.

Because there are $\lambda_j^{(i)}$ indices of weight $w^{(i)}$ in $\pp_j$, 
then each $\pp_{j, x}$ also has exactly $\lambda_j^{(i)}$ indices of weight $w^{(i)}$. 
Thus we can match each preimage $\pttp_F$ with a solution to a balls-into-boxes problem: 

First, there are altogether $\sum_{i=1}^m \lambda^{(i)}$ balls with distinguishable types.
The $i$-th type has $\lambda^{(i)}$ balls, marked with distinct symbols $w^{(i)}_{j}$, $1 \leq j \leq \lambda^{(i)}$.

On the other side, there are $\sum_{j=1}^k n_j$ boxes with distinguishable types.
The $j$-th type has $n_j$ boxes, denoted by $\pp_{j, x}$, $1 \leq x \leq n_j$, all regarded as identical.
Moreover, each box of the $j$-th type must receive exactly $\lambda_j^{(i)}$ balls of the $i$-th type.

\bigskip
Since the number of preimages equals the number of solutions to this balls-into-boxes problem, we only need to find the latter.

In the original balls-into-boxes problem, boxes of the same type are regarded as identical. 
We will next regard all boxes $\pp_{j, x}$ of the same type are distinct. 
So the number of arrangements differs by a factor of $\prod_{j=1}^k n_j!$.

Fix attention on the $i$-th type of balls.
We know that each of the $n_j$ boxes of the $j$-th type must receive $\lambda_j^{(i)}$ such balls.
Therefore, the number of arrangements is given by a multinomial coefficient.

\begin{equation*}
    \frac{\lambda^{(i)} !}{\prod_{j=1}^m {\left( \lambda_j^{(i)} ! \right)^{n_j}}} \ .
\end{equation*}

Taking all types of balls into account, the total number of arrangements is the product of the above expressions, giving

\begin{equation*}
    \prod_{i=1}^k  \left( \frac{\lambda^{(i)} !}{\prod_{j=1}^m {\left( \lambda_j^{(i)} ! \right)^{n_j}}} \right) = 
    \frac{\prod_{i=1}^k \lambda^{(i)} !}{\prod_{j=1}^m \prod_{i=1}^k {\left( \lambda_j^{(i)} ! \right)^{n_j}}} = 
    \frac{(\ppt)!}{\prod_{j=1}^m [(\pp_j)!]^{n_j}}
    \ .
\end{equation*}

Finally, the number of arrangements for the original problem is

\begin{equation*}
    \frac{(\ppt)!}{\prod_{j=1}^k n_j ! \cdot \prod_{j=1}^m [(\pp_j)!]^{n_j}} 
    = \frac{(\ppt)!}{(\pttp_T)!} \ . 
    \tag*{\qedhere}
\end{equation*}
\end{proof}

To apply Kochetkov's formula (Theorem \ref{thm:passport_simple}), we can relate each trivial passport $\ppt = (S, \lambda, \wt)$ to a full passport $\ppf$, such that $\ppt = \triv(\ppf)$.
For each $s \in S$, the index set $S_F$ of $\ppf$ should contain $\lambda(s)$ different indices of the same weight $W(s)$. Hence we define   
\begin{equation}
    S_F := \defset{(s, k)}{s \in S, 1 \leq k \leq \lambda(s)} ,\quad \wt_F(s, k) := \wt(s) 
\end{equation}
as the full passport $\ppf = (S_F, \one, \wt_F)$ corresponding to $\ppt$. 

\begin{proof}[Proof of Theorem \ref{thm:number_of_TR_trivial_passport}]

Let $\ppt = \triv(\pp)$ be the trivialization and $\ppf$ is the full passport corresponding to $\ppt$. 

First, using Proposition \ref{prop:number_of_TR_from_trivialization} and Lemma \ref{lem:rooted_no_isom}, we reduce the rooted tree enumeration for a general passport to the tree enumeration for a full passport.

\begin{equation*}
      \abs{\TR(\pp)} 
    = \frac{(\ppt)!}{(\pp)!} \abs{\TR(\ppt)} 
    = \frac{(\ppt)!}{(\pp)!} \frac1{(\ppt)!} \abs{\TR(\ppf)} 
    = \frac1{(\pp)!} (\abs{\ppf} - 1) \abs{\Tree(\ppf)} \ .
\end{equation*}

Then, by Lemma \ref{lem:partition_full_trivial}, we replace the partition of the full passport $\ppf$ in Kochetkov's formula \ref{thm:passport_simple} with the partition of the trivial passport $\ppt$.

\begin{equation*}
\begin{split}
    \abs{\Tree(\ppf)} 
    &= \sum_{\pttp_F \in \Ptt(\ppf)} (-1)^{|\pttp_F|-1} (\absbig{\ppf}-1)^{|\pttp_F|-2} X(\pttp_F) \\
    &= \sum_{\pttp \in \Ptt(\ppt)} \sum_{\pttp_F \in \triv^{-1}(\pttp)} (-1)^{|\pttp_F|-1} (\absbig{\ppf}-1)^{|\pttp_F|-2} X(\pttp_F) \\
    &= \sum_{\pttp \in \Ptt(\ppt)} \big| \triv_p^{-1}(\pttp) \big| (-1)^{|\pttp|-1} (\absbig{\ppf}-1)^{|\pttp|-2} X(\pttp) \\
    &= (\ppt)! \sum_{\pttp \in \Ptt(\ppt)} (-1)^{\absbig{\pttp} - 1} (\absbig{\ppt} - 1)^{\absbig{\pttp} - 2} \frac{X(\pttp)}{(\pttp) !} \ .
\end{split}
\end{equation*}

Combining the above two equations, we obtain
\begin{equation*}
    \abs{\TR(\pp)} = \frac1{(\pp)!} (\abs{\ppf} - 1) \abs{\Tree(\ppf)} = \frac{(\ppt)!}{(\pp)!} \sum_{\pttp \in \Ptt(\ppt)} (-1)^{\absbig{\pttp} - 1} (\absbig{\ppt} - 1)^{\absbig{\pttp} - 2} \frac{X(\pttp)}{(\pttp) !} \ .
    \tag*{\qedhere}
\end{equation*}

\end{proof}

\end{appendices}

\newpage

\section*{Notation table}
\begin{table}[h]
\makebox[1.05\textwidth][c]{
\begin{minipage}[t]{0.5\textwidth} 
\centering\small
\begin{tabular}{>{\raggedright\arraybackslash}p{0.4\textwidth}
                >{\raggedright\arraybackslash}p{0.65\textwidth}||}
$\pi \vdash d$ & a partition of $d\in\Zpos$ \\ 
\quad $\pi=(1^{n_1} ...d^{n_d})$ & \quad power notation (Not.\ref{nota:partition}) \\ 
\quad $|\pi|$ & \quad length of $\sim$ \\ 
& \\
$T=(V, E, \wt_E, \Lab)$ & LWBP-tree (Def.\ref{def:labeled_tree}) \\
\quad $V^{+}\ (V^-)$ & \quad black (white) vertices \\
\quad $W_E$ & \quad weight function on edges \\
\quad $W_V$ & \quad weight function on vertices \\
\quad $\Lab$ & \quad labeling \\
\quad $E(v)$ & \quad edges from a vertex \\
\quad $\sigma_v$ & \quad cyclic action (Def.\ref{def:cyclic_action}) \\
\quad $\Aut(T)$ & \quad automorphism group \\
 & \\
$\pp = (S,\lambda,\wt)$ & passport (Def.\ref{def:passport})\\ 
\quad $S$ & \quad index set \\
\quad $\lambda$ & \quad multiplicity function \\
\quad $\wt$ & \quad weight function \\
\quad $\abs{\pp}$ & \quad number of vertices \\
\quad $\Vert\pp\Vert$ & \quad total weights \\
\quad $\ppf$ & \quad full $\sim$ (Def.\ref{def:types_passport}) \\
\quad $\ppt$ & \quad trivial $\sim$ (Def.\ref{def:types_passport}) \\
\quad $\triv(\pp)$ & \quad trivialization (Def.\ref{def:trivialization})\\
\quad $(\pp)!$ & \quad factorial (Def.\ref{def:factorial}) \\
 & \\
$\pp(S',\lambda')$ & subpassport (Def.\ref{def:subpp}) \\
$\pttp = \{\pp_1^{n_1}...\pp_k^{n_k}\}$ & partition (Def.\ref{def:partition}) \\
\quad $\abs{\pttp}$ & \quad length of $\sim$ \\
\quad $\Ptt(\pp)$ & \quad set of $\sim$ \\
\quad $X(\pttp)$ & \quad $X$-function \\
\quad $(\pttp)!$ & \quad factorial \\
\end{tabular}
\end{minipage}
\hfill 
\begin{minipage}[t]{0.51\textwidth}  
\centering\small
\begin{tabular}{>{\raggedright\arraybackslash}p{0.35\textwidth}
                >{\raggedright\arraybackslash}p{0.65\textwidth}}
$\Tree(\pp)$ & trees of given passport \\
\quad $\TR(\pp)$ & \quad rooted $\sim$ (Def.\ref{def:rooted_tree}) \\
\quad $\Tree(\pp,d)$ & \quad $d$-fold symmetric $\sim$ \\ 
\quad $\pp / d$ & \quad divided passport (Def.\ref{def:passport_divide}) \\ 
\quad $g(s)$ & \quad g.c.d at an index (Eq.\eqref{eq:def_gs})\\
\quad $D$ & \quad factor set (Eq.\eqref{eq:def_D}) \\
\quad $G(d)$ & \quad coefficient (Eq.\eqref{eq:def_G(d)}) \\
& \\
$\mathcal{D}$ & branch datum \\
\quad $f:Y \to X$ & \quad branched cover \\
\quad $g_Y , g_X$ & \quad genus of surface \\
\quad $H$ & \quad dessin d'enfant (\S\ref{ssec:d_d}) \\ 
 & \\
(Def.\ref{defn:line_element}) & character line element \\
\quad $K_0, K_1$ & \quad extremal curvature \\
\quad $\ell$ & \quad length of $\sim$ \\
\quad $h(v)$ & \quad warped function \\
\quad $R$ & \quad ratio \\
$B_{\alpha}(K_0, K_1)$ & HCMU bigon (Def.\ref{defn:hcmu_bigon})\\
$S_{\alpha}(K_0, K_1)$ & HCMU football (Def.\ref{defn:hcmu_bigon}) \\
\quad $2\pi\alpha$ & \quad top angle \\
\quad $2\pi\beta = 2\pi R\alpha$ & \quad bottom angle \\
 & \\
$\Mhcmugn$ & space of HCMU surfaces \\
$(p,q,T;K_0,L)$ & data set rep. (Thm.\ref{thm:hcmu_data}) \\
\quad $\Sigma(T,K_0,\lambda)$ & \quad generic surface (Def.\ref{def:M_T}) \\
\quad $\MM(T)$ & \quad subset in moduli (Def.\ref{def:M_T}) \\
\quad $\MM(T)^*$ & \quad closure in $\GH$ (Def.\ref{def:M_T_star})\\
\quad $\varTheta$ & \quad parameter space (Def.\ref{def:M_T_star})\\
\end{tabular}
\end{minipage}
}
\end{table}

\end{document}